\documentclass[11pt,reqno,a4paper]{article}
\usepackage{comment}

\usepackage{euscript}
\usepackage{amsmath,amssymb,amsfonts,amsthm}
\usepackage{mathtools}
\usepackage{fullpage}
\usepackage{xcolor}
\usepackage{dsfont}
\usepackage[colorlinks=true, linkcolor=blue]{hyperref}
\usepackage{sectsty}

\newtheoremstyle{special}%
{}%
{}%
{}%
{}%
{\scshape}%
{.}%
{.5em}%
{}

\newtheorem{maintheorem}{Theorem}
\newtheorem{theorem}{Theorem}
\newtheorem{proposition}[theorem]{Proposition}
\newtheorem{lemma}[theorem]{Lemma}
\newtheorem{cor}[theorem]{Corollary}

\newtheorem{assumption}{Assumption}

\theoremstyle{special}
\newtheorem{remark}[theorem]{Remark}

\renewcommand{\epsilon}{\varepsilon}
\renewcommand{\L}{\mathcal{L}}

\allsectionsfont{\centering\mdseries\normalsize\scshape}

\def\N{\mathbb{N}}

\def\R{\mathbb{R}}
\def\B{\mathcal{B}}

\def\om{\omega}

\DeclareMathOperator{\essinf}{essinf}

\author{D. Dragi\v cevi\' c \footnote{Department of Mathematics, University of Rijeka, Rijeka Croatia. {\tt E-mail: ddragicevic@math.uniri.hr}.}\and 
	Y. Hafouta\footnote{Department of Mathematics, The University of Florida, USA. {\tt E-mail: yeor.hafuta@ufl.edu}.}}

\begin{document}

\title{Effective quenched linear response for random dynamical systems}
\maketitle

\begin{abstract}
We prove ``effective''  linear response for certain classes of non-uniformly expanding random dynamical systems which are not necessarily  composed in an i.i.d manner. In applications, the results are obtained for base maps with a sufficient amount of mixing.
 The fact that the rates are effective is then applied to obtain the differentiability of the variance in the CLT as a function of the parameter, as well as the annealed linear response. These two applications are beyond the reach of the linear response obtained in the general case, when all the random variables appearing in the bounds are only tempered.
We also provide several wide examples of one-dimensional maps satisfying our conditions, as well as some higher-dimensional examples.
\end{abstract}

\section{Introduction}

\subsection{Linear response for deterministic dynamics}
Let $M$ be a compact Riemannian manifold and $(T_\varepsilon)_{\varepsilon \in I}$ a family of sufficiently regular maps $T_{\varepsilon}\colon M\to M$, where $I$ is an interval in $\R$ such that $0\in I$. Here, we  view $T_\varepsilon$ as a ``sufficiently small'' perturbation of $T_0$. Suppose that for each $\varepsilon \in I$, $T_\varepsilon$ admits a unique physical measure $\mu_\varepsilon$. The problem of linear response is concerned with the regularity of the map $\varepsilon \to \mu_\varepsilon$ at $0$. More precisely, let $\mathcal O$ be a ``large'' class consisting of real-valued observables $\varphi \colon M\to \mathbb R$. We say that a family $(T_\varepsilon)_{\varepsilon \in I}$ exhibits:
\begin{itemize}
\item \emph{statistical stability} if the map $\varepsilon \to \int_M \varphi \, d\mu_\varepsilon$ is continuous at $0$ for each $\varphi \in \mathcal O$;
\item \emph{linear response} if the map $\varepsilon \to \int_M \varphi\, d\mu_\varepsilon$ is differentiable at $0$ for each $\varphi \in \mathcal O$.
\end{itemize}
We note that in some cases the measures $\mu_\varepsilon$ can be identified as elements of a certain Banach space $\B$, and in that case it makes sense to ask whether the map $\varepsilon \mapsto \mu_\varepsilon$ is continuous or differentiable in $0$ when viewed as a map from $I$ to $\mathcal B$.

We stress that the literature dealing with the linear response for deterministic dynamical systems (as introduced above) is vast.  More precisely,  linear response (or the lack of it) has been discussed for smooth expanding 
	systems~\cite{B2,Baladibook,S}, piecewise expanding maps of the interval \cite{B1,BS1}, unimodal maps~\cite{BS2, LS}, intermittent maps \cite{BahSau,BT, L, K}, smooth hyperbolic diffeomorphisms and flows~\cite{BL1, BL2, GL, Ruelle97}, large classes of partially hyperbolic systems~\cite{Dol}, systems with cusps~\cite{BG}, as well as discontinuous perturbations of hyperbolic systems~\cite{C}. We refer to~\cite{B2} for a detailed survey of the linear response theory for deterministic dynamical systems which has many interesing applications, for instance to the continuity and differentiability of the variance in the central limit theorem (CLT) for suitable observables (see for example~\cite{bomfim2016, GKLM}).

\subsection{Linear response for random dynamics}
In the context of random dynamical systems, let us assume that for each $\varepsilon \in I$, we have a cocycle of maps $(T_{\omega, \varepsilon})_{\omega \in \Omega}$, $T_{\omega, \varepsilon}\colon M\to M$,
 where $(\Omega, \mathcal F, \mathbb P)$ is a probability space together with an invertible and ergodic measure-preserving transformation $\sigma \colon \Omega \to \Omega$.
We suppose that for each $\varepsilon \in I$, the cocycle $(T_{\omega, \varepsilon})_{\omega \in \Omega}$ has a unique physical equivariant measure, which can be viewed as a (measurable) collection $(\mu_{\om, \varepsilon})$ of probability measures on $M$ with the property that 
\[
T_{\om,\varepsilon}^* \mu_{\om, \varepsilon}=\mu_{\sigma \om, \varepsilon}, \quad \text{for $\mathbb P$-a.e. $\om \in \Omega$.}
\]
Here, $T_{\om,\varepsilon}^* \mu_{\om, \varepsilon}$ denotes the push-forward of $\mu_{\om, \varepsilon}$ with respect to $T_{\om, \varepsilon}$. As in the deterministic setting, we are interested in the regularity of the map $\varepsilon \to \mu_{\om, \varepsilon}$.

However, in the random environment it makes sense to consider two concepts of the linear response. More precisely, we say that the parameterized family of cocycles $(T_{\omega, \varepsilon})_{\omega \in \Omega}$, $\varepsilon \in I$ exhibits:
\begin{itemize}
\item \emph{quenched linear response} if the map $\varepsilon \mapsto \int_M\varphi \, d\mu_{\om, \varepsilon}$ is differentiable at $0$ for $\mathbb P$-a.e. $\om \in \Omega$, where $\varphi \colon M\to \R$ belongs to a suitable class of observables;
\item \emph{annealed linear response} if the map $\varepsilon \mapsto \int_{\Omega \times M} \Phi \, d\mu_\varepsilon$ is differentiable at $0$, where $\mu_\varepsilon$ is the measure on $\Omega \times M$ given by 
\[
\mu_\varepsilon (A \times B)=\int_A \mu_{\om, \varepsilon}(B)\, d\mathbb P(\om) \quad \text{for $A\in \mathcal F$, $B\subset M$ Borel},
\]
and $\Phi \colon \Omega \times M \to \R$ belongs to a suitable class of observables.
\end{itemize}
For annealed linear response results (mostly dealing with the case when the maps $T_{\om, \varepsilon}$ are composed in an i.i.d fashion) which rely on techniques very similar to the ones for deterministic dynamics, we refer to~\cite{BRS,GG,GS,GL}. On the other hand,  the study of the quenched linear response was initiated by  Rugh and Sedro~\cite{RS} for random expanding dynamics, followed by the works by Dragi\v cevi\' c and Sedro~\cite{DS} and Crimmins and  Nakano~\cite{CN} for random (partially) hyperbolic dynamics. More recently, in~\cite{DGTS}, the authors established quenched linear response for a class of random intermittent maps.
We emphasize that all four papers deal with cases of random dynamics which exhibit uniform decay of correlations (with respect to the random parameter $\om \in \Omega$). 

On the other hand, Dragi\v cevi\' c, Giulietti and Sedro~\cite{DGS} established the quenched linear response for a class of random dynamics which exhibits nonuniform decay of correlations. 
 More precisely, they considered the case of cocycles which are expanding on average. Namely, 
 in~\cite{DGS} it is assumed that there exists a $\log$-integrable random variable $\underline{\gamma} \colon \Omega \to (0, \infty)$ such that $\gamma_{\om, \varepsilon} \ge \underline{\gamma}$ and 
\begin{equation}\label{minexp}
\int_\Omega \log \underline{\gamma}(\om) \, d\mathbb P(\om)>0,
\end{equation}
where $\gamma_{\om, \varepsilon}$ denotes the minimal expansion of $T_{\om, \varepsilon}$.  Note that~\eqref{minexp} allows for $\gamma_{\om, \varepsilon}<1$ on a set of positive measure. Thus, in sharp contrast to~\cite{RS}, it is not required that all maps $T_{\om, \varepsilon}$ are expanding or that there exists a uniform (in $\om$) lower bound for the minimal expansion. The main result of~\cite{DGS} yields that for each $\varepsilon$, there is a measurable  family $(h_{\omega, \varepsilon})_{\om \in \Omega}$ lying in the Sobolev space $W^{3, 1}$  such that:
\begin{itemize}
\item for each $\varepsilon \in I$, the family of measures $(\mu_{\om, \varepsilon})_{\om \in \Omega}$ given by $d\mu_{\om, \varepsilon}=h_{\om, \varepsilon}\, d\text{Vol}$ is equivariant for $(T_{\om, \varepsilon})_{\om \in \Omega}$;
\item there exists a measurable family $(\hat h_\om)_{\om \in \Omega}\subset W^{1, 1}$ with the property that for each sequence $(\varepsilon_k)_{k\in \N}$ such that $\lim\limits_{k\to \infty}\varepsilon_k=0$, there exist $a>0$ and a tempered random variable $K\colon \Omega \to (0, \infty)$\footnote{We recall that this means that $\lim\limits_{n\to \pm \infty} \frac 1 n \ln K(\sigma^n \om)=0$ for $\mathbb P$-a.e. $\om \in \Omega$} such that 
\begin{equation}\label{oldln}
\|h_{\om, \varepsilon_k}-h_{\om, 0}-\varepsilon_k \hat h_\om\|_{W^{1, 1}}\le K(\om)|\varepsilon_k|^{1+a},
\end{equation}
for $\mathbb P$-a-e. $\om  \in \Omega$.
\end{itemize}
It was also illustrated in~\cite[Appendix A]{DGS} that in this setup, it is possible that the annealed linear response fails even if the quenched linear response holds. 

\subsection{Contributions of the present paper}
The main objective of the present work is to obtain a quenched linear response result for a class random expanding dynamics which exhibits nonuniform decay of correlations, and where we are able to obtain a finer control on the speed of convergence from that in~\eqref{oldln}. More precisely, our setup includes a wide collection of examples where $K$ in~\eqref{oldln} belongs to $L^s(\Omega, \mathcal F, \mathbb P)$ for some $s>0$ (where we replace $W^{1, 1}$ by the space of continuous functions $C^0$). In fact, for a given $s>0$ we have general sufficient conditions that ensure that $K(\cdot)\in L^s(\Omega, \mathcal F, \mathbb P)$. 

We refer to the new version of~\eqref{oldln} as  effective quenched linear response.
In addition, we eliminate the necessity for the discretization of variable $\varepsilon$ in~\eqref{oldln}, as the techniques in this paper avoid the use of the multiplicative ergodic theorem (see Remark~\ref{comp} for details).

Our results have two major advantages when compared to~\cite{DGS}. Firstly, we show that our class of examples exhibits  both effective quenched and annealed linear response,  see Theorems \ref{LRThm} and \ref{annealed}. Secondly, we apply our quenched linear response result to the differentiability of the variance in the quenched CLT, see Theorem \ref{DiffThm}.
We emphasize that  both of these novelties represent first results of that kind that deal with random systems with a nonuniform decay of correlations. Indeed, as already noted, in the setup of~\cite{DGS} annealed linear response can fail. Furthermore, it is not clear  whether (in the full generality of~\cite{DGS}) there even exists a class of observables with the property that $(T_{\om, \varepsilon})_{\om \in \Omega}$ satisfies quenched central limit theorem~\cite{DHS, DS1} for each $\varepsilon$. Even if this is the case, we are unable to establish the desired differentiability of the variance when the linear response is controlled with only a tempered random variable as in~\eqref{oldln}.

Our results are close in spirit to the work of the second author on limit theorems for random dynamical systems exhibiting nonuniform decay of correlations~\cite{YHAdv, YH 23}, yielding explicit conditions on the observables satisfying those  provided that the base system $(\Omega, \mathcal F, \mathbb P)$ satisfies appropriate mixing assumptions. The role of the results in \cite{YHAdv, YH 23} is that, after appropriate modification, they allow us to replace the exponential convergence obtained by applying the multiplicative ergodic theorem for each one of the cocycles $(T_{\om,\varepsilon})_{\om\in\Omega}$ by (possibly) a moderate version which holds simultaneously for all   cocycles $(T_{\om,\varepsilon})_{\om\in\Omega}$, $\varepsilon\in I$.   This allows us to verify one of the eight conditions in our  abstract result about linear response (Theorems \ref{LRThm} and \ref{annealed}, see  also Section \ref{Sec1.3} below).

\subsection{Organization of the paper}\label{Sec1.3}
The paper is organized as follows. In Sections \ref{LR}, \ref{ALR} and \ref{CLT} we present cascades of abstract necessary conditions  for (effective) quenched linear response, annealed linear response   and for the  differentiability of the asymptotic variance in the quenched  CLT as a function of the parameter $\varepsilon$. In Section \ref{Expanding} we will apply the abstract results to some classes of random non-uniformly expanding random dynamical systems. We find each one of the  conditions of the abstract results interesting on it own and non trivial. Because of that our approach in Section \ref{Expanding}  is to provide sufficient conditions to each one of the conditions of the general theorems separately. Still, for readers' convenience  Section \ref{Expanding} starts with two concrete examples (see Section \ref{SummSec}). The first example is a wide class of one dimensional expanding maps (see Theorem \ref{1D EG}). The one dimensionality is only used to the control the maximal volume growth after iterating the random dynamical system (see Remark \ref{Rem Expand}), which we can also derive for certain higher dimensional maps described in Section \ref{SummSec}. Once this property holds we can consider the rather general classes of higher dimensional maps in \cite[Section 3.3]{YH 23}, and so we believe that other high dimensional examples can be given.

\section{Effective linear response type estimates for random dynamical systems: an abstract result}\label{LR}
Let $(\Omega, \mathcal F, \mathbb P)$ be a probability space and $\sigma \colon \Omega \to \Omega$ an $\mathbb P$-preserving measurable transformation. We will assume that $\mathbb P$ is ergodic.

We consider a triplet of Banach spaces: $(\mathcal B_w, \| \cdot \|_w)$, $(\mathcal B_s, \| \cdot \|_s)$ and $(\mathcal B_{ss}, \| \cdot \|_{ss})$. We assume that $\mathcal B_{ss}$ is embedded in $\mathcal B_s$, which is embedded in $\mathcal B_w$. In addition, we suppose that 
\[
\| \cdot \|_w \le \| \cdot \|_s \le \| \cdot \|_{ss}.
\]
Let $I\subset (-1, 1)$ be an open interval such that $0\in I$. We assume that for each $\epsilon \in I$, we have a cocycle of  linear operators $(\L_{\omega, \epsilon})_{\omega \in \Omega}$, where $\L_{\omega, \epsilon}$ is bounded on each of three spaces $\B_w$, $\B_s$ and $\B_{ss}$.  As usual, set
\[
\L_{\omega, \epsilon}^n:=\L_{\sigma^{n-1}\omega, \epsilon}\circ \ldots \circ \L_{\omega, \epsilon}.
\]
Assume that there exists a nonzero bounded functional $\psi \in \B_w^*$ with the property that 
\[
\L_{\omega, \epsilon}^*\psi=\psi, \quad \omega \in \Omega, \ \epsilon \in I.
\]
\begin{remark}\label{Rem1}
In our applications, $\L_{\omega, \epsilon}$ will be the transfer operator associated to a map  $T_{\omega, \epsilon}\colon M \to M$, where $M$ is a Riemmanian manifold. Moreover, the functional $\psi$ will have the form $\psi (\varphi)=\int_M\varphi \, dm$, where $m$ is the Lebesgue (volume) measure on $M$.
\end{remark}
We now  formulate our abstract  quenched linear response result.
\begin{maintheorem}\label{LRThm}
Assume that there exist $C_i\in L^{p_i}(\Omega, \mathcal F, \mathbb P)$ with $p_i>0$ for $i\in \{0, 1, 2, 3, 4\}$, $A, B\in L^{p_5}(\Omega, \mathcal F, \mathbb P)$ with $p_5>0$, $\beta >1$, $r>0$ such that $\beta -r>1$ and $\frac{p_ir}{3}>1$ for $i\in \{1, 2, 3, 4\}$, and $\Omega'\subset \Omega$ of full measure so that the following holds:
\begin{enumerate}
\item for $n\in \N$, $\epsilon \in I$ and $\omega \in \Omega'$,
\begin{equation}\label{weakcontr}
\|\L_{\sigma^{-n}\omega, \epsilon}^n \|_w \le C_0(\omega);
\end{equation}
\item  for $n\in \N$, $\epsilon \in I$ and $\omega \in \Omega'$,
\begin{equation}\label{spectral gap}
\|\L_{\omega, \epsilon}^n h\|_w \le C_1(\omega)n^{-\beta}\|h \|_s,
\end{equation}
for $h\in V_s$, where
\[
V_s:=\{h\in \B_s: \ \psi(h)=0\};
\]
\item for $\epsilon \in I$, $\omega \in \Omega'$ and $h\in \B_{s}$,
\begin{equation}\label{tn1}
\|(\L_{\omega, \epsilon}-\L_\omega)h\|_w \le C_2(\omega)|\epsilon| \|h\|_s,
\end{equation}
where $\L_\omega:=\L_{\omega, 0}$;
\item for $\omega \in \Omega'$, there exists a linear operator $\hat{\L}_\omega \colon \B_{ss}\to V_s$ such that 
\begin{equation}\label{deriv}
\left \|\frac{1}{\epsilon}(\L_{\omega, \epsilon}-\L_{\omega})h-\hat \L_\omega h\right \|_s \le C_3(\omega)|\epsilon| \|h\|_{ss}
\end{equation}
and
\begin{equation}\label{hatL}
\| \hat \L_\omega h\|_s \le C_3(\omega)\|h\|_{ss},
\end{equation}
for $\epsilon \in I\setminus \{0\}$ and $h\in \B_{ss}$;
\item for $\epsilon \in I$ there exists a measurable family $(h_{\omega, \epsilon})_{\omega \in \Omega}\subset \B_{ss}$ such that for $\omega \in \Omega'$ and $\epsilon \in I$,
\begin{equation}\label{top0}
\L_{\omega, \epsilon}h_{\omega, \epsilon}=h_{\sigma \omega, \epsilon} \quad \text{and} \quad \psi(h_{\omega, \epsilon})=1.
\end{equation}
Moreover,  for  $\omega \in \Omega'$ and $\epsilon \in I$,
\begin{equation}\label{top}
\|h_{\omega, \epsilon} \|_{ss} \le C_4(\omega);
\end{equation}
\item for $\omega \in \Omega'$, $n\in \N$ and $0\le j \le n$,
\begin{equation}\label{AB}
    \|\L_{\sigma^{-n}\omega}^j \|_s \le A(\sigma^{-n}\omega)B(\sigma^{j-n}\omega).
\end{equation}
\end{enumerate}
Let $s>0$ be given by \begin{equation}\label{s}\frac1s=\frac{1}{p_0}+\frac1{p_1}+\frac2{p_5}+\frac1{p_4}+\frac{1}{\min(p_2,p_3)}.\end{equation}
Then, for every $\delta>0$ there exists a random variable $U_1\in L^s(\Omega,\mathcal F,\mathbb P)$ and a full measure set $\Omega''\subset \Omega$ such that  for $\omega \in \Omega''$ and $\epsilon \in I\setminus \{0\}$ we have that 
\begin{equation}\label{qlr}
\left \| \frac{1}{\epsilon}(h_{\omega, \epsilon}-h_{\omega})-\hat h_\omega \right \|_w\leq U_1(\omega)|\epsilon|^{a},
\end{equation}
where $a=\frac{\beta-1-r}{\beta+1-r+1/s+\delta}$,  $h_\omega:=h_{\omega, 0}$ and 
\begin{equation}\label{hath}
\hat h_\omega:=\sum_{n=0}^\infty\L_{\sigma^{-n}\omega}^{n}\hat\L_{\sigma^{-(n+1)}\omega}h_{\sigma^{-(n+1)}\omega}.
\end{equation}
\end{maintheorem}

In the course of the proof of Theorem \ref{LRThm} we will repeatedly use the following simple result (see \cite[Lemma 13]{DH23}).
\begin{lemma}\label{B lemma}
Suppose that  $B\in L^q(\Omega,\mathcal F,\mathbb P)$ for some $q>0$.  Then, for every sequence of positive numbers $(a_n)_{n\in \N}$ such that $\sum_{n\geq 1}a_n^q<+\infty$, there is a random variable $R\in L^q(\Omega,\mathcal F,\mathbb P)$ such that
\begin{equation}\label{BR}
B(\sigma^{-n}\omega)\leq R(\omega)a_n^{-1},
\end{equation}
for $\mathbb P$-a.e. $\omega \in \Omega$ and every $n\in \N$. In particular, for every $\delta>0$ there exists $R_\delta\in L^q(\Omega,\mathcal F,\mathbb P)$ such that $B(\sigma^{-n}\om)\leq R_\delta(\om)n^{1/q+\delta}$ for $\mathbb P$-a.e. $\om \in \Omega$ and $n\in \N$.
\end{lemma}

\begin{proof}[Proof of Theorem~\ref{LRThm}]
We first show that the series defining $\hat h_\omega$ converges. Indeed, \eqref{spectral gap}, \eqref{hatL} and~\eqref{top}  give that 
\[
\begin{split}
\sum_{n=1}^\infty \|\L_{\sigma^{-n}\omega}^{n}\hat\L_{\sigma^{-(n+1)}\omega}h_{\sigma^{-(n+1)}\omega}\|_w 
&\le \sum_{n=1}^\infty C_1(\sigma^{-n}\omega)n^{-\beta} C_3(\sigma^{-(n+1)}\omega)C_4(\sigma^{-(n+1)}\omega),
\end{split}
\]
for $\omega \in \Omega'$. By Lemma~\ref{B lemma}, for $i\in \{1, 2, 3, 4\}$ there exist $C_i'\in L^{p_i}(\Omega, \mathcal F, \mathbb P)$ such that 
\begin{equation}\label{ii}
C_i(\sigma^{-n}\omega)\le C_i'(\omega)n^{\frac{r}{3}},
\end{equation}
for $\mathbb P$-a.e. $\omega \in \Omega$ and $n\in \N$. Without any loss of generality, we may (and do) suppose that~\eqref{ii} holds for each $\omega \in \Omega'$.  Hence, 
\[
\|\hat h_\omega \|_w \le C_3(\sigma^{-1}\omega)C_4(\sigma^{-1}\omega)+
C_1'(\omega)C_3'(\omega)C_4'(\omega)\sum_{n=1}^\infty n^{-\beta} n^{\frac r 3}(n+1)^{\frac{2r}{3}}<+\infty,
\]
for $\omega \in \Omega'$. 
Next, observe that 
\begin{equation}\label{1o2}
\tilde h_{\omega, \epsilon}-\L_{\sigma^{-1}\omega, \epsilon}\tilde h_{\sigma^{-1}\omega, \epsilon}=\tilde \L_{\sigma^{-1}\omega, \epsilon}h_{\sigma^{-1}\omega},
\end{equation}
for  $\omega \in \Omega'$ and $\epsilon \in I$, where
\begin{equation}\label{hathL}
\tilde h_{\omega, \epsilon}:=h_{\omega, \epsilon}-h_\omega \quad \text{and} \quad \tilde \L_{\omega, \epsilon}:=\L_{\omega, \epsilon}-\L_\omega.
\end{equation}
By iterating~\eqref{1o2}, we obtain that for $\omega \in \Omega'$, $\epsilon \in I$ and  $n\in \N$,
\begin{equation}\label{gI}
\tilde h_{\omega, \epsilon}=\L_{\sigma^{-n}\omega, \epsilon}^n \tilde h_{\sigma^{-n}\omega, \epsilon}+\sum_{j=0}^{n-1}  \L_{\sigma^{-j}\omega, \epsilon}^j \tilde \L_{\sigma^{-(j+1)}\omega, \epsilon}h_{\sigma^{-(j+1)}\omega}.
\end{equation}
Note that~\eqref{spectral gap}, \eqref{top} and~\eqref{ii} imply that 
\[
\|\L_{\sigma^{-n}\omega, \epsilon}^n \tilde h_{\sigma^{-n}\omega, \epsilon}\|_w \le 2C_1(\sigma^{-n}\omega)C_4(\sigma^{-n}\omega)n^{-\beta} \le 2C_1'(\omega)C_4'(\omega)n^{-\beta+\frac{2r}{3}},
\]
for  $\omega \in \Omega'$, $\epsilon \in I$ and  $n\in \N$. Letting $n\to \infty$ in~\eqref{gI}, we conclude that 
\begin{equation}\label{12a}
\tilde h_{\omega, \epsilon}=\sum_{n=0}^\infty \L_{\sigma^{-n}\omega, \epsilon}^n\tilde \L_{\sigma^{-(n+1)}\omega, \epsilon}h_{\sigma^{-(n+1)}\omega}, \quad \text{for  $\omega \in \Omega'$ and $\epsilon \in I$.}
\end{equation}
Hence, for $\omega \in \Omega'$ and $\epsilon \in I\setminus \{0\}$ we have that 
\[
\begin{split}
\frac{1}{\epsilon}(h_{\omega, \epsilon}-h_{\omega}) -\hat h_\omega &=\frac{1}{\epsilon} \tilde h_{\omega, \epsilon}-\hat h_\omega \\
&=\frac{1}{\epsilon}\sum_{n=0}^\infty \L_{\sigma^{-n}\omega, \epsilon}^n\tilde \L_{\sigma^{-(n+1)}\omega, \epsilon}h_{\sigma^{-(n+1)}\omega} \\
&\phantom{=}-\sum_{n=0}^\infty\L_{\sigma^{-n}\omega}^{n}\hat\L_{\sigma^{-(n+1)}\omega}h_{\sigma^{-(n+1)}\omega}\\
&=\sum_{n=0}^\infty \L_{\sigma^{-n}\omega, \epsilon}^n \left (\frac{1}{\epsilon}\tilde \L_{\sigma^{-(n+1)}\omega, \epsilon}-\hat \L_{\sigma^{-(n+1)}\omega}\right )h_{\sigma^{-(n+1)}\omega} \\
&\phantom{=}+\sum_{n=0}^\infty (\L_{\sigma^{-n}\omega, \epsilon}^n-\L_{\sigma^{-n}\omega}^n)\hat \L_{\sigma^{-(n+1)}\omega}h_{\sigma^{-(n+1)}\omega}\\
&=:(I)_{\omega, \epsilon}+(II)_{\omega, \epsilon}.
\end{split}
\]
Note that it follows from~\eqref{spectral gap}, \eqref{deriv}, \eqref{top} and~\eqref{ii} that
\begin{equation}\label{I-1}
\begin{split}
\| (I)_{\omega, \epsilon} \|_w
&=\left \|\sum_{n=0}^\infty \L_{\sigma^{-n}\omega, \epsilon}^n \left (\frac{1}{\epsilon}\tilde \L_{\sigma^{-(n+1)}\omega, \epsilon}-\hat \L_{\sigma^{-(n+1)}\omega}\right )h_{\sigma^{-(n+1)}\omega}\right \|_w \\
&\le C_3(\sigma^{-1}\omega)C_4(\sigma^{-1}\omega)|\epsilon|+|\epsilon|\sum_{n=1}^\infty C_1(\sigma^{-n}\omega)n^{-\beta }C_3(\sigma^{-(n+1)}\omega)C_4(\sigma^{-(n+1)}\omega) \\
&\le C_3(\sigma^{-1}\omega)C_4(\sigma^{-1}\omega)|\epsilon|+
|\epsilon|C_1'(\omega)C_3'(\omega)C_4'(\omega)\sum_{n=1}^\infty n^{-\beta+\frac r 3}(n+1)^{\frac{2r}{3}},
\end{split}
\end{equation}
for $\omega \in \Omega'$ and $\epsilon \in I\setminus \{0\}$. We now analyze $(II)_{\omega, \epsilon}$. Note that for each $n\in \N$, we have (using~\eqref{spectral gap}, \eqref{deriv}, \eqref{top} and~\eqref{ii}) that 
\[
\begin{split}
\|(\L_{\sigma^{-n}\omega, \epsilon}^n-\L_{\sigma^{-n}\omega}^n)\hat \L_{\sigma^{-(n+1)}\omega}h_{\sigma^{-(n+1)}\omega}\|_w &\le 2C_1(\sigma^{-n}\omega)C_3(\sigma^{-(n+1)}\omega)C_4(\sigma^{-(n+1)}\omega)n^{-\beta} \\
&\le 2C_1'(\omega)C_3'(\omega)C_4'(\omega)n^{-\beta+\frac r 3}(n+1)^{\frac{2r}{3}},
\end{split}
\]
for $n\in \N$ and $\omega \in \Omega'$. Let $q>0$ be given by $\frac1q=\frac{1}{p_1}+\frac{1}{p_3}+\frac1{p_4}$ and let $K_N=\sum_{n>N}n^{-\beta+r/3}(n+2)^{ 2r/3}\asymp N^{-(\beta-r-1)}$ (recall that $\beta>r+1$).
We conclude that there exists a random variable $D\colon \Omega \to (0, \infty)$, $D\in L^{q}(\Omega,\mathcal F,\mathbb P)$ such that
\begin{equation}\label{Z1I}
\sum_{n=N+1}^\infty \|(\L_{\sigma^{-n}\omega, \epsilon}^n-\L_{\sigma^{-n}\omega}^n)\hat \L_{\sigma^{-(n+1)}\omega}h_{\sigma^{-(n+1)}\omega}\|_w \le  D(\omega)K_N,
\end{equation}
for  $\omega \in \Omega'$, $N\in \N$ and $\epsilon \in I$. Next, note that 
\begin{equation}\label{fgew}
\L_{\sigma^{-n}\omega}^n-\L_{\sigma^{-n}\omega, \epsilon}^n=\sum_{j=1}^n\L_{\sigma^{-(n-j)}\omega, \epsilon}^{n-j}(\L_{\sigma^{-n+j-1}\omega}-\L_{\sigma^{-n+j-1}\omega, \epsilon})\L^{j-1}_{\sigma^{-n}\omega},
\end{equation}
and therefore (using~\eqref{weakcontr}, \eqref{tn1}, \eqref{top} and~\eqref{AB})
\[
\begin{split}
&\sum_{n=1}^{N}\|(\L_{\sigma^{-n}\omega, \epsilon}^n-\L_{\sigma^{-n}\omega}^n)\hat \L_{\sigma^{-(n+1)}\omega}h_{\sigma^{-(n+1)}\omega}\|_w \\
&\le \sum_{n=1}^{N}\sum_{j=1}^n
\|\L_{\sigma^{-(n-j)}\omega, \epsilon}^{n-j}(\L_{\sigma^{-n+j-1}\omega}-\L_{\sigma^{-n+j-1}\omega, \epsilon})\L^{j-1}_{\sigma^{-n}\omega}\hat \L_{\sigma^{-(n+1)}\omega}h_{\sigma^{-(n+1)}\omega} \|_w  \\
&\le \sum_{n=1}^{N}\sum_{j=1}^nC_0(\sigma^{n-j}\omega)
\|(\L_{\sigma^{-n+j-1}\omega}-\L_{\sigma^{-n+j-1}\omega, \epsilon})\L^{j-1}_{\sigma^{-n}\omega}\hat \L_{\sigma^{-(n+1)}\omega}h_{\sigma^{-(n+1)}\omega} \|_w \\
&\le |\epsilon|\sum_{n=1}^N \sum_{j=1}^n C_0(\sigma^{n-j}\omega) C_2(\sigma^{-n+j-1}\omega) \|\L^{j-1}_{\sigma^{-n}\omega}\hat \L_{\sigma^{-(n+1)}\omega}h_{\sigma^{-(n+1)}\omega} \|_s \\
&\le  |\epsilon|\sum_{n=1}^N \sum_{j=1}^n C_0(\sigma^{n-j}\omega)C_2(\sigma^{-n+j-1}\omega)\|\L_{\sigma^{-n}\omega}^{j-1}\|_s \cdot  \|\hat \L_{\sigma^{-(n+1)}\omega}h_{\sigma^{-(n+1)}\omega}\|_{s} \\
&\le |\epsilon|\sum_{n=1}^N \sum_{j=1}^n C_0(\sigma^{n-j}\omega)C_2(\sigma^{-n+j-1}\omega)\|\L_{\sigma^{-n}\omega}^{j-1}\|_s \cdot C_{3}(\sigma^{-(n+1)}\omega)C_4(\sigma^{-(n+1)}\omega)\\
&\le |\epsilon|\sum_{n=1}^N \sum_{j=1}^n C_0(\sigma^{n-j}\omega)C_2(\sigma^{-n+j-1}\omega)A(\sigma^{-n}\omega)B(\sigma^{-n+j-1}\omega)C_{3}(\sigma^{-(n+1)}\omega)C_4(\sigma^{-(n+1)}\omega),
\end{split}
\]
for $\epsilon \in I$ and $\omega \in \Omega'$. By using Lemma~\ref{B lemma} (see~\eqref{s}), one can easily show that for every $\delta>0$ there exists a random variable $D'\in L^s(\Omega,\mathcal F,\mathbb P)$ such that for $\mathbb P$-a.e. $\omega \in \Omega$,
$$
C_0(\sigma^{n-j}\omega)C_2(\sigma^{-n+j-1}\omega)A(\sigma^{-n}\omega)B(\sigma^{-n+j-1}\omega)C_{3}(\sigma^{-(n+1)}\omega)C_4(\sigma^{-(n+1)}\omega)
$$
$$
\leq n^{\frac 1s+\delta}D'(\omega).
$$
We again assume without loss of generality that the above estimate  holds for each $\omega \in \Omega'$. Thus, setting $K'_N=N^{2+1/s+\delta}$ we have 
\begin{equation}\label{Z2I}
\sum_{n=1}^{N}\|(\L_{\sigma^{-n}\omega, \epsilon}^n-\L_{\sigma^{-n}\omega}^n)\hat \L_{\sigma^{-(n+1)}\omega}h_{\sigma^{-(n+1)}\omega}\|_w  \le D'(\omega)K_N'|\epsilon|, 
\end{equation}
for $\omega \in \Omega'$, $\epsilon \in I$ and $N\in \N$. Combining \eqref{Z1I} and \eqref{Z2I}  we conclude that for  $\omega \in \Omega'$, $N\in \N$ and $\epsilon \in I$ we have 
$$
\| (II)_{\omega, \epsilon} \|_w\leq CD(\omega)N^{-(\beta-r-1)}+D'(\omega)|\varepsilon| N^{2+1/s+\delta},
$$
where $C>0$ is a constant. Let $N=N_\varepsilon$ be given by $N=[|\varepsilon|^{-\zeta}]$, $\zeta=\frac{1}{\beta+1+1/s+\delta-r}$ (so that $N^{-(\beta-r-1)}\asymp |\varepsilon|N^{2+1/s+\delta}$.)
 Then with $D'':=D+D'\in L^{s}(\Omega,\mathcal F,\mathbb P)$ we have 
\begin{equation}\label{II-1}
\| (II)_{\omega, \epsilon} \|_w\leq C'D''(\omega)N^{-(\beta-r-1)}\leq C'D''(\omega)|\epsilon|^{\frac{\beta-1-r}{\beta+1+1/s+\delta-r}}
\end{equation}
where $C'>0$ is a constant.
Combining~\eqref{I-1} and~\eqref{II-1} we get that with  $$
U(\omega):=C_3(\sigma^{-1}\omega)C_4(\sigma^{-1}\omega)+C_1'(\omega)C_3'(\omega)C_4'(\omega)
$$ 
we have
$$
\left\|\frac{1}{\epsilon}(h_{\omega, \epsilon}-h_{\omega}) -\hat h_\omega\right\|_w\leq  C'U(\omega)|\epsilon|+C'D''(\omega)|\epsilon|^{\frac{\beta-1-r}{\beta+1+1/s+\delta-r}},
$$
for $\omega \in \Omega'$ and $\epsilon \in I\setminus \{0\}$.
Note that  $U\in L^s(\Omega,\mathcal F, \mathbb P)$. This immediately implies that~\eqref{qlr} holds with 
\[
U_1(\omega):=C'U(\omega)+C''D''(\omega)\in L^s(\Omega, \mathcal F, \mathbb P).
\]
\end{proof}

$$
$$

As a byproduct of Theorem~\ref{LRThm}, we can formulate the following statistical stability result.
\begin{proposition}\label{Stab Prop}
Assume that there exist $C_i\in L^{p_i}(\Omega, \mathcal F, \mathbb P)$ with $p_i>0$ for $i\in \{ 1, 2,  4\}$,  $\beta >1$, $r>0$ such that $\beta -r>1$ and $\frac{p_ir}{3}>1$ for $i\in \{1, 2, 4\}$, and $\Omega'\subset \Omega$ of full measure so that the following holds:
\begin{enumerate}
\item for $n\in \N$, $\epsilon \in I$ and $\omega \in \Omega'$, \eqref{spectral gap} holds;
\item for $\epsilon \in I$, $\omega \in \Omega'$ and $h\in \B_{s}$, 
\begin{equation}\label{tn11}
\|(\L_{\omega, \epsilon}-\L_\omega)h\|_s \le C_2(\omega)|\epsilon| \|h\|_{ss};
\end{equation}
\item for $\epsilon \in I$ there exists a measurable family $(h_{\omega, \epsilon})_{\omega \in \Omega}\subset \B_{ss}$ such that~\eqref{top0} and~\eqref{top} holds for $\epsilon \in I$ and $\omega \in \Omega'$.
\end{enumerate}
Let $q>0$ be given by $\frac 1 q=\frac{1}{p_1}+\frac{1}{p_2}+\frac{1}{p_3}$. Then, there exists $\tilde U\in L^q(\Omega, \mathcal F, \mathbb P)$ and a full measure set $\Omega''\subset \Omega$ such that for $\epsilon \in I$ and $\omega \in \Omega''$,
\[
\|h_{\omega, \epsilon}-h_\omega \|_w \le \tilde U(\omega)| \epsilon|,
\]
where $h_\omega=h_{\omega,0}$.
\end{proposition}

\begin{proof}
By arguing as in the proof of Theorem~\ref{LRThm}, we have that~\eqref{12a} holds. Using~\eqref{spectral gap}, \eqref{tn11}, \eqref{top} and~\eqref{ii} we obtain that for $\epsilon \in I$ and $\omega \in \Omega'$,
\[
\begin{split}
  \| \tilde h_{\omega, \epsilon}\|_w  & \le  \|\tilde \L_{\sigma^{-1}\omega, \epsilon}h_{\sigma^{-1}\omega}\|_w+\sum_{n=1}^\infty \|\L_{\sigma^{-n}\omega, \epsilon}^n\tilde \L_{\sigma^{-(n+1)}\omega, \epsilon}h_{\sigma^{-(n+1)}\omega}\|_w \\
&\le C_2(\sigma^{-1}\omega)C_4(\sigma^{-1}\omega)|\epsilon|+|\epsilon|\sum_{n=1}^\infty C_1(\sigma^{-n}\omega)n^{-\beta}C_2(\sigma^{-(n+1)}\omega)C_4(\sigma^{-(n+1)}\omega) \\
&\le C_2(\sigma^{-1}\omega)C_4(\sigma^{-1}\omega)|\epsilon|+|\epsilon|C_1'(\omega)C_2'(\omega)C_4'(\omega)\sum_{n=1}^\infty n^{-\beta}n^{\frac r 3}(n+1)^{\frac{2r}{3}},
\end{split}
\]
which readily implies the desired conclusion.

\end{proof}
\begin{remark}\label{comp}
We would like to compare Theorem~\ref{LRThm} with the abstract quenched linear response given in~\cite[Theorem 11]{DGS}.
The major difference is that the assumptions of Theorem~\ref{LRThm} yield that $U_1$ in~\eqref{qlr} belongs to $L^p(\Omega, \mathcal F, \mathbb P)$ for some $p>0$. On the other hand, the conclusion of~\cite[Theorem 11]{DGS} gives~\eqref{qlr} with $U_1$ being only a tempered random variable.
The stronger conclusion we obtain will be essential in our applications of Theorem~\ref{LRThm} to the differentiability of the variance in quenched CLT given in Section~\ref{CLT}.

Furthermore, we note that in~\eqref{spectral gap} we require that $(\L_{\omega, \epsilon})_{\omega \in \Omega}$ exhibits only a polynomial decay of correlations for each $\epsilon \in I$. Despite this,  all of our examples will deal with cocycles which exhibit exponential decay of correlations. In other words, by applying the appropriate version of the multiplicative ergodic theorem (MET), one can show (for examples outlined in the following section) that $n^{-\beta}$ in~\eqref{spectral gap} can be replaced by $e^{-\lambda n}$ with $\lambda>0$. However, doing so causes two major complications:
\begin{itemize}
\item we lose integrability of $C_1$ (and obtain modified~\eqref{spectral gap} with $C_1$ being only tempered);
\item we can not ensure (by applying MET for each cocycle) that the full-measure set on which modified~\eqref{spectral gap} holds is independent on $\epsilon$, and that the same applies for both $C_1$ and $\lambda$.
\end{itemize}
Both of these problematic points cause obstructions to the proof (and conclusion) of Theorem~\ref{LRThm}. We verify~\eqref{spectral gap}  by using techniques developed in~\cite{YH 23}. These  rely on cone-contraction arguments combined with appropriate mixing assumptions.

Finally, we note that in~\cite[Theorem 11]{DGS} the variable $\epsilon$ is discretized, i.e. replaced with a sequence $(\epsilon_k)_{k\in \N}$ such that $\lim\limits_{k\to \infty}\epsilon_k=0=:\epsilon_0$. The reason for this is that the existence of a family $(h_{\omega, \epsilon})_{\omega \in \Omega}$ and the corresponding version of~\eqref{top} are verified by relying on MET. In that case, it is challenging to show that~\eqref{top} holds on a set of full-measure which does not depend on $\epsilon$. In the present paper, we do not rely on the MET and thus we do not have such concerns.
\end{remark}

\begin{remark}
We note that the result similar to
Theorem~\ref{LRThm} \emph{for deterministic dynamics} is formulated in~\cite[Theorem 2.3]{GP}. 
\end{remark}

\begin{remark}
In our applications, $\B_w$ will be a space $C^0(M)$ of continuous functions on a compact Riemannian manifold $M$ equipped with the supremum norm. In this context,  it is easy to conclude from~\eqref{qlr} that 
\[
\left \| \frac{1}{\epsilon}(H_\epsilon-H_0)-\hat H \right \|_{L^s(\Omega \times M)} \le \|U_1\|_{L^s(\Omega, \mathcal F, \mathbb P)}|\epsilon|^a,
\]
where $H_\epsilon (\omega, \cdot)=h_{\omega, \epsilon}$ and $\hat H(\omega \cdot)=\hat h_\omega$.
\end{remark}


\section{Annealed linear response}\label{ALR}
We will now formulate an annealed linear response result. 
In the sequel, we will assume that $\B_w$ (and consequently also $\B_s$ and $\B_{ss}$) consist of real-valued measurable functions defined on a space $M$ equipped with a probability measure $m$. Moreover, we require that  
$\B_w\subset L^1(m)$ and 
\[
\|h\|_{L^1(m)} \le \|h\|_w, \quad h\in \B_w.
\]
Finally,  we suppose that $\psi$ is given by $\psi(\varphi)=\int_M \varphi \, dm$.
\begin{maintheorem}\label{annealed}
Assume that there exist $C_i\in L^{p_i}(\Omega, \mathcal F, \mathbb P)$ with $p_i>0$ for $i\in \{0, 1, 2, 3, 4\}$, $A, B\in L^{p_5}(\Omega, \mathcal F, \mathbb P)$ with $p_5>0$, $\beta >1$, and for each $\epsilon \in I$, a full-measure set $\Omega_\epsilon \subset \Omega$ such that the following holds:
\begin{enumerate}
\item for $n\in \N$, $\epsilon \in I$ and $\omega \in \Omega_\epsilon$, \eqref{weakcontr} holds;
\item for  $n\in \N$, $\epsilon \in I$, $\omega \in \Omega_\epsilon$ and $h\in V_s$, \eqref{spectral gap} holds;
\item for  $\epsilon \in I$, $\omega \in \Omega_\epsilon$ and $h\in \B_s$, \eqref{tn1} holds;
\item for $\mathbb P$-a.e. $\omega \in \Omega$ there exists a linear operator $\hat{\L}_\omega \colon \B_{ss}\to V_s$ such that~\eqref{deriv} holds for $\epsilon \in I\setminus \{0\}$, $\omega \in \Omega_\epsilon$ and $h\in \B_{ss}$. Moreover, 
\eqref{hatL} holds for $\mathbb P$-a.e. $\omega \in \Omega$ and $h\in \B_s$;
\item for $\epsilon \in I$ there exists a measurable family $(h_{\omega, \epsilon})_{\omega \in \Omega}\subset \B_{ss}$ with $h_{\omega, \epsilon}\ge 0$ such that \eqref{top0} and~\eqref{top} hold for each $\epsilon \in I$ and $\omega \in \Omega_\epsilon$;
\item for $\mathbb P$-a.e. $\omega \in \Omega$, $n\in \N$ and $0\le j \le n$, \eqref{AB} holds.
\end{enumerate}
Let $\Phi\colon \Omega \times M \to \R$ be a measurable map such that $\Phi (\omega, \cdot)\in L^\infty (m)$ and 
\[
G(\omega):=\| \Phi(\omega, \cdot)\|_{L^\infty(m)}\in L^{p_6}(\Omega, \mathcal F, \mathbb P),
\]
for some $p_6>0$ such that 
\begin{equation}\label{coeff}
   \frac{1}{p_6}+ \frac{1}{p_1}+\frac{1}{p_3}+\frac{1}{p_4}\le 1, 
\end{equation}
and
\begin{equation}\label{coeff2}
\frac{1}{p_6}+\frac{1}{p_0}+\frac{1}{p_2}+\frac{2}{p_5}+\frac{1}{p_3}+\frac{1}{p_4}\le 1.
\end{equation}
Moreover, let $\mu_\epsilon$ be a measure on $\Omega \times M$ given by 
\[
\int_{\Omega \times M}\Phi \, d\mu_\epsilon=\int_{\Omega}\int_M \Phi(\omega, \cdot)h_{\omega, \epsilon}\, dm\, d\mathbb P(\omega).
\] 
Then, the map
\[
\epsilon \mapsto \int_{\Omega \times M}\Phi \, d\mu_\epsilon
\]
is differentiable in $\epsilon=0$.
\end{maintheorem}

\begin{proof}
Let $\hat h_\omega$ be given by~\eqref{hath}. Observe that~\eqref{spectral gap}, \eqref{hatL} and~\eqref{top}  imply that
\[
\|\hat h_\omega \|_w \le \sum_{n=1}^\infty n^{-\beta} C_1(\sigma^{-n}\omega)C_3(\sigma^{-(n+1)} \omega)C_4(\sigma^{-(n+1)}\om)+C_3(\sigma^{-1}\om)C_4(\sigma^{-1}\om),
\]
for  $\mathbb P$-a.e. $\omega \in \Omega$. By~\eqref{coeff} and the H\"{o}lder  inequality (together with the $\sigma$-invariance of $\mathbb P$), we have that 
\[
\| \|\hat h_\omega \|_w \|_{L^1}\le \|C_3\|_{L^{p_3}} \|C_4\|_{L^{p_4}} \left ( \|C_1\|_{L^{p_1}} \sum_{n=1}^\infty n^{-\beta}+ 1 \right )<+\infty.
\]
This in particular establishes that $\hat h_\omega$ is well-defined for $\mathbb P$-a-e. $\omega \in \Omega$. We have (see~\eqref{gI}) that for $\epsilon \in I$ and $n\ge 1$,
\[
\begin{split}
\int_{\Omega \times M}\Phi \, d\mu_{\epsilon}-\int_{\Omega\times M}\Phi \, d\mu &=\int_\Omega \int_M \Phi(\omega, \cdot)h_{\omega, \epsilon}\, dm\, d\mathbb P(\omega)-\int_\Omega \int_M \Phi(\omega, \cdot)h_{\omega}\, dm\, d\mathbb P(\omega)\\
&=\int_\Omega \int_M\Phi(\omega, \cdot)\tilde h_{\omega, \epsilon}\, dm\, d\mathbb P(\omega) \\
&=\int_\Omega \int_M\Phi(\omega, \cdot)\L_{\sigma^{-n}\omega, \epsilon}^n\tilde h_{\sigma^{-n}\omega, \epsilon}\, dm\, d\mathbb P(\omega) \\
&\phantom{=}+\sum_{j=0}^{n-1}\int_\Omega \int_M \Phi(\omega, \cdot)\L_{\sigma^{-j}\omega, \epsilon}^j \tilde \L_{\sigma^{-(j+1)}\omega, \epsilon}h_{\sigma^{-(j+1)}\omega}\, dm\, d\mathbb P(\omega),
\end{split}
\]
where $\tilde h_{\omega, \epsilon}$ and $\tilde \L_{\om, \epsilon}$ are given by~\eqref{hathL}, $h_\omega=h_{\omega, 0}$ and $\mu=\mu_0$.
Now note that by~\eqref{spectral gap}, \eqref{top}, \eqref{coeff} and the H\"{o}lder inequality,
\[
\begin{split}
\left |\int_\Omega \int_M\Phi(\omega, \cdot)\L_{\sigma^{-n}\omega, \epsilon}^n\tilde h_{\sigma^{-n}\omega, \epsilon}\, dm\, d\mathbb P(\omega)\right | &\le \int_\Omega \left |\int_M\Phi(\omega, \cdot)\L_{\sigma^{-n}\omega, \epsilon}^n\tilde h_{\sigma^{-n}\omega, \epsilon}\, dm\right | \, d\mathbb P(\omega) \\
&\le \int_\Omega G(\omega)\|\L_{\sigma^{-n}\omega, \epsilon}^n\tilde h_{\sigma^{-n}\omega, \epsilon}\|_w \, d\mathbb P(\omega) \\
&\le 2n^{-\beta}\int_\Omega G(\omega)C_1(\sigma^{-n}\omega)C_4(\sigma^{-n}\omega)\, d\mathbb P(\omega) \\
&\le 2n^{-\beta}\|G\|_{L^{p_6}}  \|C_1\|_{L^{p_1}} \|C_4\|_{L^{p_4}},
\end{split}
\]
for $\epsilon \in I$ and $n\ge 1$.
Hence, we obtain that 
\[
\lim_{n\to \infty}\left |\int_\Omega \int_M\Phi(\omega, \cdot)\L_{\sigma^{-n}\omega, \epsilon}^n\tilde h_{\sigma^{-n}\omega, \epsilon}\, dm\, d\mathbb P(\omega)\right |=0,
\]
and consequently, 
\[
\int_{\Omega \times M}\Phi \, d\mu_{\epsilon}-\int_{\Omega\times M}\Phi \, d\mu=\sum_{n=0}^{\infty}\int_\Omega \int_M \Phi(\omega, \cdot)\L_{\sigma^{-n}\omega, \epsilon}^n \tilde \L_{\sigma^{-(n+1)}\omega, \epsilon}h_{\sigma^{-(n+1)}\omega}\, dm\, d\mathbb P(\omega).
\]
Therefore, for each $\epsilon \in I\setminus \{0\}$, we have that 
\[
\begin{split}
  &  \frac{1}{\epsilon} \left (\int_{\Omega \times M}\Phi \, d\mu_{\epsilon}-\int_{\Omega\times M}\Phi \, d\mu \right)-\int_{\Omega}\int_M \Phi(\omega, \cdot)\hat h_\omega \, dm\, d\mathbb P(\omega) \\
  &=\sum_{n=0}^\infty \int_\Omega \int_M \Phi(\omega, \cdot) \L_{\sigma^{-n}\omega, \epsilon}^n \left (\frac{1}{\epsilon}\tilde \L_{\sigma^{-(n+1)}\omega, \epsilon}-\hat \L_{\sigma^{-(n+1)}\omega}\right )h_{\sigma^{-(n+1)}\omega}\, dm\, d\mathbb P(\omega) \\
  &\phantom{=}+\sum_{n=0}^\infty \int_{\Omega}\int_M \Phi(\omega, \cdot)(\L_{\sigma^{-n}\omega, \epsilon}^n-\L_{\sigma^{-n}\omega}^n)\hat \L_{\sigma^{-(n+1)}\omega}h_{\sigma^{-(n+1)}\omega}\, dm\, d\mathbb P(\omega) \\
  &=:(I)_\epsilon+(II)_\epsilon.
\end{split}
\]
It follows from~\eqref{spectral gap}, \eqref{deriv}, \eqref{top}, \eqref{coeff} and the H\"{o}lder
inequality  that 
\[
\begin{split}
|(I)_\epsilon| &\le |\epsilon| \sum_{n=1}^\infty n^{-\beta} \int_\Omega G(\omega)C_1(\sigma^{-n}\omega)C_3(\sigma^{-(n+1)}\omega)C_4(\sigma^{-(n+1)}\omega)\, d\mathbb P(\omega) \\
&\phantom{\le}+ |\epsilon| \int_\Omega G(\omega)C_3(\sigma^{-1}\om )C_4(\sigma^{-1} \om)\, d\mathbb P(\om) \\
&\le |\epsilon|  \cdot \|G\|_{L^{p_6}} \|C_3\|_{L^{p_3}} \|C_4\|_{L^{p_4}} \left ( \|C_1\|_{L^{p_1}}\sum_{n=1}^\infty n^{-\beta}+1\right ).
\end{split}
\]
Thus,
\[
\lim_{\epsilon \to 0}|(I)_\epsilon|=0.
\]
On the other hand, \eqref{spectral gap},  \eqref{deriv}, \eqref{top} and~\eqref{coeff} yield that 
\[
\begin{split}
&\left |\int_\Omega \int_M\Phi(\omega, \cdot)(\L_{\sigma^{-n}\omega, \epsilon}^n-\L_{\sigma^{-n}\omega}^n)\hat \L_{\sigma^{-(n+1)}\omega}h_{\sigma^{-(n+1)}\omega}\, dm\, d\mathbb P(\omega) \right | \\
&\le 2n^{-\beta}\|G\|_{L^{p_6}}  \|C_1\|_{L^{p_1}}\| \|C_3\|_{L^{p_3}}  \|C_4\|_{L^{p_4}},
\end{split}
\]
for $n\ge 1$. Hence, 
\begin{equation}\label{KNN}
K_N:=\sum_{n=N+1}^\infty\left |\int_\Omega \int_M\Phi(\omega, \cdot)(\L_{\sigma^{-n}\omega, \epsilon}^n-\L_{\sigma^{-n}\omega}^n)\hat \L_{\sigma^{-(n+1)}\omega}h_{\sigma^{-(n+1)}\omega}\, dm\, d\mathbb P(\omega) \right | \to 0,
\end{equation}
when $N\to \infty$. Moreover,  using~\eqref{fgew} we have that 
\[
\begin{split}
&\sum_{n=1}^N\int_\Omega \int_M\Phi(\omega, \cdot)(\L_{\sigma^{-n}\omega, \epsilon}^n-\L_{\sigma^{-n}\omega}^n)\hat \L_{\sigma^{-(n+1)}\omega}h_{\sigma^{-(n+1)}\omega}\, dm\, d\mathbb P(\omega)\\
&=\sum_{n=1}^N\sum_{j=1}^n\int_\Omega \int_M\Phi(\omega, \cdot)
\L_{\sigma^{-(n-j)}\omega, \epsilon}^{n-j}\tilde \L_{\sigma^{-n+j-1}\omega, \epsilon}\L^{j-1}_{\sigma^{-n}\omega}\hat \L_{\sigma^{-(n+1)} \omega}h_{\sigma^{-(n+1)}\omega}\, dm\, d\mathbb P(\omega).
\end{split}
\]
Observe that the combination of ~\eqref{weakcontr}, \eqref{tn1}, \eqref{deriv}, \eqref{top} and \eqref{AB}  imply that 
\[
\begin{split}
& \left |\int_\Omega \int_M\Phi(\omega, \cdot)
\L_{\sigma^{-(n-j)}\omega, \epsilon}^{n-j}\tilde \L_{\sigma^{-n+j-1}\omega, \epsilon}\L^{j-1}_{\sigma^{-n}\omega}\hat \L_{\sigma^{-(n+1)} \omega}h_{\sigma^{-(n+1)}\omega}\, dm\, d\mathbb P(\omega) \right |\\
&\le |\epsilon| \int_\Omega G(\omega)C_0(\sigma^{-(n-j)}\omega)C_2(\sigma^{-(n+j-1)}\omega)A(\sigma^{-n}\omega)B(\sigma^{-n+j-1}\omega)C_{34}(\sigma^{-(n+1)}\omega)\, d\mathbb P(\omega),
\end{split}
\]
where $C_{34}(\om):=C_3(\om)C_4(\om)$.
Hence, using~\eqref{coeff2} and the H\"{o}lder  inequality, we conclude that there exists a constant $D>0$ such that 
\begin{equation}\label{last}
\sum_{n=1}^N \left |\int_\Omega \int_M\Phi(\omega, \cdot)(\L_{\sigma^{-n}\omega, \epsilon}^n-\L_{\sigma^{-n}\omega}^n)\hat \L_{\sigma^{-(n+1)}\omega}h_{\sigma^{-(n+1)}\omega}\, dm\, d\mathbb P(\omega) \right | \le DN^2|\epsilon|.
\end{equation}
From~\eqref{KNN} and~\eqref{last} we obtain that 
\[
\lim_{\epsilon \to 0}|(II)_\epsilon|=0,
\]
which readily implies the conclusion of the theorem.\end{proof}

\begin{remark}
We note that in contrast to Theorem~\ref{LRThm}, in Theorem~\ref{annealed} we allowed for a full-measure set on which various conditions hold to depend on $\epsilon \in I$.
\end{remark}

\begin{remark}
In Theorem~\ref{annealed} we showed that under conditions analogous to those in the statement of Theorem~\ref{LRThm}, we obtained annealed linear response. On the other hand, in~\cite{DGS} the authors gave an explicit example which illustrates that in general quenched linear response does not imply the annealed one.
\end{remark}

\begin{remark}
We remark that one can (in the statement of Theorem~\ref{annealed}) replace the term $n^{-\beta}$ in~\eqref{spectral gap} with $\phi(n)$, where $(\phi(n))_{n\in \N}$ is any sequence of positive numbers such that $\sum_{n\ge 1}\phi(n)<+\infty$.
\end{remark}

\section{Differentiability of the variance in the CLT}\label{CLT}
Throughout this section  we suppose that $\mathcal L_{\om, \varepsilon}$ is the transfer operators of  a map $T_{\om, \varepsilon}$ acting on a compact Riemannian manifold $M$ (like in Remark \ref{Rem1}). In particular,
we assume that $\psi$ has the form $\psi(\varphi)=\int_M\varphi \, dm$, where $m$ is the normalized Lebesgue (volume) measure on $M$.
Let $f:\Omega\times M\to \R$ be a measurable function such that $\om\longmapsto \|f_\om\|_{C^3}\in L^{p_6}(\Omega,\mathcal F,\mathbb P)$ for some $p_6>8$, where $f_\om=f(\om, \cdot)$. In the circumstances of the following theorem it will follow that for each $\varepsilon\in I$ there is a number $\Sigma_{\varepsilon}^2 \ge 0$ such that for $\mathbb P$-a.e. $\om \in \Omega$ we have 
$$
\Sigma_{\varepsilon}^2=\lim_{n\to\infty}\frac1n\text{Var}_{\mu_{\om, \varepsilon}}(S_{n, \varepsilon}^\om f),
$$
where $d\mu_{\omega, \epsilon}=h_{\omega, \epsilon}\, dm$ and
$$
S_{n, \varepsilon}^\om f=\sum_{j=0}^{n-1}f_{\sigma^j\om}\circ T_{\om, \varepsilon}^j.
$$
Moreover, if we denote $f_\varepsilon(\om,\cdot)=f_{\om, \varepsilon}:=f_\om-\mu_{\om, \varepsilon}(f_\om)$ and $h_\varepsilon(\om,x)=h_{\om, \varepsilon}(x)$,
then
\begin{equation}\label{Var form}
\begin{split}
\Sigma_{\varepsilon}^2 &=\int_\Omega \int_M f_{\om, \varepsilon}^2h_{\om, \varepsilon}\, dm\, d\mathbb P+2\sum_{n\geq1}\int_\Omega\int_M f_{\om, \varepsilon} (f_{\sigma^n \omega, \varepsilon}\circ T_{\om, \epsilon}^n)h_{\om, \varepsilon}\, dm\,  d\mathbb P \\
&=\int_{\Omega \times M}f_{\varepsilon}^2 h_{\varepsilon}\, d(\mathbb P \times m)+2\sum_{n\ge 1}\int_{\Omega \times M} (h_\epsilon f_{\varepsilon}) \cdot (f_{\varepsilon} \circ \tau_\epsilon^n)\, d(\mathbb P\times m),
\end{split}
\end{equation}
where $\tau_{\varepsilon}:\Omega\times M\to \Omega\times M$ is the skew product transformation defined  by \[\tau_\varepsilon(\om,x)=(\sigma \om,T_{\om, \varepsilon}x), \quad (\om, x)\in \Omega \times M. \]

\begin{maintheorem}\label{DiffThm}
Let the conditions of Theorem \ref{LRThm} be in force with $\B_w=C^0$, $\B_s=C^1$ and $\B_{ss}=C^3$, $p_i\geq 30$, $i \in \{1, \ldots, 5\}$ and  $\beta>4$  large enough so that $\beta>1+1/a$, where $a$ is as in~\eqref{qlr}. Assume also  that for $r\in \{0, 1,2,3\}$ we have
\begin{equation}\label{TOB}
\|\mathcal L_{\om, \varepsilon}^j\|_{C^r}\leq A_r(\sigma^j\om)
\end{equation}
with $A_r\in L^{8}(\Omega,\mathcal F,\mathbb P)$  and
\begin{equation}\label{C om}
 \sup_{\|g\|_{C^2}\leq 1}\|\hat{\mathcal L}_\om g\|_{C^1}\leq C(\om)   
\end{equation}
with  $C(\om)\in L^{8}(\Omega,\mathcal F,\mathbb P)$. Suppose also that  $p_6\geq 8$.

Then the limit $\Sigma_\varepsilon^2$ exists for every $\varepsilon\in I$ and satisfies \eqref{Var form}. Moreover, the
function $\varepsilon\to \Sigma_\varepsilon^2$ is differentiable at $\varepsilon=0$. In addition, 
$$
d:=\frac{d\Sigma_{\varepsilon}^2}{d\varepsilon}\Big|_{\varepsilon=0}
$$
is given by differentiating each one of the summands in~\eqref{Var form} separately. 
\end{maintheorem}

\begin{remark}
\begin{enumerate}
\item[(i)] As will be seen from the proof, for the existence of  the limit $\Sigma_\varepsilon^2$ we need weaker integrability conditions, but this part follows a standard route and the proof is included only for readers' convenience. 

\item[(ii)] Arguing like in the proof of \cite[Theorem 2.11]{YH 23} under weaker conditions it follows that for $\mathbb P$-a.e. $\om$ and  every $\varepsilon\in I$ we have that $n^{-1/2}S_{n,\varepsilon}^\om f_\varepsilon$ converges in distribution as $n\to\infty$ to a zero mean normal random variable with variance $\Sigma^2_{\varepsilon}$, when considered as a random variable on the probability space $(M, \mu_{\om,\varepsilon})$. Thus $\Sigma^2_\varepsilon$ is the asymptotic variance in the corresponding CLT. 

\item[(iii)] A more careful analysis of the arguments in the proof yields that 
$$
\left|\Sigma_\varepsilon^2-\Sigma_0^2-\varepsilon d\right|\leq C|\varepsilon|^{1+b}
$$
for some $C>0$ and  $b=b(\beta)$ which converges to $1$ as $\beta\to\infty$. However, the proof of the differentiability itself is quite lengthy and so we will not give a precise formula for $b(\beta)$.
\end{enumerate}
\end{remark}

\begin{remark}
We note that the application of the linear response to the regularity of the variance for random dynamical systems was first discussed in~\cite[Theorem 17]{DS}. However,  Theorem~\ref{DiffThm} is the first result in the literature that deals with systems exhibiting nonuniform decay of correlations. 
\end{remark}

\begin{proof}[Proof of Theorem \ref{DiffThm}]
Let us first prove that the limit $\Sigma_\varepsilon^2$ exists and satisfies \eqref{Var form}. Relying on \eqref{spectral gap}, the proof takes a standard route (see for example \cite[Theorem 2.3]{Kifer 1998} or \cite[Lemma 12]{Nonlin}), but for readers' convenience  we will provide all the details. In order to simplify the notation, in the sequel  we omit the subscript $\varepsilon$. Moreover, we assume that $\mu_\om(f_\om)=0$ for $\mathbb P$-a.e. $\om \in \Omega$.
Firstly,
$$
\|S_n^\om f\|_{L^2(\mu_\om)}^2=\sum_{j=0}^{n-1}\mu_{\sigma^j\om}(f_{\sigma^j\om}^2)+2\sum_{i=0}^{n-1} \sum_{j=i+1}^{n-1}\mu_{\sigma^i\om}(f_{\sigma^i\om}(f_{\sigma^j\om}\circ T_{\sigma^i\om}^{j-i}))=:I_n(\om)+2J_n(\om).
$$
Applying  Birkhoff's ergodic theorem with the function $g(\om)=\mu_{\om}(f_{\om}^2)$ (using that $|g(\om)|\leq \|f_\om\|_\infty^2\in L^1(\Omega,\mathcal F,\mathbb P)$) we see that, $\mathbb P$-a.s. we have 
$$
\lim_{n\to\infty}\frac1n I_n(\om)=\int_\Omega g(\om) \, d\mathbb P(\om)=\int_\Omega \int_M f_{\om}^2h_{\om}\, dm\, d\mathbb P.
$$
Let us now handle $J_n(\om)$. Define 
$$
\Psi(\om)=\sum_{n=1}^\infty \mu_\om( f_\om  (f_{\sigma^n\om}\circ T_\om^n))=\sum_{n=1}^\infty\int_{M}\mathcal L_\om^n(f_\om h_\om)f_{\sigma^n\om}\, d m.
$$
Since $m(h_\om f_\om)=\mu_\om(f_\om)=0$, by \eqref{spectral gap}, \eqref{top} and that $\|f_\om h_\om\|_{C^1} \le 2\|f_\om\|_{C^1}\|h_\om\|_{C^1}$
we have  
$$
|\Psi(\om)|\leq 2C_1(\om)\|f_\om\|_{C^1}C_4(\om)\sum_{n=1}^\infty \|f_{\sigma^n\om}\|_\infty n^{-\beta}=:\psi(\om).
$$
Note that by the triangle and the  H\"older inequality
$$
\|\psi\|_{L^1(\Omega,\mathcal F,\mathbb P)}\leq 2\|C_4\|_{L^4}\|C_1\|_{L^4}\left\|\|f_\om\|_{C^1}\right\|_{L^4}^2\sum_{n\geq 1}n^{-\beta}<\infty. 
$$
Thus, $\Psi\in L^1(\Omega,\mathcal F,\mathbb P)$ and so 
$$
\lim_{n\to\infty}\frac1n\sum_{i=0}^{n-1}\Psi(\sigma^i\om)=\int_{\Omega}\Psi(\om)d\mathbb 
 P(\om)=\sum_{n=1}^\infty \int_\Omega\int_M f_{\om} (f_{\sigma^n \omega}\circ T_{\om}^n)h_{\om}\, dm \, d\mathbb P(\om),
$$
for $\mathbb P$-a.e. $\om \in \Omega$. Thus, it remains to show that for $\mathbb P$-a.e $\om \in \Omega$ we have
$$
\lim_{n\to\infty}\frac1n\left(J_n(\om)-\sum_{i=0}^{n-1}\Psi(\sigma^i\om)\right)=0.
$$
To this end, we write
\[
\begin{split}
&\left|J_n(\om)-\sum_{i=0}^{n-1}\Psi(\sigma^i\om)\right| \\
&=
\left|\sum_{i=0}^{n-1} \sum_{j=i+1}^{n-1}\mu_{\sigma^i\om}(f_{\sigma^i\om}(f_{\sigma^j\om}\circ T_{\sigma^i\om}^{j-i}))-\sum_{i=0}^{n-1}\Psi(\sigma^i\om)\right| \\
&=
\left|\sum_{i=0}^{n-1} \sum_{j=i+1}^{n-1}\mu_{\sigma^i\om}(f_{\sigma^i\om}(f_{\sigma^j\om}\circ T_{\sigma^i\om}^{j-i}))-\sum_{i=0}^{n-1}\sum_{k=1}^\infty \mu_{\sigma^i\om}(f_{\sigma^i\om}(f_{\sigma^{i+k}\om}\circ T_{\sigma^i\om}^{k}))\right| \\
&\leq\sum_{i=0}^{n-1}\sum_{k=n-i}^\infty\left|\mu_{\sigma^i\om}(f_{\sigma^i\om}(f_{\sigma^{i+k}\om}\circ T_{\sigma^{i}\om}^{k}))\right| \\
&=\sum_{i=0}^{n-1}\sum_{k=n-i}^\infty\left|\int_{M}\mathcal L_{\sigma^i\om}^k(f_{\sigma^i\om} h_{\sigma^i\om})f_{\sigma^{i+k}\om}\, dm\right| \\
&\leq 2\sum_{i=0}^{n-1}C_1(\sigma^i\om)C_4(\sigma^i\om)\|f_{\sigma^i\om}\|_{C^1}\sum_{k=n-i}^\infty \|f_{\sigma^{k+i}\om}\|_\infty k^{-\beta},
\end{split}
\]
where  the last estimate uses \eqref{spectral gap} and \eqref{top}.
Now, since $C_1\in L^{p_1}$, $C_4\in L^{p_4}$ and $\|f_\om\|_\infty\in L^{p_6}$, as a consequence of Birkhoff's ergodic theorem, 
there are random
variables $R_i\colon \Omega \to (0, \infty)$ for $i\in \{1, 4, 6\}$ such that $C_1(\sigma^\ell\om)\leq R_1(\om)\ell^{1/p_1}$,
$C_4(\sigma^\ell\om)\leq R_4(\om)\ell^{1/p_4}$ and $\|f_{\sigma^\ell\om}\|_\infty\leq R_6(\om)\ell^{1/p_6}$ for $\mathbb P$-a.e. $\om \in \Omega$ and all $\ell\geq 1$. Thus, with $R(\om)=R_1(\om)R_4(\om)(R_6(\om))^2$,
\[
\begin{split}
\left|J_n(\om)-\sum_{i=0}^{n-1}\Psi(\sigma^i\om)\right| &\leq R(\om)\sum_{i=0}^{n-1}i^{1/p_1+1/p_4+1/p_6}\sum_{k=n-i}^\infty (k+i)^{1/p_6}k^{-\beta} \\
&\leq R(\om)\sum_{i=0}^{n-1}i^{1/p_1+1/p_4+2/p_6}
\sum_{k=n-i}^\infty (k+1)^{1/p_6}k^{-\beta} \\
&\leq C R(\om)\sum_{i=0}^{n-1}i^{1/p_1+1/p_4+2/p_6}(n-i)^{-(\beta-1-1/p_6)} \\
&\leq C'R(\om)n^{1/p_1+1/p_4+2/p_6}=o(n),
\end{split}
\]
where $C, C'>0$ are some constants independent on $\om$ and $n$.
Here we used that $1/p_1+1/p_4+2/p_6<1$ and that
$$
\sum_{i=0}^{n-1}(n-i)^{-(\beta-1-1/p_6)}=\sum_{k=1}^{n}k^{-(\beta-1-1/p_6)}\leq \sum_{k\geq  1}k^{-(\beta-1-1/p_6)}<\infty,
$$
which is true since $\beta-1-1/p_6>1$.

Next we prove the differentiability of $\Sigma_\varepsilon^2$ at $\varepsilon=0$.
Let us first deal with the  term  
$$
d_0(\varepsilon):=\int_{\Omega \times M} f_\varepsilon^2 h_{\varepsilon}\, d(\mathbb P\times m).
$$
We have 
\[
\begin{split}
(d_0(\varepsilon)-d_0(0))/\varepsilon &=
\int_\Omega\psi(\varepsilon^{-1}(h_{\om, \varepsilon}-h_{\om, 0})f_{\om, \varepsilon}^2)\, d\mathbb P(\om)+\int_\Omega \psi(h_{\om, 0}\varepsilon^{-1} (f_{\om, \varepsilon}^2-f_{\om, 0}^2))\, d\mathbb P(\om) \\
&=:I_1(\varepsilon)+I_2(\varepsilon).
\end{split}
\]
Next, using~\eqref{qlr} and that $\|f_{\om, \varepsilon}\|_\infty \le 2\|f_\om\|_\infty$, we have
$$
\|\varepsilon^{-1}(h_{\om, \varepsilon}-h_{\om, 0})f_{\om, \varepsilon}^2-\hat h_\om f_{\om, \varepsilon}^2\|_\infty\leq   4U_1(\omega)  \|f_\om\|_\infty^2 |\varepsilon|^a,
$$
and so 
$$
|I_1(\varepsilon)-J_1(\varepsilon)|\leq 4 \mathbb E[\|f_\om\|_\infty^2 U_1(\om)]|\varepsilon|^a,
$$
where 
$$
J_1(\varepsilon):=\int_\Omega \psi(\hat h_\om f^2_{\om, \varepsilon})\, d\mathbb P(\om)
=\int_\Omega \psi(\hat h_\om f^2_{\om, 0})d\mathbb P(\om)+\int_\Omega \psi(\hat h_\om (f^2_{\om, \varepsilon}-f^2_{\om, 0}))d\mathbb P(\om).
$$
Now, using that 
\begin{equation}\label{f diff}
f_{\om,\varepsilon}=f_{\om, 0}+\psi(f_\om(h_{\om, 0}-h_{\om,\varepsilon}))  \end{equation}
together with Proposition \ref{Stab Prop}  we see that
$$
\|f^2_{\om, \varepsilon}-f^2_{\om, 0}\|_\infty \leq 4|\varepsilon|\|f_\om\|_\infty^2\tilde U(\om).
$$
Integrating with respect to $\mathbb P$ we get that 
$$
|J_1(\varepsilon)-J_1(0)|\leq C|\varepsilon|
$$
where
\begin{equation}\label{constantC}
C:=4\mathbb E[\|f_\om\|^2\tilde U(\om)\|\hat h_\om\|_\infty]<\infty.
\end{equation}
Note that by applying~\eqref{qlr} with $\epsilon=\epsilon_0$ where $\epsilon_0 \in I\cap (0, 1)$ is arbitrary,  using that $\|\cdot\|_w=\|\cdot\|_{\infty}$ and \eqref{top}, we get that 
\begin{equation}\label{hat h bd}
\|\hat h_\om\|_\infty\leq  U_1(\om)+ \frac{2}{\epsilon_0}C_4(\om)\in L^{s}(\Omega,\mathcal F,\mathbb P). 
\end{equation}
Thus $C$ in~\eqref{constantC} is finite in view of our assumptions that guarantee that $1/s+2/p_6+1/p_1+1/p_2+1/p_3\leq 1$.
Combining the above estimates we get that,
$$
|I_1(\varepsilon)-J_1(0)|\leq C'|\varepsilon|^a.
$$

To estimate $I_2(\varepsilon)$ we need to further expand $f_{\om, \varepsilon}^2$. First, using~\eqref{f diff} and that 
$a^2-b^2=-(2a+(b-a))(b-a)$ for all $a,b\in\mathbb R$ we see that 
\begin{equation}\label{e}
 f^2_{\om, \varepsilon}=f^2_{\om, 0}-(2f_{\om, 0}-\psi(f_\om(h_{\om, \varepsilon}-h_{\om, 0})))\psi(f_\om(h_{\om, \varepsilon}-h_{\om, 0})).   
\end{equation}

Now, using~\eqref{qlr} we have 
$$
|\psi(f_\om(h_{\om, \varepsilon}-h_{\om, 0}))-\varepsilon\psi(f_\om\hat h_\om)|\leq \|f_\om\|_\infty U_1(\om)|\varepsilon|^{1+a},
$$
and so using also Proposition~\ref{Stab Prop} to bound the term $(\psi(f_\om(h_{\om, \varepsilon}-h_{\om, 0})))^2$ we see that 
$$
\|f^2_{\om, \varepsilon}-f^2_{\om, 0}+2\varepsilon f_{\om, 0}\psi(\hat h_\om f_\om)\|_\infty\leq 4\|f_\om\|_\infty^2 U_1(\om)|\varepsilon|^{1+a}+|\varepsilon|^2\|f_\om\|_\infty^2\tilde U^2(\om).
$$
Combining the above estimates and using \eqref{top} and that $a<1$ we see that 
$$
\left|\psi(\varepsilon^{-1}h_{\om,0}(f_{\om,\varepsilon}^2-f_{\om,0}^2))+2\psi(h_{\om,0}f_{\om,0})\psi(\hat h_\om f_\om)\right|\leq C|\varepsilon|^aH(\om)
$$
where
$$
H(\om):=
 (\tilde U^2(\om)+U_1(\om))\|f_\om\|_\infty^2 C_4(\om) 
$$
and $C>0$ is a constant.
Now, notice that with  $q$  as  in Proposition~\ref{Stab Prop} we have  $2/p_6+2/q+1/p_4\leq 1$ and $2/p_6+2/q+1/p_4+1/s\leq 1$. 
Thus, 
$
H(\om)\in L^1(\Omega,\mathcal F,\mathbb P).
$
Integrating with respect to $\mathbb P$ we conclude that
$$
|I_2(\varepsilon)-J_2|\leq C|\varepsilon|^{a}
$$
where 
$$
J_2=-2\int_\Omega \psi(h_{\om, 0}f_{\om, 0})\psi(\hat h_\om f_\om)\, d\mathbb P(\om).
$$
Thus, 
$$
d_0'(0)=J_1(0)+J_2
$$
and, in fact,
$$
|d_0(\varepsilon)-d_0(0)-\varepsilon d_0'(0)|\leq C|\varepsilon|^a.
$$

Now, let us deal with the second term 
$$
\sum_{n\ge 1}\int_{\Omega \times M} (h_{\varepsilon} f_{\varepsilon} ) \cdot ( f_{\varepsilon} \circ \tau_\epsilon^n)\, d(\mathbb P\times m).
$$
Notice that~\eqref{spectral gap} and~\eqref{top} imply that 
\begin{equation}\label{estt0}
\|\mathcal L_{\om, \varepsilon}^n(h_{\om, \varepsilon}f_{\om, \varepsilon})\|_\infty \leq 4C_1(\om)\|f_\om\|_{C^1}C_4(\om)n^{-\beta}   .
\end{equation}
Let us denote 
$$
C_n(\varepsilon):=\int_{\Omega \times M} (h_{\varepsilon} f_{\varepsilon}) \cdot (f_{\varepsilon} \circ \tau_\epsilon^n)\, d(\mathbb P\times m)=
\int_\Omega \psi(\mathcal L_{\omega,\varepsilon}^n(h_{\om, \varepsilon}f_{\om, \varepsilon}) f_{\sigma^n\om, \varepsilon})\, d\mathbb P(\om)
$$
and 
$$
D_n(\varepsilon):=\frac{C_n(\varepsilon)-C_n(0)}{\varepsilon}.
$$
Then by \eqref{estt0},
\begin{equation}\label{estt 1}
\begin{split}
 |D_n(\varepsilon)| &\leq  |\varepsilon|^{-1}(|C_n(\varepsilon)|+|C_n(0)|) \\
 &\leq 8|\varepsilon|^{-1}n^{-\beta}\int_\Omega \|f_\om\|_{C^1}\|f_{\sigma^n\om}\|_{C^1} C_1(\omega)C_4(\om)\, d\mathbb P(\om) \\
 &\leq C |\varepsilon|^{-1}n^{-\beta},
 \end{split}
\end{equation}
for some constant $C>0$ since 
  $\omega \mapsto \|f_\om\|_{C^1}, C_1$ and $C_4$ are in $L^4(\Omega,\mathcal F,\mathbb P)$ (due to $p_6,p_1,p_4\geq 4$).
Now let us fix some  $0<\gamma<\min(a,1/3)$ such that $\gamma(\beta-1)>1$. This is possible since $\beta>\max(4,1+1/a)$. Then by \eqref{estt 1},
\begin{equation}\label{esst 2}
\left|\sum_{n\geq |\varepsilon|^{-\gamma}}D_n(\varepsilon)\right|\leq C|\varepsilon|^{-1}\sum_{n \ge |\varepsilon|^{-\gamma}}n^{-\beta}\leq C'|\varepsilon|^{\gamma(\beta-1)-1}.
\end{equation}
Since $\gamma(\beta-1)>1$ we see that the contribution of the sums $\sum_{n\geq |\varepsilon|^{-\gamma}}D_n(\varepsilon)$ is negligible.

Now, let us analyze $\sum_{n<|\varepsilon|^{-\gamma}}D_n(\varepsilon)$. Fix some $n<|\varepsilon|^{-\gamma}$. 
Then 
$$
D_n(\varepsilon)=d_{1,n}(\varepsilon)+d_{2,n}(\varepsilon)+d_{3,n}(\varepsilon)
$$
where 
$$
d_{1,n}(\varepsilon)=\int_\Omega \psi(\mathcal L_{\omega,\varepsilon}^n(h_{\om, \varepsilon}f_{\om, \varepsilon})\varepsilon^{-1}[f_{\sigma^n\om, \varepsilon}-f_{\sigma^n\om, 0}]) \, d\mathbb P(\om),
$$
$$
d_{2,n}(\varepsilon)=\int_\Omega \psi(\mathcal L_{\om, \varepsilon}^n(\varepsilon^{-1}(h_{\om, \varepsilon}f_{\om, \varepsilon}-h_{\om, 0}f_{\om, 0}))f_{\sigma^n\om, 0})\, d\mathbb P(\om)
$$
and 
$$
d_{3,n}(\varepsilon)=
\int_\Omega \psi([\varepsilon^{-1}(\mathcal L_{\om, \varepsilon}^n-\mathcal L_{\omega}^n)](h_{\om, 0}f_{\om, 0})f_{\sigma^n\om, 0})\, d\mathbb P(\om).
$$
We note that $d_{1, n}(\varepsilon)=0$. In fact, since $\varepsilon^{-1}[f_{\sigma^n \omega, \varepsilon}-f_{\sigma^n \omega, 0}]$ depends only on $\varepsilon$ and $\omega$ (see~\eqref{f diff}), we have 
\[
\begin{split}
d_{1, n}(\varepsilon) &=\int_\Omega \varepsilon^{-1}[f_{\sigma^n\om, \varepsilon}-f_{\sigma^n\om, 0}]  \psi(\mathcal L_{\omega,\varepsilon}^n(h_{\om, \varepsilon}f_{\om, \varepsilon})) \, d\mathbb P(\om) \\
&=\int_\Omega \varepsilon^{-1}[f_{\sigma^n\om, \varepsilon}-f_{\sigma^n\om, 0}]  \psi(h_{\om, \varepsilon}f_{\om, \varepsilon})\, d\mathbb P(\om) \\
&=0,
\end{split}
\]
as $\psi(h_{\omega, \varepsilon}f_{\omega, \varepsilon})=0$.
In order to estimate $d_{2,n}(\varepsilon)$,  we  note that it follows from~\eqref{f diff} that 
\[
\begin{split}
\varepsilon^{-1}(h_{\om, \varepsilon}f_{\om, \varepsilon}-h_{\om, 0}f_{\om, 0}) &=\hat h_\om f_{\om, 0}-h_{\om, 0}\psi(f_\om \hat h_\om)\\
&\phantom{=}+\left(\delta_{\om, \varepsilon}f_{\om, \varepsilon}+\hat h_\om \psi(f_\om(h_{\om, 0}-h_{\om, \varepsilon})) -h_{\om, 0}\psi(f_\om\delta_{\om, \varepsilon})\right),
\end{split}
\]
where with
\begin{equation}\label{eta delta def}
\eta_{\om, \varepsilon}:=\frac{h_{\om, \varepsilon}-h_{\om, 0}}{\varepsilon}, \,\,\delta_{\om,\varepsilon}:=\eta_{\om, \varepsilon}-\hat h_\om.    
\end{equation}
Next, using Proposition~\ref{Stab Prop} we have
$$
|\hat h_\om \psi(f_\om(h_{\om, \varepsilon}-h_{\om, 0}))|\leq \|\hat h_\om\|_\infty\|f_\om\|_\infty \tilde U(\om)|\varepsilon|. 
$$
Moreover, by Theorem \ref{LRThm} we have 
$$
\|\delta_{\om, \varepsilon}f_{\om, \varepsilon}\|_\infty+\|h_{\om,0}\psi(f_\om \delta_{\om, \varepsilon})\|_\infty\leq U_1(\om)\|f_\om\|_\infty|\varepsilon|^a(2+C_4(\om)).
$$
By integrating with respect to  $\mathbb P$ and summing up all the $|\varepsilon|^{-\gamma}$ terms we conclude that
$$
\sum_{n<|\varepsilon|^{-\gamma}}|d_{2,n}(\varepsilon)-\tilde d_{2,n}(\varepsilon)|\leq C|\varepsilon|^{a-\gamma},
$$
where 
\[
\begin{split}
\tilde d_{2,n}(\varepsilon) &=\int_\Omega \psi(\mathcal L_{\om, \varepsilon}^n(\hat h_\om f_{\om, 0}-h_{\om, 0}\psi(f_\om \hat h_\om))f_{\sigma^n\om, 0})\, d\mathbb P \\
&=\int_\Omega \psi(\mathcal L_{\om, \varepsilon}^n(\hat h_\om f_{\om, 0})f_{\sigma^n\om, 0})\, d\mathbb P-\int_\Omega \psi(\mathcal L_{\om, \varepsilon}^n(h_{\om, 0}\psi(f_\om \hat h_\om))f_{\sigma^n\om, 0})\, d\mathbb P \\
&=:\tilde d^{(1)}_{2,n}(\varepsilon)-\tilde d^{(2)}_{2,n}(\varepsilon).
\end{split}
\]
Note that $C$ above is finite because of \eqref{hat h bd}, that $\psi(\mathcal L_{\om,\varepsilon}\textbf{1})=1$ and that $s,p_6,q,p_4\geq 4$. Recall also that $\gamma<a$.

Let us now verify the summability of each one of $\tilde d_{2,n}^{(i)}(\varepsilon)$ (uniformly in $\varepsilon$) for $i=1,2$.
We begin with the case $i=2$. We have 
$$
\tilde d^{(2)}_{2,n}(\varepsilon)=
\int_\Omega\psi(f_\om \hat h_\om) \psi(\mathcal L_{\om, \varepsilon}^n(h_{\om, 0})f_{\sigma^n\om, 0})\, d\mathbb P.
$$
By \eqref{spectral gap} and~\eqref{top},
$$
\|\mathcal L_{\om,\varepsilon}^n h_{\om, 0}-h_{\sigma^n\om, \varepsilon}\|_\infty \leq 2C_1(\om)C_4(\om)n^{-\beta}.
$$
Therefore, using also \eqref{f diff} and that $\psi(h_{\sigma^n \om, \varepsilon}f_{\sigma^n \om, \varepsilon})=0$, we see that 
\[
\begin{split}
\left|\psi(\mathcal L_{\om, \varepsilon}^n(h_{\om, 0})f_{\sigma^n\om, 0})\right| &\leq 
\left|\psi(\mathcal L_{\om, \varepsilon}^n(h_{\om, 0})(f_{\sigma^n\om, \varepsilon}-f_{\sigma^n\om, 0}))\right| 
+\left|\psi(\mathcal L_{\om, \varepsilon}^n(h_{\om, 0})f_{\sigma^n\om, \varepsilon})\right| \\
&\leq \|\mathcal L_{\om, \varepsilon}^n h_{\om, 0}\|_{L^1}\|f_{\sigma^n\om}\|_\infty\|h_{\sigma^n \om, \varepsilon}-h_{\sigma^n \om, 0}\|_\infty  \\
&\phantom{\leq}+4C_1(\om)C_4(\om)\|f_{\sigma^n \om}\|_\infty n^{-\beta}.
\end{split}
\]
Now, using Proposition \ref{Stab Prop} and that $\|\mathcal L_{\om, \varepsilon}^n h_{\om, 0}\|_{L^1}=\|h_{\om, 0}\|_{L^1}=1$ we conclude that 
$$
\left|\psi(\mathcal L_{\om, \varepsilon}^n(h_{\om, 0})f_{\sigma^n\om, 0})\right| \leq  \tilde U(\om)\|f_{\sigma^n\om}\|_\infty|\varepsilon|+ 4C_1(\om)C_4(\om)\|f_{\sigma^n \om}\|_\infty n^{-\beta}.
$$
Taking into account $|\varepsilon|<n^{-1/\gamma}$ and \eqref{hat h bd} we get the desired summability  of $\tilde d_{2,n}^{(2)}(\varepsilon)$ by integrating with respect to $\mathbb P$,
since $\beta>1$ and $\gamma<1$ and $C_4(\om), C_1(\om), U_1(\om), \|f_\om\|_\infty, \tilde U(\om)$ belong to $L^5(\Omega,\mathcal F,\mathbb P)$.

Now we estimate $\tilde d_{2,n}^{(1)}(\varepsilon)$.
First, by Theorem \ref{LRThm}  for every $r>0$ sufficiently small we have
$$
\|\hat h_\om-\eta_{\om,r}\|_\infty\leq U_1(\om)r^a
$$
where $\eta_{\om,r}$ is as in~\eqref{eta delta def}. Therefore, 
\begin{equation}\label{aabove}
 \|\hat h_\om f_{\om,0}-\eta_{\om,r} f_{\om, 0}\|_\infty\leq 2U_1(\om)\|f_\om\|_\infty r^a.   
\end{equation}
Now, since 
$$
\psi(f_{\om,0}\hat h_\om)=\psi(f_\om  \hat h_\omega)
$$
we see that for a given $n,\varepsilon$, $r$ small enough and $x\in M$ we have 
$$
|\mathcal L_{\om, \varepsilon}^n(\hat h_\om f_{\om, 0})(x)-\mathcal L_{\om, \varepsilon}^n(\eta_{\om,r}f_{\om, 0})(x)|\leq 2\mathcal L_{\om, \varepsilon}^n\textbf{1}(x)
U_1(\om)\|f_\om\|_\infty r^a,
$$
and so 
\[
\begin{split}
\left|\psi(\mathcal L_{\om, \varepsilon}^n(\hat h_\om f_{\om, 0})f_{\sigma^n\om, 0})-\psi(\mathcal L_{\om, \varepsilon}^n(\eta_{\om,r}f_{\om, 0})f_{\sigma^n\om, 0})\right| &\leq 2U_1(\om)\|f_\om\|_\infty r^a\psi(\mathcal L_{\om, \varepsilon}^n\textbf{1}|f_{\sigma^n\om, 0}|)  \\
&\leq 4U_1(\om)\|f_\om\|_\infty\|f_{\sigma^n\om}\|_\infty r^a.
\end{split}
\]
On the other hand, since $\|\eta_{\om,r}\|_{C^1}\leq 2r^{-1}C_4(\om)$,  using \eqref{spectral gap} and \eqref{aabove} we see that
\[
\begin{split}
&\|\mathcal L_{\om, \varepsilon}^n(\eta_{\om,r} f_{\om, 0})-\psi(\hat h_\om f_{\om, 0})h_{\sigma^n\om, \varepsilon}\|_\infty 
 \\
 &\leq
\|\mathcal L_{\om, \varepsilon}^n(\eta_{\om,r} f_{\om, 0})-\psi(\eta_{\om,r}f_{\om, 0})h_{\sigma^n\om, \varepsilon}\|_\infty 
+|\psi(\eta_{\om,r}f_{\om, 0})h_{\sigma^n\om, \varepsilon}-\psi(\hat h_\om f_{\om, 0})h_{\sigma^n\om, \varepsilon}| 
\\ 
&\leq C_1(\om)n^{-\beta} (\| \eta_{\om, r}f_{\om, 0}\|_{C^1}+\|\psi(\eta_{\om, r}f_{\om, 0})h_{\om, \varepsilon}\|_{C^1})+C_4(\sigma^n \om)|\psi(\eta_{\om,r}f_{\om, 0})-\psi(\hat h_\om f_{\om, 0})|\\
&\leq C_1(\om)n^{-\beta}(8\|f_\om\|_{C^1}C_4(\om)r^{-1}+8(C_4(\om))^2\|f_\om\|_{C^1}r^{-1}) +2C_4(\sigma^n \om)U_1(\om)\|f_\om\|_\infty r^a \\
&\leq   8\left(C_1(\om)(C_4(\om))^2n^{-\beta}\|f_\om\|_{C^1} r^{-1}+C_1(\om)C_4(\om)n^{-\beta}\|f_\om\|_{C^1}r^{-1}+
U_1(\om)C_4(\sigma^n\om)\|f_\om\|_\infty r^a\right).
\end{split}
\]
We conclude that for all $r>0$ sufficiently small,
$$
|\psi(\mathcal L_{\om, \varepsilon}^n(\hat h_\om f_{\om, 0}) f_{\sigma^n\om, 0})-\psi(\hat h_\om f_{\om, 0})\psi(h_{\sigma^n\om, \varepsilon}f_{\sigma^n\om, 0})|
$$
$$
\leq 24\|f_{\sigma^n\om}\|_\infty \bar C(\om, n) (r^{a}+r^{-1}n^{-\beta}),
$$
where 
\[
\bar C(\om, n):=\max \{C_1(\om)C_4(\om)\|f_\om\|_{C^1}, C_1(\om)(C_4(\om))^2\|f_\om\|_{C^1}, U_1(\om)\|f_\om\|_\infty C_4(\sigma^n\om)\}.
\]
Notice next that
\[
\begin{split}
\left|\psi(\hat h_\om f_{\om, 0})\psi(h_{\sigma^n\om, \varepsilon}f_{\sigma^n\om, 0})\right| &=\left|\psi(\hat h_\om f_{\om, 0})\psi(h_{\sigma^n\om, \varepsilon}(f_{\sigma^n\om, 0}-f_{\sigma^n\om, \varepsilon}))\right| \\
&\leq 
2\|\hat h_\om\|_{\infty}\|f_\om\|_\infty C_4(\sigma^n\om)\|f_{\sigma^n\om, 0}-f_{\sigma^n\om, \varepsilon}\|_\infty  \\
&\leq
2\|\hat h_\om\|_{\infty}\|f_\om\|_\infty C_4(\sigma^n\om)\|f_{\sigma^n\om}\|_\infty\tilde U(\sigma^n\om)|\varepsilon|  \\
&\leq 
2\|\hat h_\om\|_{\infty}\|f_\om\|_\infty C_4(\sigma^n\om)\|f_{\sigma^n\om}\|_\infty\tilde U(\sigma^n\om)n^{-1/\gamma}, 
\end{split}
\]
where in the penultimate inequality we have used \eqref{f diff} and Proposition~\ref{Stab Prop}, and in the last inequality we have used that $|\varepsilon|<n^{-1/\gamma}$.

Taking $r=r_n=n^{-\frac\beta{a+1}}$ (so that $r^{a}=r^{-1}n^{-\beta}$) we conclude that 
\[
\begin{split}
&|\psi(\mathcal L_{\om, \varepsilon}^n(\hat h_\om 
f_{\om, 0})f_{\sigma^n\om, 0})|  \\
&\leq
c\|f_{\sigma^n\om}\|_\infty \bar C(\om, n) n^{-\frac{\beta a} {a+1}} 
+2\|\hat h_\om\|_{\infty}\|f_\om\|_\infty C_4(\sigma^n\om)\|f_{\sigma^n\om}\|_\infty\tilde U(\sigma^n\om)n^{-1/\gamma}.
\end{split}
\] 
This together with~\eqref{hat h bd}, that $s,p_6,p_4, p_1, q\geq 5$, $\gamma<1$ and our assumption that 
$\beta>1+\frac 1a$ implies that   
for every $\varepsilon$ and $n$ such that $n<|\varepsilon|^{-\gamma}$ we have 
$|\tilde d_{2,n}(\varepsilon)|\leq Cn^{-1-\zeta}$ for some $C,\zeta>0$ which do not depend on $n$. This allows us to pass to sum of the limits $\lim_{\varepsilon\to 0}\tilde d_{2,n}(\varepsilon)$. 

Next we handle $d_{3,n}(\varepsilon)$ for $n< |\varepsilon|^{-\gamma}$. First,
\[
\begin{split}
\varepsilon^{-1}(\mathcal L_{\om, \varepsilon}^n-\mathcal L_{\omega, 0}^n)
&=\sum_{j=0}^{n-1}
\mathcal L_{\sigma^{j+1}\om, \varepsilon}^{n-j-1}\varepsilon^{-1}(\mathcal L_{\sigma^j\om, \varepsilon}-\mathcal L_{\sigma^j\om, 0})\mathcal L_{\om, 0}^{j} \\
&=\sum_{j=0}^{n-1}\mathcal L_{\sigma^{j+1}\om, \varepsilon}^{n-j-1}\hat{\mathcal L}_{\sigma^j\om}\mathcal L_{\om, 0}^{j}+\sum_{j=0}^{n-1}\mathcal L_{\sigma^{j+1}\om, \varepsilon}^{n-j-1} L_{\sigma^j\om,\varepsilon}\mathcal L_{\om, 0}^{j}\\
&=:I_n(\om, \varepsilon)+J_n(\om, \varepsilon),
\end{split}
\]
where 
$$
L_{\om,\varepsilon}:=\varepsilon^{-1}(\mathcal L_{\om, \varepsilon}-\mathcal L_{\om, 0})-\hat{\mathcal L}_\om. 
$$
In these notations we have
$$
d_{3,n}(\varepsilon)=\mathcal I_n(\varepsilon)+\mathcal J_n(\varepsilon),
$$
where 
$$
\mathcal I_n(\varepsilon)=\int_\Omega \psi(f_{\sigma^n\om, 0}\cdot I_n(\om, \varepsilon)[h_{\om, 0}f_{\om, 0}]) \,d\mathbb P(\om)
$$
and 
$$
\mathcal J_n(\varepsilon)=\int_\Omega \psi(f_{\sigma^n\om, 0}\cdot J_n(\om, \varepsilon)[h_{\om, 0}f_{\om, 0}]) \, d\mathbb P(\om).
$$
Now, using our assumption \eqref{TOB} 
we see that for all $g\in C^3$,
$$
\|J_n(\om, \varepsilon)g\|_{\infty}\leq \|g\|_{C^3}|\varepsilon|\sum_{j=0}^{n-1}C_3(\sigma^j\om)A_3(\sigma^j\om)A_0(\sigma^{n}\om).
$$
Taking into account that $|\varepsilon|<n^{-1/\gamma}$ and using that $\|h_{\om, 0}\|_{C^3}\leq C_4(\om)$, we see that
$$
|\mathcal J_n(\varepsilon)|\leq Cn^{-1/\gamma}\sum_{j=0}^{n-1}\int_\Omega C_4(\om)\|f_{\om}\|_{C^3}\|f_{\sigma^n\om}\|_\infty A_0(\sigma^n\om)A_3(\sigma^j\om)C_3(\sigma^j\om) \, d\mathbb P(\om).
$$
Thus, since  $C_4(\om), \|f_\om\|_{C^3}, A_0(\om),A_3(\om), C_3(\om)\in L^{6}(\Omega,\mathcal F,\mathbb P)$, we have that 
$$
|\mathcal J_n(\varepsilon)|\leq C'n^{1-1/\gamma},
$$
for some $C'>0$.
Since $1/\gamma>2$ we get the appropriate summability of the terms $\mathcal J_n(\varepsilon)$.

Next, let us write 
$$
\mathcal I_n(\varepsilon)=
\mathcal A_n+
\mathcal D_n(\varepsilon)
$$
where with $\L_\om=\L_{\om, 0}$,
$$
\mathcal A_n=\sum_{j=0}^{n-1}\int_\Omega \psi(\mathcal L_{\sigma^{j+1}\om}^{n-j-1}\hat{\mathcal L}_{\sigma^j\om}\mathcal L_{\om}^j(h_{\om, 0} f_{\om, 0})f_{\sigma^n\om, 0}) \, d\mathbb P (\om)
$$
and 
$$
\mathcal D_n(\varepsilon)=\sum_{j=0}^{n-1}\int_\Omega \psi([\mathcal L_{\sigma^{j+1}\om, \varepsilon}^{n-j-1}-\mathcal L_{\sigma^{j+1}\om}^{n-j-1}]\hat{\mathcal L}_{\sigma^j\om}\mathcal L_{\om}^j(h_{\om, 0} f_{\om, 0})f_{\sigma^n\om, 0})\,  d\mathbb P(\om).
$$
Let us bound $|\mathcal D_n(\varepsilon)|$.
We have
$$
\mathcal L_{\sigma^{j+1}\om, \varepsilon}^{n-j-1}-\mathcal L_{\sigma^{j+1}\om}^{n-j-1}=\sum_{k=0}^{n-j-2}\mathcal L_{\sigma^{j+k+2}\om, \varepsilon}^{n-j-k-2}(\mathcal L_{\sigma^{j+k+1}\om, \varepsilon}-\mathcal L_{\sigma^{j+k+1}\om})\mathcal L_{\sigma^{j+1}\om}^k
$$
and so by \eqref{TOB} applied with $r=0$, \eqref{tn1} and then  \eqref{TOB} applied with $r=1$,    for every $C^1$ function $g$ we have
$$
\left\|\mathcal L_{\sigma^{j+1}\om, \varepsilon}^{n-j-1}g-\mathcal L_{\sigma^{j+1}\om}^{n-j-1}g\right\|_{\infty}\leq |\varepsilon|A_0(\sigma^n\om)\|g\|_{C^1}\sum_{k=0}^{n-j-2}C_2(\sigma^{j+k+1}\om)A_1(\sigma^{j+k+1}\om).
$$
Using  \eqref{C om}
 we conclude that
\[
\begin{split}
& |\mathcal D_n(\varepsilon)| \\
&\leq C 
|\varepsilon|\int_\Omega A_0(\sigma^n\om)\|f_{\sigma^n\om}\|_\infty\|f_\om\|_{C^2}C_4(\om)\sum_{j=0}^{n-1}\sum_{k=0}^{n-j-2}C(\sigma^j\om)C_2(\sigma^{j+k+1}\om)A_1(\sigma^{j+k+1}\om)A_2(\sigma^{j}\om)\, d\mathbb P(\om) \\
&
\leq C'n^2|\varepsilon|\leq C'n^{-(1/\gamma-2)},
\end{split}
\]
where  the last inequality uses that 
$$
\om\mapsto A_0(\om), \|f_\om\|_{C^2}, C_4(\om), C(\om), C_2(\om), A_1(\om), A_2(\om) \in L^8(\Omega,\mathcal F,\mathbb P).
$$
Thus we get the summability since $\gamma<1/3$.

Now in order to prove 
summability in $n$ of $\mathcal A_n$, it is enough to prove summability of $A_n$ defined by  
$$
A_n:=\sum_{j=0}^{n-1}\left\|\psi(\mathcal L_{\sigma^{j+1}\om}^{n-j-1}\hat{\mathcal L}_{\sigma^j\om}\mathcal L_{\om}^j(h_{\om, 0} f_{\om, 0})f_{\sigma^n\om, 0})\right\|_{L^1(\mathbb P)}.
$$
If $n-j-1\geq [n/2]-1$ then we use \eqref{spectral gap}  with $\mathcal L_{\sigma^{j+1}\om}^{n-j-1}$ to get that
$$
\left\|\psi(\mathcal L_{\sigma^{j+1}\om}^{n-j-1}\hat{\mathcal L}_{\sigma^j\om}\mathcal L_{\om}^j(h_{\om, 0} f_{\om, 0})f_{\sigma^n\om, 0})\right\|_{L^1(\mathbb P)}\leq Cn^{-\beta}.
$$
Here we used that
$C_1(\om),\|f_\om\|_{C^2},C_4(\om), A_2(\om), C(\om)\in L^6$ and that 
\[
\begin{split}
\|\hat{\mathcal L}_{\sigma^j\om}\mathcal L_{\om}^j(h_{\om, 0} f_{\om, 0})\|_{C^1} &\leq
4C(\sigma^j\om)\|\mathcal L_\om^j\|_{C^2}\|f_\om\|_{C^2}\|h_{\om, 0}\|_{C^2} \\
&\leq 4C(\sigma^j\om)A_2(\sigma^j\om)C_4(\om)\|f_\om\|_{C^2},
\end{split}
\]
where
$C(\om)$ is as in~\eqref{C om}.

If $j\geq [n/2]$ we cannot directly use \eqref{spectral gap} since \eqref{spectral gap}  provides estimates in the supremum norm and $\hat{\mathcal L}_{\sigma^j\om}$ is not continuous as an operator from $C^1$ to $C^0$.  However, using \eqref{deriv} with the function $h=h_{\om,j}=\mathcal L_{\om}^j(h_{\om, 0} f_{\om, 0})$ we see that for all $\delta >0$ sufficiently small, 
$$
\left\|\hat{\mathcal L}_{\sigma^j\om}(h_{\om,j})-\Delta_{\sigma^j\om,\delta}(h_{\om,j})\right\|_{C^1}\leq 6\delta C_3(\sigma^j\om)A_3(\sigma^j\om)C_4(\om)\|f_\om\|_{C^3},
$$
where we used that $\|h_{\om,j}\|_{C^3}\leq 6A_3(\sigma^j\om)C_4(\om)\|f_\om\|_{C^3}$. Here 
$$
\Delta_{\om,\delta}:=\frac{\mathcal L_{\om, \delta}-\mathcal L_{\om, 0}}{\delta},
$$
which satisfies 
$$
\|\Delta_{\om,\delta}\|_{\infty}\leq 2A_0(\om)/\delta.
$$
Using \eqref{spectral gap} we get that 
$$
\|h_{\om,j}\|_\infty=\|\mathcal L_{\om}^j(h_{\om, 0} f_{\om, 0})\|_{\infty}\leq 4\|h_{\om, 0}\|_{C^1}\|f_\om\|_{C^1}C_1(\om)j^{-\beta}\leq C_\beta C_4(\om)C_1(\om)\|f_\om\|_{C^1}n^{-\beta},
$$
for some constant $C_\beta>0$ which depends only on $\beta$. 
Thus 
\[
\begin{split}
\|\hat{\mathcal L}_{\sigma^j\om}(h_{\om,j})\|_\infty &\leq  6\delta C_3(\sigma^j\om)A_3(\sigma^j\om)C_4(\om)\|f_\om\|_{C^3} \\
&\phantom{\leq}+2C_\beta\delta^{-1}A_0(\sigma^j\om) C_4(\om)C_1(\om)\|f_\om\|_{C^1}n^{-\beta}.
\end{split}
\]
Taking $\delta=n^{-\beta/2}$ we conclude that 
$$
\|\hat{\mathcal L}_{\sigma^j\om}(h_{\om,j})\|_\infty\leq C(\om,j,f)n^{-\beta/2}
$$
where 
$$
C(\om,j,f):=2C_\beta A_0(\sigma^j\om)C_4(\om)C_1(\om)\|f_\om\|_{C^1}+ 6C_3(\sigma^j\om)A_3(\sigma^j\om)C_4(\om)\|f_\om\|_{C^3}.
$$
Therefore, if $j \ge [n/2]$ then
$$
\left|\psi(\mathcal L_{\sigma^{j+1}\om}^{n-j-1}\hat{\mathcal L}_{\sigma^j\om}\mathcal L_{\om}^j(h_\om f_{\om, 0})f_{\sigma^n\om, 0})\right|\leq 2C(\om,j,f)A_0(\sigma^n\om)\|f_{\sigma^n\om}\|_\infty n^{-\beta/2},
$$
and consequently 
$$
\left\|\psi(\mathcal L_{\sigma^{j+1}\om}^{n-j-1}\hat{\mathcal L}_{\sigma^j\om}\mathcal L_{\om}^j(h_\om f_{\om, 0})f_{\sigma^n\om, 0})\right\|_{L^1(\mathbb P)}\leq Cn^{-\beta/2}
$$
where we have used that  $C_4(\om),C_3(\om),C_1(\om), A_0(\om), \|f_\om\|_{C^3}, A_3(\om)\in L^7(\Omega,\mathcal F,\mathbb P)$.

We conclude that 
$$
|A_n|\leq C n^{-\beta/2+1}
$$
and therefore, since  $\beta>4$ we get the desired summability.

Finally, putting together all the above estimates  we conclude that $\varepsilon\to \Sigma^2_\varepsilon$ is differentiable at $0$ and
$$
\frac{d\Sigma_\varepsilon^2}{d\varepsilon}\Big|_{\varepsilon=0}=d_0'(0)+\sum_{n\geq 1}[\tilde d_{2,n}(0)+\mathcal A_n]. 
$$
This completes the proof of the theorem.
\end{proof}
\begin{remark}
We note that arguments similar to those in the proof of Theorem~\ref{DiffThm} have recently been used to discuss the differentiability of the variance in the quenched central limit theorem for random intermittent systems (see~\cite[Theorem 9]{DL}). However, there are differences between these two results.  More precisely, in the context of~\cite{DL} the assumptions of Theorem~\ref{DiffThm} are not satisfied with $\mathcal B_w=C^0$, which means that it is necessary to combine the approach carried out in this paper with the so-called cone techniques. 
On the other hand, the class of random dynamics studied in~\cite{DL} exhibits uniform decay of correlations, meaning that some  arguments developed in this paper can be simplified. 

\end{remark}

\section{Application to some classes of expanding maps}\label{Expanding}
In this section we will present general strategies to verify all the conditions of Theorem~\ref{LRThm} individually (for random expanding maps). This is done because we think that most of these conditions are interesting on their own. In the next section we will present two applications of these general estimates. The first is to quite general one-dimensional maps (Theorem \ref{1D EG}) and the second is for a particular example of a higher-dimensional expanding maps on the torus (see also Remark \ref{Rem Expand}). The proof of Theorem \ref{1D EG} appears at the end of this section (see Section \ref{pf1d}) after the more general analysis. The proof of the results for the higher dimensional example requires minor modifications which  are left for the reader.

\subsection{Two concrete examples}\label{SummSec}

The first class is one dimensional.  Let  $\mathbf T \colon \Omega \to C^5(I\times \mathbb T, \mathbb T)$, where $I\subset (-1, 1)$ is an open interval containing $0$, and where $\mathbb T$ denotes the unit circle. Set   $T_{\om, \varepsilon}=\mathbf T(\om)(\varepsilon, \cdot)$.
 We assume that there are random variables $\mathcal A(\om)>1$ and $\gamma_\om>1$ such  that 
 $$
\|\mathbf T(\om)\|_{C^5( I\times \mathbb T, \mathbb T)}\leq \mathcal A(\om)
 $$
 and 
 $$
\min_{x\in\mathbb T}|T_{\om,\varepsilon}'(x)|\geq \gamma_\om.
 $$
Like in Appendix~\ref{App} we consider here the following type of mixing assumptions on the base map $\sigma$.

Let $(X_j)_{j\in\mathbb Z}$ be a stationary ergodic sequence of random variables defined on a common probability space $(\Omega_0,\mathcal F_0,\mathbb P_0)$. For every $k,k_1,k_2\in\mathbb Z$ such that $k_1\leq k_2$ we define 
$$
\mathcal F_{-\infty,k}=\mathcal F\{X_j: j\leq k\},
\mathcal F_{k_1,k_2}=\mathcal F\{X_j:  k_1\leq j\leq k_2\}\,\text{ and }\, 
\mathcal F_{k,\infty}=\mathcal F\{X_j: j\geq k\}.
$$
Here $\mathcal F\{X_j: j\in A\}$ denotes the $\sigma$-algebra generated by the family of random variables $\{X_j: j\in A\}$, and $A\subset\mathbb Z$ is a set.
We suppose that $(\Omega,\mathcal F,\mathbb P,\sigma)$ is the left shift system formed by   $(X_j)_{j\in\mathbb Z}$. Namely, $\Omega=\Omega_0^\mathbb Z$, $\mathcal F$ is appropriate product $\sigma$-algebra, $\mathbb P$ is the unique measure such that for every finite collection of sets $A_i\in \mathcal F_0, |i|\leq m$ the  corresponding cylinder set
$A=\{(\om_k)_{k=-\infty}^\infty: \om_i\in A_i, |i|\leq m\}$  satisfies $\mathbb P(A)=\mathbb P_0(X_i\in A_i; |i|\leq m)$. Moreover, for $\om=(\om_k)_{k\in\mathbb Z}$ we have  $\sigma(\om)=(\om_{k+1})_{k\in\mathbb Z}$. This means that, when considered as a random point,  $(\om_j)_{j\in\mathbb Z}$ has the same distribution as the random path $(X_j)_{j\in\mathbb Z}$. 
 Recall that the upper   $\psi$-mixing coefficients  of the process $(X_j)_{j\in\mathbb Z}$ are given by 
$$
\psi_U(n)=\sup_{k\in\mathbb Z}\sup\left\{\frac{\mathbb P_0(A\cap B)}
{\mathbb P_0(A)\mathbb P_0(B)}-1:A\in\mathcal F_{-\infty,k}, 
B\in\mathcal F_{k+n,\infty}, \mathbb P_0(A)\mathbb P_0(B)>0\right\}.
$$
When $X_i$ are i.i.d then $\psi_U(n)=0$ for all $n$. In general,  $\psi_U(n)$ measures the amount of dependence after $n$ steps from above. 
We assume\footnote{The proof will actually only require that $\psi(n_0)<\delta_0$ for a some $n_0\in\mathbb N$ and a  sufficiently small $\delta$ which depends only on the distribution of the random variables $\gamma_\om$ and $\mathcal A(\om)$, but the goal in this section is not to consider the most general cases.} here that
\begin{equation}\label{psi cond}
\limsup_{n\to\infty}\psi_U(n)=0.    
\end{equation}
 In order to simplify the presentation of our result we will also assume that $\om\mapsto \gamma_\om$, $\omega\mapsto \mathcal A(\om)$ and $\om\mapsto T_{\om,\varepsilon}$ depend only on  $\om_0$, where $\om=(\om_j)_{j\in\mathbb Z}$. The case when $\gamma_\om\geq 1$ but $\mathbb P(\gamma_\om=1)<1$ and when $\gamma_\om$ can only be approximated by functions of finitely many coordinates $\om_j$ can also be considered. Additionally, the case of  $\alpha$-mixing sequences with  $\alpha(n)=O(n^{-r})$ for $r$ large enough can be considered, as well (see \eqref{alpha def} for the definition of $\alpha(n)$).  We refer to Assumptions \ref{Ass A u} and \ref{Poly Ass} in Section \ref{MixMom} for the exact more general conditions under which such extensions would hold.
 
\begin{theorem}\label{1D EG}
Suppose $\|\cdot\|_w=\|\cdot\|_{\infty}, \|\cdot\|_s=\|\cdot\|_{C^1}$ and $\|\cdot\|_{ss}=\|\cdot\|_{C^3}$. 
Let $\bar p\geq 4$ and suppose that $\om\mapsto\mathcal A(\om)\in L^{\bar p}(\Omega,\mathcal F,\mathbb P)$. Then all the conditions of Theorem \ref{LRThm} and Theorem \ref{annealed}  hold with any choice of  $p_0<\frac12\sqrt{\bar p}$,  $\beta>1$, $p_1<\frac12\sqrt{\bar p}$, $p_2<\frac14\bar p$,  $p_3<\frac{1}{12}\bar p$, $p_4<\frac1{82}\sqrt{\bar p}$ and $p_5<\frac1{8}\sqrt{\bar p}$ (where in Theorem \ref{annealed} we choose $p_6$ appropriately).

Moreover, condition 
 \eqref{TOB} holds with $A_r(\om)\in L^{p_4}$ for $p_4$ as above 
 and 
condition \eqref{C om} holds with $C(\om)\in L^{p_3}$ with $p_3$ as above. Thus, if $\bar p$ is large enough then all the conditions of Theorem \ref{DiffThm} hold true.
\end{theorem}
The proof of Theorem \ref{1D EG} is a combination of the more general estimates in the following sections. Since it heavily relies on these results for readers' convenience  the proof of Theorem \ref{1D EG} is postponed to Section \ref{pf1d}.

We note that we did not attempt to optimize the choice of $p_i$. Probably by taking a careful look at the proof (namely the estimates in the following sections) larger $p_i$'s can be provided, but the purpose of the above theorem is to demonstrate the type of results that can be obtained by our general analysis in the one dimensional case.

 \begin{remark}\label{Rem Expand}
In fact, the only place where the one dimensionality will be used in the proof of Theorem \ref{1D EG} is in Section \ref{Sec 3.2}, where apriori upper bounds of the form  $$
\sup_{n\in\N,\varepsilon\in I}\|\mathcal L_{\sigma^{-n}\om,\varepsilon}^n\textbf{1}\|_\infty\leq A_0(\om)
$$
are obtained (i.e. the maximal amount of  volume growth after $n$ steps is bounded by $A_0(\sigma^n\om)$). Thus, when such estimates hold with $A_0\in L^{p'}(\Omega,\mathcal F,\mathbb P)$ for $p'$ large enough, Theorem \ref{1D EG} holds without restrictions on the dimension. Below we will provide an explicit example of such systems, and we believe that  other examples could be given.
 \end{remark}

Let us discuss some classes of multidimensional examples with piecewise sufficiently smooth dependence  on $\varepsilon$. We assume here that $T_\om=T_{\omega, 0}$ is a piecewise injective map on the torus $M=\mathbb T^d, d\in\N$ such that 
\eqref{Pair1.0} holds with all pairs of points $x,x'$. To have a more concrete example we suppose that there is  partition $\mathcal I_\om=\{I_{\om,1},\ldots ,I_{\om,D_\om}\}$ of $\mathbb T^d$ into rectangles such that each restriction $T_{\om, i}:=T_\om|_{I_{\om,i}}$  expands distances by at least $\gamma_\om>1$ and $T_\om(I_{\om,i})=M$. We also assume that $D_\om$ is measurable. Now we construct the maps $T_{\om, \varepsilon}$ by perturbing each one of $T_{\om, i}$ without changing the image. Let us denote the resulting map on $I_{\om,i}$ by $T_{\om,\varepsilon,i}$
 Next, instead of assuming that $(x,\varepsilon)\to T_{\om,\varepsilon}(x)$ is of class $C^5$ we suppose that each one of the maps $(x,\varepsilon)\to T_{\om,\varepsilon,i}(x)$ are of class $C^5$, and let $\mathcal A(\om)>1$ be a random variable satisfying 
$$
\max_{i}\|T_{\om,\cdot,i}(\cdot)\|_{C^5(I_{\om,i}\times I)}\leq \mathcal A(\om).
$$
Then up to minor modifications Theorem \ref{1D EG}
 still holds for the above random maps. 
The most significant difference in the proof is that 
since $T_\om(I_{\om,i})=M$ we can  apply Theorem \ref{OSC1} with $m(\om)=0$ and all $n$ without the apriori estimates like the ones discussed in Remark \ref{rem small n}.
 This yields \eqref{spectral gap}, which implies appropriate estimates on $\|\mathcal L_{\sigma^{-n}\om,\varepsilon}\textbf{1}\|_\infty$ (see Lemma \ref{Lemma bdd}) which in the one dimensional case were needed to prove \eqref{spectral gap} for small $n$'s. The rest of the modification to the proof are minor, for instance instead of considering the function $\phi_{\om,\varepsilon}=\ln J(T_{\om,\varepsilon})$ (where $J(T_{\om, \varepsilon})$ denotes the Jacobian of $T_{\omega, \varepsilon}$)  we only need to consider $\phi_{\om,\varepsilon,i}=\ln J(T_{\om,\varepsilon,i})$  which are $C^5$ in both $x$ and $\varepsilon$, as opposed to $\phi_{\om,\varepsilon}$. We decided not to include a precise statement in order not to overload the paper and to avoid repetitions.

\subsection{A general class of maps satisfying \eqref{spectral gap}}\label{3.1}
Let $(M,d)$ be a compact Riemannian manifold, normalized in size such that $\text{diam}(M)\leq 1$. Let $T_\varepsilon:\Omega\times M\to \Omega\times M$ be  a family of measurable maps, where $\varepsilon\in I\subset (-1,1)$ and $I$  is an open interval containing $0$. Denote $T_{\omega, \varepsilon}=T_\varepsilon(\omega,\cdot)$. We assume that 
 there are random variables $\xi_\omega \in(0,1]$ and
$\gamma_\omega\geq 1$  such that, $\mathbb P$-a.s.
for every $x,x'\in M$ with $d(x,x')\leq \xi_{\sigma\omega}$ we can write 
\begin{equation}\label{Pair1.0}
T_{\omega, \varepsilon}^{-1}(\{x\})=\{y_i=y_{\varepsilon,i,\omega}(x): i<k\}\,\,\text{ and }\,\,T_{\omega, \varepsilon}^{-1}(\{x'\})=\{y_i'=y_{\varepsilon,i,\omega}(x'): i<k\}
\end{equation}
and we have
\begin{equation}\label{Pair2.0}
d(y_i,y_i')\leq (\gamma_\omega)^{-1}d(x,x')
\end{equation}
for all  $1\leq i<k=k(\varepsilon,\omega,x)$ (where either $k\in\N$ or $k=\infty$). To simplify\footnote{Note that we can always decrease $\xi_\omega$ and force it to be smaller than $1$, but when we can take $\xi_\omega=1$ then our setup allows maps with infinite degrees. Note also that since the diameter of $X$ does not exceed $1$ then we do not need to consider the case $\xi_\om>1$ since then the condition already holds with $\xi_\om=1$.} the presentation and proofs we suppose that either $\xi_\omega<1$ for $\mathbb P$-a.e. $\omega \in \Omega$ or $\xi_\omega=1$ for $\mathbb P$-a.e.  $\omega\in \Omega$. 
In the first case,  we also assume that there is a finite random variable $D_\omega\geq 1$ such that for every $\varepsilon\in I$,
\begin{equation}\label{Domega}
\deg (T_{\omega, \varepsilon})=\sup\{|T_{\omega, \varepsilon}^{-1}(\{x\})|: x\in M\}\leq D_\omega.
\end{equation}
In particular, in this case $k(\varepsilon,\omega,x)$ introduced above is always finite. When $\xi_{\omega}=1$ but $\deg (T_{\omega, \varepsilon})=\infty$,  we also assume that there is a random variable $D_\omega\geq 1$ such that 
\[\|\mathcal L_{\omega, \varepsilon} \textbf{1}\|_\infty e^{-\|\phi_{\omega, \varepsilon}\|_\infty}
\leq D_\omega, \] where $\mathcal L_{\omega, \varepsilon}$ is the operator associated to $T_{\omega, \varepsilon}$. We recall that 
\[\mathcal L_{\omega, \varepsilon} g(x)=\sum_{y: T_{\omega, \varepsilon} y=x}e^{\phi_{\omega, \varepsilon}(y)}g(y),\] where $g\colon M\to\mathbb R$ and
$$
\phi_{\omega, \varepsilon}=-\ln J(T_{\omega, \varepsilon})
$$
where $J(T)$ is the Jacobian of a map $T$.
Then in both cases we have  
\begin{equation}\label{D need}
 \|\mathcal L_{\omega, \varepsilon} \textbf{1}\|_\infty\leq e^{\|\phi_{\omega, \varepsilon}\|_\infty}D_\omega. 
\end{equation}


Next, when $\xi_\omega<1$ we suppose that there is a positive integer valued random variable $m(\omega)$ with the property that $\mathbb P$-a.s.
$$
T_{\omega, \varepsilon}^{m(\omega)}(B(x,\xi_\omega))=M,
$$
for every $x\in M$ and $\varepsilon \in I$,  where $T_{\om,\varepsilon}^n=\ T_{\sigma^{n-1}\om,\varepsilon}\circ\ldots\circ T_{\sigma\om,\varepsilon}\circ T_{\om,\varepsilon}$ for $n\in\N$ and $\om\in\Omega$
and $B(x,\xi)$ denotes the ball of radius $\xi$ around $x$ in $M$. Notice that since the maps  $T_{\omega, \varepsilon}$ are surjective, it follows that 
\begin{equation}\label{Cover}
T_{\omega, \varepsilon}^{n}(B(x,\xi_\omega))=M,
\end{equation}
for all $n\geq m(\omega)$. Henceforth, when $\xi_\omega=1$ we  set $m(\omega)=0$.

We also assume here that there exists a random variable $E(\omega)\in L^{e_1}(\Omega,\mathcal F,\mathbb P)$ such that for all $\varepsilon \in I$,
\begin{equation}\label{Con1}
\|\mathcal L_{\omega, \varepsilon} \textbf{1}\|_\infty\leq E(\omega). 
\end{equation}
Note that 
\begin{equation}\label{E bound}
\|\mathcal L_{\omega, \varepsilon} \textbf{1}\|_\infty\leq\deg(T_{\omega, \varepsilon})\|1/J(T_{\omega, \varepsilon})\|_\infty.    
\end{equation}
Thus, condition \eqref{Con1} holds if $\deg(T_{\omega, \varepsilon})\leq D_\omega$
and $J(T_{\omega, \varepsilon})\geq c_\omega^{-1}$ for some random variables $D_\omega,c_\omega>0$ such that $\omega\mapsto c_\omega D_\omega\in L^{e_1}(\Omega,\mathcal F,\mathbb P)$.

Let us also assume that there is a random variable $B_\omega>0$ such that 
\begin{equation}\label{B need}
\|\phi_{\omega, \varepsilon}\|_{C^1}\leq B_\omega.   
\end{equation}
 Moreover, suppose  that there is a random variable $N(\omega)>0$ such that for $\mathbb P$-a.e. $\omega \in \Omega$ and all $\epsilon \in I$ we have  that 
\begin{equation}\label{N need}
   \|DT_{\omega, \varepsilon}\|_\infty\leq N(\omega). 
\end{equation}
Let $\|\cdot\|_w=\|\cdot\|_\infty$ (sup norm) and $\|\cdot\|_{s}=\|\cdot\|_{C^1}$.
Then, \eqref{spectral gap}
holds 
when $(\Omega,\mathcal F,\mathbb P,\sigma)$ has a sufficient amount of mixing and the random variables $B_\omega, N(\omega), D_\omega$ and $m(\omega)$ satisfy appropriate moment and approximation conditions; see Appendix~\ref{App}. 

By applying \cite[Lemma 4.6]{YH 23} and Lemma~\ref{B lemma} in the circumstances of Theorem \ref{OSC1} (see Appendix~\ref{App}), 
there exists a random variable $E_\omega\in L^{q_0}(\Omega,\mathcal F,\mathbb P), E_\om\geq 1$ such that for $\mathbb P$-a.e. $\omega \in \Omega$ and all $n\geq 1$ we have
\begin{equation}\label{gam dec}
\max\left(\prod_{j=0}^{n-1}\gamma_{\sigma^j\omega}^{-1},\prod_{j=1}^{n-1}\gamma_{\sigma^{-j}\omega}^{-1}\right)\leq  E_\omega n^{-a_0}.    
\end{equation}
Here $a_0$ is as in the statement of Theorem~\ref{OSC1}
and $q_0$ is either the number $u$ from Assumption \ref{Ass A u} (i) or Assumption \ref{Ass A u} (ii), or $q_0$ can be taken arbitrarily large under  Assumption \ref{Ass A u} (iii).
\begin{remark}\label{R17}
We stress that for all $r>1$ the assumptions in Theorem \ref{OSC1} provide a set of easy to verify  conditions which guarantee that $q_0,a_0>r$. In what follows we will formulate our conditions in terms of $a_0$ and $q_0$. In the proof of Theorem \ref{1D EG} (Section \ref{pf1d}) we will see how to choose $r$ in the circumstances of that theorem.
\end{remark}

We refer to \cite[Section 3]{YH 23} for a variety of concrete examples of maps satisfying the above conditions. For reader's convenience let us describe the class of examples in \cite[Section 3.3]{YH 23}, which are  higher dimensional versions of the maps considered in Theorem \ref{1D EG}.
Here we assume that there is a random variable $\gamma_\omega\geq 1$ such that $\mathbb P$-a.e. $\omega \in \Omega$ and every $\varepsilon\in I$ we have 
$$
\gamma_\omega\leq \|(DT_{\omega, \varepsilon})^{-1}\|_\infty^{-1}.
$$
Set 
\begin{equation}\label{Z def}
Z_\omega=\sum_{j=1}^\infty\prod_{i=1}^{j}\gamma_{\sigma^{-i}\omega}^{-1}.   
\end{equation}
By \eqref{gam dec}, provided that $a_0>1$, we have 
$$
Z_\om\leq E_\om\sum_{j=1}^\infty j^{-a_0}\leq C_{a_0}E_\om,
$$
where $C_{a_0}>0$ depends only on $a_0$. Hence,
 $\omega\mapsto Z_\om\in L^{q_0}(\Omega,\mathcal F,\mathbb P)$.
Next (see \cite[Section 3.3]{YH 23}), we can take $\xi_\omega=C\min(1,Z_\omega^{-1})$, where $C=\frac12\min(1,\rho_M)$ and $\rho_M$ is the injectivity radius of $M$. Moreover,
we can take
 \begin{equation}\label{D def}
D_\omega=C_0\left(N(\omega)Z_\omega\right)^{\text{dim }M},   
 \end{equation}
 for some constant $C_0>0$ where $N(\om)$ satisfies $\sup_{\varepsilon\in I}\|DT_{\om,\varepsilon}\|_{\infty}\leq N(\om)$. Note that if $N(\om)\in L^{q_1}(\Omega,\mathcal F,\mathbb P)$ then $D_\om\in L^{d}(\Omega,\mathcal F,\mathbb P)$ with $\frac{1}{d}=\text{dim }M(\frac{1}{q_0}+\frac{1}{q_1})$.
 
Furthermore, we can choose
\begin{equation}\label{m def}
m(\omega)=\min\left\{n: \xi_\omega^{-1}\prod_{j=0}^{n-1}\gamma_{\sigma^j\omega}^{-1}\leq R\right\}, 
\end{equation}
for some constant $R>0$. Using \eqref{m def}, in \cite[Lemma 3.10]{YH 23} and  \cite[Lemma 3.11]{YH 23} we showed that all the requirements on $m(\cdot)$ in Assumptions \ref{Inner aprpox} and \ref{Poly Ass} in Appendix~\ref{App} are satisfied.


In contrast with \cite{YH 23} we will also need the following condition (c.f. Remark \ref{Rem Expand}).
\begin{assumption}\label{Add Ass}
Either $\xi_\om=1$ (so $m(\omega)=0$) or
there exists 
$c(\cdot)\in L^p(\Omega,\mathcal F,\mathbb P), p>0$ such  that
for $\mathbb P$-a.e. $\om \in \Omega$, all $\varepsilon \in I$ and $n \in \N$ we have
\begin{equation}\label{1 bound}
\|\mathcal L_{\sigma^{-n}\omega, \varepsilon}^n\mathbf {1}\|_{\infty}\leq c(\omega).   
\end{equation}

\end{assumption}
Note that the  Assumption~\ref{Add Ass} does not appear in \cite{YH 23}. The reason is  that in~\cite{YH 23},  instead of transfer operators $\mathcal L_{\om, \varepsilon}$  we considered the \emph{normalized} transfer operators $L_{\om, \varepsilon}$ given by $L_{\om, \varepsilon}(g)=\L_{\om, \varepsilon}(gh_{\om, \varepsilon})/h_{\sigma \om, \varepsilon}$, which satisfy $L_{\om, \varepsilon}\textbf{1}=\textbf{1}$. Thus, \eqref{1 bound} trivially holds with $c(\om)=1$ if we replace $\mathcal L_{\sigma^{-n}\omega, \varepsilon}^n$ with $L_{\sigma^{-n}\om, \epsilon}^n$. When proving  limit theorems, it is sufficient to deal with normalized transfer operators. However, when studying linear response, it is necessary to deal with transfer operators $\L_{\omega, \epsilon}$ as the family $(h_{\omega, \epsilon})_{\omega \in \Omega}$ is precisely a random fixed point of a cocycle $(\L_{\omega, \epsilon})_{\omega \in \Omega}$.


  The condition $\xi_\omega=1$ means that we can pair the inverse images of every two points. This is the case in the multidimensional example discussed after Remark \ref{Rem Expand}.
We will verify condition \eqref{1 bound} in the one-dimensional case in Section \ref{Sec 3.2}.

Finally, in Appendix \ref{App}, for every $\beta,p_1>1$ we will show 
there are sets of mixing, approximation and moment conditions  which are sufficient for  \eqref{spectral gap} (and for \eqref{gam dec} with appropriate $q_0\geq p_1$ and $a_0\geq \beta$) for maps satisfying the above conditions. 
In fact, what follows is that 
\begin{equation}\label{spectral gap1}
\|\L_{\omega, \epsilon}^n h-\psi(h)h_{\sigma^n\om, \epsilon}\|_\infty \le C_1(\omega)n^{-\beta}\|h \|_{C^1}.
\end{equation}
where $\psi(h)=\int_M h \, d m$ and $m$ is the normalized volume measure on $M$.
By taking $h$ with $\psi(h)=0$ we get \eqref{spectral gap}. 
In the following section we will verify the rest of the conditions of Theorem \ref{LRThm} under additional assumptions, and in Section \ref{pf1d} we will prove Theorem \ref{1D EG} using this more general analysis.

\subsection{Upper bounds on $\|\mathcal L_{\sigma^{-n}\omega, \varepsilon}^n\textbf{1}\|_\infty$ in the one dimensional case}\label{Sec 3.2}
Here we provide sufficient conditions for \eqref{1 bound}, which we recall in the case when $\xi_\om<1$ is needed for \eqref{spectral gap}. We will also need \eqref{1 bound} to verify the rest of the conditions of Theorem \ref{LRThm}.

We suppose that $M=[0,1]$ (or $M=S^1$), that $T_{\om, \varepsilon}$ are piecewise expanding, and that each monotonicity interval  can be extended to a $C^2$ function. Henceforth, $T_{\om, \varepsilon}'$ and $T_{\om, \varepsilon}''$ will be interpreted as the first and second derivatives of these extensions on the appropriate intervals.

Next, we assume that there is a random variable $q(\om)$ such that for $\mathbb P$ a.e. $\om  \in \Omega$ and all $\varepsilon \in I$ we have 
\begin{equation}\label{q bound}
\|T_{\om, \varepsilon}''\|_\infty\leq q(\om).    
\end{equation}
Now, since  $|T'_{\om, \varepsilon}|\geq \gamma_\om$ we have 
$$
\left\|\frac{T_{\om, \varepsilon}''}{(T_{\om, \varepsilon}')^2}\right\|_\infty\leq c(\omega)
$$
where
$
c(\om)=\gamma_\om^{-2}q(\om).
$
This is a random version of the so-called Adler condition. The above condition means that 
 that  for every inverse branch $z$ of $T_{\om, \varepsilon}$ we have 
\begin{equation}\label{z inv}
  |z''|\leq c(\omega)|z'|.  
\end{equation}
Indeed, this readily follows from $z''=-\frac{(T_{\om, \varepsilon}''\circ z)\cdot z'}{(T_{\om, \varepsilon}'\circ z)^2}$.

The main result in this section is 
\begin{proposition}\label{Sec3.2 Main}
If $c(\cdot)\in L^p$ with $a_0>1/p+1$ then for $\mathbb P$-a.e. $\om \in \Omega$ and all $n\in\N$ and $\varepsilon\in I$ we have 
$$
\|\mathcal L_{\sigma^{-n}\omega, \varepsilon}^n\mathbf {1}\|_\infty\leq 1+c_1(\om),
$$
with $c_1(\om)$  is given by Lemma~\ref{FL}.
\end{proposition}

Before proving Proposition \ref{Sec3.2 Main} we need the following result.
\begin{lemma}\label{FL}
If $c(\cdot)\in L^p (\Omega, \mathcal F, \mathbb P)$ with $a_0>1/p+1$, then there is a constant $C>0$ and a random variable $R(\cdot) \in L^p(\Omega, \mathcal F, \mathbb P)$ such that for every inverse branch $y$ of $T_{\sigma^{-n}\omega, \varepsilon}^n$ we have
$$
\|y''/y'\|_\infty\leq C R(\omega)E_\omega=:c_1(\omega).
$$
\end{lemma}
\begin{proof}
Let us first fix some inverse branch $y$  of $T_{\sigma^{-n} \om, \varepsilon}^n$ and write it as a composition of inverse branches  $z_j$ of $T_{\sigma^{-j}\omega, \varepsilon}$:
$$
y=z_n\circ z_{n-1}\circ\ldots \circ z_1.
$$
Then 
$$
y''=y'\sum_{k=1}^n\frac{F_k'}{F_k}
$$
where 
$$
F_k=z_k'\circ z_{k-1}\circ \ldots\circ z_1. 
$$
Now, using \eqref{z inv} and that $|z_j'|\leq \gamma_{\sigma^{-j}\omega}^{-1}$ we get that 
$
|F_k'/F_k|\leq c(\sigma^{-k}\omega)\prod_{j=1}^{k-1}\gamma_{\sigma^{-j}\omega}^{-1}
$
and so 
\begin{equation}\label{and so }
|y''/y'|\leq \sum_{k=1}^n c(\sigma^{-k}\omega)\prod_{j=1}^{k-1}\gamma_{\sigma^{-j}\omega}^{-1}.    
\end{equation}
Next, let $\delta>0$ be such that $a_0>1/p+\delta+1$. Then by Lemma \ref{B lemma}, we have $c(\sigma^{-k}\omega)\leq R(\omega)k^{1/p+\delta}$, with  some $R\in L^p(\Omega, \mathcal F, \mathbb P)$. Now the desired result readily follows from \eqref{gam dec} and \eqref{and so }.
\end{proof}

\begin{proof}[Proof of Proposition \ref{Sec3.2 Main}]
Let $v$ denote the usual variation on $[0,1]$. Then for differentiable functions $f$ we have that 
$$
v(f)=\int_{0}^1|f'(x)|dx.
$$
Next, let $y_i=y_{\varepsilon,\omega,i,n}$ be the inverse branches of $T_{\sigma^{-n}\omega, \varepsilon}^n$. Then
$
\mathcal L_{\sigma^{-n}\omega, \varepsilon}^n\mathbf {1}=\sum_{i}|y_i'|.
$
Thus, 
$$
\left|\left(\mathcal L_{\sigma^{-n}\omega, \varepsilon}^n\mathbf{1}\right)'\right|\leq \sum_{i}|y_i''|.
$$
Now, by Lemma \ref{FL} we have
\begin{equation}\label{Pre Adler0}
|y_i''|\leq c_1(\om)|y_i'|.
\end{equation}
Therefore
$$
\left|\left(\mathcal L_{\sigma^{-n}\omega, \varepsilon}^n\textbf{1}\right)'\right|\leq c_1(\omega)\mathcal L_{\sigma^{-n}\omega, \varepsilon}^n\textbf{1}
$$
and so, since $\int_{0}^1\mathcal L_{\sigma^{-n}\omega, \varepsilon}^n\textbf{1}\, dm=1$ we have 
$$
v(\mathcal L_{\sigma^{-n}\omega, \varepsilon}^n\textbf{1})\leq c_1(\om).
$$
Consequently, 
$$
\|\mathcal L_{\sigma^{-n}\omega, \varepsilon}^n\textbf{1}\|_\infty\leq v(\mathcal L_{\sigma^{-n}\omega, \varepsilon}^n\textbf{1})+1\leq c_1(\om)+1,
$$
where  the first inequality uses that  $\min_{x\in M}(\mathcal L_{\sigma^{-n}\omega, \varepsilon}^n\textbf{1}(x))\leq 1$ (since the average is $1$).
\end{proof}

\subsection{On the verification of conditions \eqref{weakcontr}, \eqref{top} and \eqref{AB} with appropriate norms}
\subsubsection{Verification of \eqref{top}}\label{top ver sec}
Let us first obtain some estimates in the   supremum norm.
The basic idea is that 
$$
h_{\omega,\varepsilon}=\lim_{n\to\infty}\mathcal L_{\sigma^{-n}\omega,\varepsilon}^n\textbf{1},
$$
where $\textbf{1}$ is the constant function taking the value $1$. 

When \eqref{1 bound} does not apriori hold (which in our case means that  $\xi_\om=1$), then 
in order to bound $\|\mathcal L_{\sigma^{-n}\omega,\varepsilon}^n\textbf{1}\|_\infty$ we use 
the following result. 

\begin{lemma}\label{Lemma bdd}
Under \eqref{spectral gap1} with $p_1\geq 1$,  for $\mathbb P$-a.e. $\om \in \Omega$ and all $n\geq 1$  and $\varepsilon\in I$ we have     $$
\|\mathcal L_{\sigma^{-n}\omega,\varepsilon}^n\mathbf {1}\|_\infty\leq B_0(\om)+E(\sigma^{-1}\om)
$$
where $B_0(\cdot)\in L^{p_1}(\Omega,\mathcal F,\mathbb P)$.
\end{lemma}
\begin{proof}
Since $\psi(\textbf{1})=\textbf{1}$, by \eqref{spectral gap1},
$$
\|\mathcal L_{\sigma^{-n}\omega,\varepsilon}^n\textbf{1}-h_{\om,\varepsilon}\|_\infty
\leq C_1(\sigma^{-n}\om)n^{-\beta}.
$$
Next, using that $\om\mapsto C_1(\om)\in L^{p_1}(\Omega, \mathcal F, \mathbb P)$, 
by Lemma~\ref{B lemma} for every $\delta>0$ there is a random variable $Q(\om)\in L^{p_1}(\Omega, \mathcal F, \mathbb P)$ such that 
$$
C_1(\sigma^{-n}\om)\leq Q(\om)n^{1/p_1+\delta}.
$$
Now, since $1/p_1\leq1<\beta$,  by taking $\delta$ small enough we see that 
\begin{equation}\label{tzw}
\|\mathcal L_{\sigma^{-n}\omega,\varepsilon}^n\textbf{1}-h_{\om,\varepsilon}\|_\infty\leq Q(\om).
\end{equation}
Thus, 
$$
\|\mathcal L_{\sigma^{-n}\omega,\varepsilon}^n\textbf{1}\|_\infty\leq Q(\om)+\|h_{\om,\varepsilon}\|_\infty.
$$
On the other hand, by taking $n=1$ in~\eqref{tzw} and using that $\|\mathcal L_{\sigma^{-1}\omega,\varepsilon}\mathbf {1}\|_\infty \leq E(\sigma^{-1}\om)$, we have that 
$$
\|h_{\om,\varepsilon}\|_\infty\leq Q(\om)+E(\sigma^{-1}\om)
$$
and so we can take $B_0(\om)=2Q(\om)$.
\end{proof}

Next, let us provide some sufficient conditions for \eqref{top} to hold in the $C^3$ norm. Recall that $h_{\omega,\varepsilon}$ is a uniform limit of $\mathcal L_{\sigma^{-n}\omega,\varepsilon}^n\textbf{1}$.  Since the unit ball in the $C^4$ norm is relatively compact in $C^3$, in order to show that $h_{\omega,\varepsilon}$ belongs to $C^3$ and that 
$$
\|h_{\omega,\varepsilon}\|_{C^3}\leq C_4(\om)
$$
for some random variable $C_4(\omega)$ in $L^{p_4}$, it is enough to show that 
\begin{equation}\label{2nd}
 \|\L_{\sigma^{-n}\omega,\varepsilon}^n\textbf{1}\|_{C^4}\leq C_4(\omega) .  \end{equation}
Indeed,  since the ball of radius $a(\omega)$ in $C^4$ is relatively compact in $C^3$ we get that the uniform limit $h_\omega$ must be a limit  in $C^3$ and it must belong to that ball. In what follows we will prove \eqref{2nd} with $C_4(\om)$ given in Remark \ref{Rem C4}. In fact, we will prove even more general estimates that will be used in the sequel to verify some of the other conditions of our main results.

Next, in order to prove \eqref{2nd} we consider the following condition: there exists a random variable $c(\omega)\in L^p (\Omega, \mathcal F, \mathbb P)$, $p>1$ such that for every inverse branch $y_{\omega,\varepsilon}$ of $T_{\omega, \varepsilon}$ we have 
\begin{equation}\label{Pre Adler}
\max(\|D^2 y_{\omega,\varepsilon}\|_\infty, \|D^3 y_{\omega,\varepsilon}\|_\infty, \|D^4 y_{\omega,\varepsilon}\|_\infty)\leq c(\omega).
\end{equation}
Let $\mathcal A(\om)$ be such that for $\mathbb P$-a.e. $\om \in \Omega$  we have $\|D(T_{\om, \varepsilon})\|_{C^4}\leq \mathcal A(\om)$ (for all $\varepsilon\in I$). Using Lemma \ref{Der lemma} in Appendix \ref{App A} and that $\|Dy_{\om,\varepsilon}\|_\infty\leq\gamma_\om^{-1}$ we get the following result.
\begin{lemma}\label{ver der}
Condition \eqref{Pre Adler} holds if $\omega\mapsto c_i(\om)\in L^p(\Omega,\mathcal F,\mathbb P), i=1,2,3$, where
\begin{equation}\label{1D1}
c_1(\om):=\gamma_\omega^{-1}\left(1+\gamma_\omega^{-2}\mathcal A(\om)
\right),
\end{equation}
\begin{equation}\label{1D2}
c_2(\om):=\gamma_\om^{-1}\left(1
+\mathcal A(\om)\gamma_{\om}^{-3}+
3\mathcal A(\om)c_1(\om)\gamma_\om^{-1}\right)
\end{equation}
and, 
with $g_\om=1+\gamma_\om^{-1}$,
\begin{equation}\label{1D3}
c_3(\om)=\gamma_\om^{-1}\left(1+
\mathcal A(\om)g_\om[(c_1(\om))^2+c_2(\om)]+
\mathcal A(\om)\gamma_\om^{-2}c_1(\om)+
\mathcal A(\om)\gamma_\om^{-4}\right).
\end{equation}
\end{lemma}

We will also need  the following two results.

\begin{lemma}\label{Lemma 4}
Suppose that $\gamma_\omega\geq1$ and that $\gamma_\omega\in L^q$ for some $q>1$. Let   \eqref{Pre Adler} hold and
assume also that $a_0>1+\frac{1}{p}+\frac1{q}$, where $a_0$ comes from \eqref{gam dec} and $p$ comes from \eqref{Pre Adler}.
Then, for every $\delta>0$ (small enough) there exists a random variable $C_\omega\geq 1$ such that
for all $n$, $\varepsilon\in I$ and every inverse branch $y$ of $T_{\sigma^{-n}\omega,\varepsilon}^n$ we have 
\begin{equation}\label{11}
\|D^2y\|_\infty\leq C_\omega n^{-\eta}    
\end{equation}
where $\eta=a_0-1-1/s-\delta$,
$s$ is given by $1/s=1/p+1/q$,  $\omega\mapsto C_\omega\in L^{t}(\Omega,\mathcal F,\mathbb P)$, and $t$ is given by $\frac{1}{t}=\frac{1}{q_0}+\frac{1}{p}+\frac{1}{q}$ ($a_0$ and $q_0$ come from \eqref{gam dec}). 

Moreover, if $a_0>2+1/s$ then
\begin{equation}\label{12}
\|D^3y\|_\infty\leq A_\om n^{-\zeta}    
\end{equation}
where $\zeta=a_0-2-1/s-\delta$, 
and $A_\om\geq 1$ is a random variable such that $\om\mapsto L^u(\Omega, \mathcal F,\mathbb P)$, where $u$ is given by $\frac{1}u=\frac{1}{q_0}+\frac{2}{p}+\frac{2}{q}$.

Furthermore, if $2a_0>3-2/s$ then
\begin{equation}\label{13}
\|D^4y\|_{\infty}\leq R_\om n^{-\kappa}
\end{equation}
where $\kappa=2a_0-3-3/s-\delta$ and  $R_\om\geq1$ is 
a random variable
such that $\omega\mapsto R_\om\in L^{u_1}(\Omega,\mathcal F,\mathbb P)$, with $u_1$  defined by $\frac1{u_1}=\frac{4}{q_0}+\frac{3}{s}$. 
\end{lemma}

Without the assumption that $\gamma_\omega\geq1$ we have the following slightly weaker conclusion.
\begin{lemma}\label{Lemma 5} 
 Suppose  that $\gamma_\omega\in L^q$ for some $q>1$. Let  \eqref{Pre Adler} hold and assume also that $a_0>1+\frac{1}{p}+\frac{1}{q_0}+\frac1{q}$, where $a_0$ and $q_0$ come from \eqref{gam dec} and $p$ comes from \eqref{Pre Adler}.
Then for every $\delta>0$ (small enough) there exists a random variable $C_\omega\geq 1$ such that
for all $n$, $\varepsilon\in I$ and every inverse branch $y$ of $T_{\sigma^{-n}\omega,\varepsilon}^n$ we have 
$$
\|D^2y\|_\infty\leq C_\omega n^{-\eta}
$$
where $\eta=a_0-1-1/s-\delta$,
$s$ is given by $\frac{1}s=\frac{1}{p}+\frac{1}{q}+\frac{1}{q_0}$,  $\omega\mapsto C_\omega\in L^{t}(\Omega,\mathcal F,\mathbb P)$, and $t$ is given by $\frac{1}{t}=\frac{2}{q_0}+\frac{1}{p}+\frac{1}{q}$. 

Moreover, if also $a_0$ from \eqref{gam dec} is larger than $1/2$ and
$$
\zeta:=\min(2a_0-1/s-1/p-2/q_0-1-4\delta,3a_0-3/q_0-2/p-4\delta)>0
$$
then
$$
\|D^3y\|_\infty\leq A_\om n^{-\zeta}
$$
for a random variable $A_\om\geq 1$ such that $\om\mapsto L^u(\Omega, \mathcal F,\mathbb P)$, where $u$ is given by $\frac{1}u=\frac{7}{q_0}+\frac{2}{p}+\frac{2}{q}$.

Furthermore, if
$
\kappa:=4a_0-3-3/s-\delta>0
$
then
$$
\|D^4y\|_{\infty}\leq R_\om n^{-\kappa}
$$
where  $R_\om\geq 1$ is such that $R_\om\in L^{u_1}(\Omega,\mathcal F,\mathbb P)$, with $u_1$ defined by $\frac{1}{u_1}=\frac{4}{q_0}+\frac{3}{s}+\frac{2}{q}$. 
\end{lemma}

The proofs of Lemmata \ref{Lemma 4} and \ref{Lemma 5} rely on \eqref{gam dec}, Lemma \ref{B lemma} and the formulas for the derivatives of order four or less of compositions of functions of the form $y=y_n\circ y_{n-1}\circ \ldots \circ y_i$, where in our case we take $y_i$ to be an inverse branch of $T_{\sigma^{-i}\om,\varepsilon}$. Since this is a general  principle we postpone the (lengthy) proofs to Appendix \ref{App A}.

Next, let us verify \eqref{2nd} under \eqref{Pre Adler}.
Let $\phi_{\omega,\varepsilon}=-\ln J(T_{\omega,\varepsilon})$. Then  
$$
\L_{\sigma^{-n}\omega,\varepsilon}^n\textbf{1}=\sum_{i}e^{(S_n^{\sigma^{-n}\omega} \phi_\varepsilon)\circ y_{i,n}}
$$
where $y_{i,n}=y_{i,n,\sigma^{-n}\omega,\varepsilon}$ are the inverse branches of $T_{\sigma^{-n}\omega,\varepsilon}^n$.
Before verifying \eqref{2nd} we need the following result. 
\begin{lemma}\label{phi lem}
Let the  conditions of either
Lemma~\ref{Lemma 4} or Lemma \ref{Lemma 5} we  be in force. Suppose that $\eta,\zeta,\kappa>1$. Assume also that  $\|\phi_{\om, \varepsilon}\|_{C^4}\leq B_4(\om)$ for some random variable $B_4(\om)\in L^d(\Omega,\mathcal F,\mathbb P)$ (for some $d>0$). Let also \eqref{Pre Adler} be in force.  

Then for every $\varepsilon$ and every inverse branch $y$ of $T_{\sigma^{-n}\om, \varepsilon}^n$ 
for $r=1,2,3,4$ we have  
$$
\|D^r(S_n^{\sigma^{-n}\omega}\phi_\varepsilon \circ y)\|_{\infty}\leq V_r(\om),
$$
where $V_i(\om)\in L^{v_i}$, $V_i\geq 1$ where
$
\frac1{v_1}=\frac{1}{d}+\frac{1}{q_0}, \frac1{v_2}=\frac{1}{d}+\frac{2}{\min(q_0,2t)},$
$$
\frac{1}{v_3}=\frac{1}{d}+\max\left(\frac{3}{q_0},\frac{1}{q_0}+\frac{1}t,\frac{1}u\right) 
$$
and 
$$
\frac{1}{v_4}=\frac1d+\max\left(\frac{4}{q_0},\frac{2}{q_0}+\frac{1}t,\frac{1}u+\frac{1}{q_0},\frac{2}{t},\frac1{u_1}\right).
$$

\end{lemma}
The proof of Lemma \ref{phi lem} also relies on a general computation of the first four derivatives of composition of two functions $\phi=\phi_{\sigma^{-j}\om, \varepsilon}$ and $y=y_j$ where $y_j$ is an inverse branch of $T_{\sigma^{-n}\om, \varepsilon}^j$. Since this is a general elementary idea the proof is included in Appendix \ref{App A}.

\begin{cor}\label{Aux Cor}
Let the  conditions of either
Lemma~\ref{Lemma 4} or Lemma \ref{Lemma 5} we  be in force.
Then, for $r \in \{1,2,3,4 \}$ we have
$$ 
\|\mathcal L_{\sigma^{-n}\om,\varepsilon}^n\|_{C^r} \leq \|\mathcal L_{\sigma^{-n}\om,\varepsilon}^n\textbf{1}\|_\infty Q_r(\om)
$$
where $Q_1(\om)\in L^{d_i}(\Omega,\mathcal F,\mathbb P)$ with $d_1$ given by $\frac1{d_1}=\frac{1}{v_1}+\frac{1}{q_0}$, and for $i\in\{2,3,4\}$, $d_i=\max\{d_1,\tilde d_2, \ldots ,\tilde d_i\}
$
where $\tilde d_i$ are given by
$$
\frac{1}{\tilde d_2}=\frac{2}{\min(v_1,2v_2)}+\frac{2}{\min(q_0, 2t)},\,\,\,\,\,\frac{1}{\tilde d_3}=\frac{6}{\min(2v_1,3v_2,6v_3)}+\frac{6}{\min(2t,3q_0,6u)}
$$
and 
$
\frac{1}{\tilde d_4}=\frac{4}{\min(v_1,2v_2,2v_3,4v_4)}+\frac{4}{\min(2t,q_0,2u,4u_1)}.
$
Here $v_i$ are as in Lemma \ref{phi lem}, $t,u,u_1$ are as in Lemmas~\ref{Lemma 4} and~\ref{Lemma 5} and $q_0$ is such that $E_\om\in L^{q_0}$.
\end{cor}
\begin{remark}\label{Rem C4}
In the circumstances of  Lemma \ref{Lemma bdd} 
we get that  for $r\in \{0, 1, 2, 3, 4\}$ and with $Q_0(\om)=1$,
\begin{equation}\label{Norm bounds}
\|\mathcal L_{\sigma^{-n}\om, \varepsilon}^n \mathbf 1 \|_{C^r}\leq Q_r(\om)(B_0(\om)+E(\sigma^{-1}\om))=:A_r(\om).  
\end{equation}
Note that $A_i(\cdot)\in L^{t_i}$, where $t_0=\min(p_1,e_1)$ and for $i>0$ we have $\frac{1}{t_i}=\frac1{d_i}+\frac{1}{\min(p_1,e_1)}$ where $e_1$ is such that $E(\om)\in L^{e_1}$.  Thus we can take $C_4(\om)=A_4(\om)$ in \eqref{2nd}.
\end{remark}
\begin{proof}[Proof of Corollary \ref{Aux Cor}]
Recall that 
$$
\mathcal L_{\sigma^{-n}\om,\varepsilon}^n g=\sum_{i}e^{(S_n^{\sigma^{-n}\om}\phi_\varepsilon)\circ y_i}g\circ y_{i},
$$
where $y_{i}=y_{i,\om,\varepsilon}$ are the inverse branches of $T_{\sigma^{-n}\om,\varepsilon}^n$. 
The corollary now follows 
by differentiating four times the function
$H(x)=e^{S(x)}G(x)$, where
 $S(x)=S_{n}^{\sigma^{-n}\om}\phi_{\varepsilon}$ and $G(x)=g(y(x))$  with $y$ being an inverse branch of $T_{\sigma^{-n}\om,\varepsilon}^n$, and using that $\|D(y)\|_\infty\leq \gamma_{\sigma^{-n}\om,n}\leq E_\om$ and Lemmas \ref{Lemma 4} and \ref{Lemma 5} to bound the second, third and fourth derivatives  of $y$. 
A tedious computation and using estimates of the form $ab\leq a^2+b^2$ and $abc\leq a^3+b^3+c^3$ for all $a,b,c\geq 0$
shows that we can take
 $Q_1=V_1(\om)+E_\om$ and for $i=2,3,4$ 
$$
Q_2(\om)=Q_1(\om)+c_4\left((V_1(\om))^2+V_2(\om)\right)(E_\om^2+C_\om),
$$
$$
Q_3(\om)=Q_2(\om)+c_4\left((V_1(\om))^3+(V_2(\om))^2+V_3(\om)\right)(C_\om^2+E_\om^3+A_\om)
$$
$$
Q_4(\om)=Q_3(\om)+c_4\left((V_1(\om))^4+(V_2(\om))^2+(V_3(\om))^2+V_4(\om)\right)(E_\om^4+C_\om^2+A_\om^2+R_\om)
$$
where $c_4>0$ is a constant.
\end{proof}

\subsubsection{Verification of \eqref{weakcontr} with $\B_w=C^0$}
Using Lemma~\ref{Lemma bdd}, we obtain that 
$$
\|\mathcal L_{\sigma^{-n}\omega,\varepsilon}^n\|_\infty= 
\|\mathcal L_{\sigma^{-n}\omega,\varepsilon}^n\textbf{1}\|_\infty\leq C_0(\omega):=B_0(\omega)+E(\sigma^{-1}\omega).
$$

\subsubsection{Verification of \eqref{AB} with $\B_s=C^1$}\label{AB ver sec} 
Let $g:M\to\mathbb R$ be such that $\|g\|_{C^1}\leq 1$.  
First, by \eqref{spectral gap1} we have
$$
\|\L_{\sigma^{-n}\omega}^{j}g-m(g)h_{\sigma^{-n+j}\omega}\|_\infty\leq C_1(\sigma^{-n}\omega)j^{-\beta},\,\, C_1(\omega)\in L^{p_1}.
$$
Taking $j=n=1$, $g=\mathbf 1$ and using~\eqref{Con1}, we see that
\[\|h_\omega\|_\infty\leq C_1(\sigma^{-1}\omega)+E(\sigma^{-1}\omega)=:V(\omega).\] 
Thus,
\begin{equation}\label{D 0}
 \|\L_{\sigma^{-n}\omega}^{j}g\|_\infty\leq C_1(\sigma^{-n}\omega)+V(\sigma^{j-n}\omega).    
\end{equation}
Next, by Corollary \ref{Aux Cor}  we have
$$
\|D(\mathcal L_{\sigma^{-n}\omega}^{j}g)\|_\infty
=\|D(\mathcal L_{\sigma^{-j}(\sigma^{j-n}\omega)}^{j}g)\|_\infty
\leq \|\mathcal L_{\sigma^{-n}\omega}^j\textbf{1}\|_\infty Q_1(\sigma^{j-n}\om).
$$
Using also \eqref{D 0} with $g=\mathbf 1$ we get that 
$$
\|D(\mathcal L_{\sigma^{-n}\omega}^{j}g)\|_\infty\leq \left(C_1(\sigma^{-n}\omega)+V(\sigma^{j-n}\omega)\right)Q_1(\sigma^{j-n}\omega).
$$
Therefore,
$$
\|\L_{\sigma^{-n}\omega}^{j}\|_{C^1}\leq 
\left(C_1(\sigma^{-n}\omega)+V(\sigma^{j-n}\omega)\right)\left(1+Q_1(\sigma^{j-n}\omega)\right).
$$
Using that $x+y\leq (1+x)(1+y)$ for all $x,y\geq 0$ we conclude that
$$
\|\L_{\sigma^{-n}\omega}^{j}\|_{C^1}\leq A(\sigma^{-n}\omega)B(\sigma^{j-n}\omega).
$$
where 
\begin{equation}\label{A def}
A(\omega)=1+C_1(\omega)   \end{equation}
and
\begin{equation}\label{B def}
B(\omega)=(1+V(\omega))(1+Q_1(\omega)).  \end{equation}
Note that $\omega\mapsto A(\omega)\in L^{p_1}$,  and 
$\omega\mapsto B(\omega)\in L^{q'}$, where $q'$ is given by $\frac1{q'}=\frac{1}{d_1}+\frac{1}{d_2}$, where $d_1$ is as in Corollary \ref{Aux Cor} and $d_2=\min(p_1,e_1)$, where $e_1$ is such that $E(\om)$ in \eqref{Con1} satisfies $\om\mapsto E(\om)\in L^{e_1}(\Omega,\mathcal F,\mathbb P)$.

\subsection{Verification of~\eqref{tn1} with $\B_w=C^0$ and $\B_s=C^1$}
For the sake of simplicity, we derive~\eqref{tn1}  in the one-dimensional case. 
Let us assume that there is a random variable 
and $q(\omega)$ is such that 
\[
d_{C^1}(T_{\omega, \epsilon}, T_\omega) \le q(\omega)|\epsilon|,
\]
and that for every point $x\in M$ the inverse branches $y_{i,\om}=y_{i, \om}(x)$ and $y_{i,\om,\varepsilon}=y_{i,\om,\varepsilon}(x)$ of $T_\om:=T_{\om,0}$ and $T_{\om,\varepsilon}$, respectively, can be paired such that for all $i$
$$
\|y_{\epsilon, i, \omega}-y_{i, \omega}\|_\infty \le q(\omega)| \epsilon|.
$$
\begin{remark}\label{Rem 5.5}
Suppose that $(\varepsilon,x)\to T_{\om, \varepsilon}(x)$ is a function of class $C^2$. Since $T_{\om, \varepsilon}\circ y_{\varepsilon,\om}(x)=x$, $y_{\varepsilon,\om}=y_{\varepsilon,i,\om}$,
if we denote $y_{\varepsilon,\om}(x)=y_\om(\varepsilon,x)$ and $T_\om (\varepsilon, x)=T_{\om, \varepsilon}(x)$,
then
 $$
(D_\varepsilon T_\om)(\varepsilon, y_\om(\varepsilon,x))+(D_x T_\om)(\varepsilon, y_\om(\varepsilon,x))(D_\varepsilon y_\om)(\varepsilon,x)=0.
 $$
and so
 \begin{equation}\label{D y eps}
(D_\varepsilon y_\om)(\varepsilon,x)=-\left((D_x T_\om)(\varepsilon,y_\om(\varepsilon,x))\right)^{-1}(D_\varepsilon T_\om)(\varepsilon,y_\om(\varepsilon,x)).     
 \end{equation}
 Therefore,
$$
\|y_{\epsilon, i, \omega}-y_{i, \omega}\|_\infty\leq |\varepsilon|\sup_{\varepsilon} (\|(D_x T_{\om, \varepsilon})^{-1}\|_\infty \cdot \|D_\varepsilon T_{\om, \varepsilon}\|_\infty)\leq \gamma_\om^{-1}|\varepsilon| \sup_{\varepsilon} \|D_\varepsilon T_{\om, \varepsilon}\|_\infty.
$$
Clearly,  we have 
$$
d_{C^1}(T_{\omega, \epsilon}, T_\omega)\leq |\varepsilon| \sup_{\varepsilon}(\|D_\varepsilon T_{\om, \varepsilon}\|_\infty+\|D_\varepsilon T_{\om, \varepsilon}'\|_\infty),
$$
and so we can take any measurable $q(\om)$ such that for $\mathbb P$ a.e. $\om$ and all $\varepsilon\in I$ we have
$$
\|D_\varepsilon T_{\om, \varepsilon}'\|_\infty+\|D_\varepsilon T_{\om, \varepsilon}\|_\infty(1+ \gamma_\om^{-1})\leq q(\om). 
$$
\end{remark}
\begin{lemma}\label{Lem 5.5}
For $\mathbb P$ a.e. $\om$ and every $g\in C^1(M)$ and $\varepsilon\in I$ we have
\[
\|(\L_{\omega, \epsilon}-\L_\omega)g\|_\infty \le |\epsilon| \bar q(\omega)\|g\|_{C^1},
\]
where $\mathcal L_\om=\mathcal L_{\om,0}$ and
\[
\bar q(\omega):=E(\omega)q(\omega) (1+\gamma_\omega^{-1}+\gamma_\omega^{-1}\|T_\omega''\|_\infty).
\]
Thus, \eqref{tn1} holds true with $C_2(\om)=\bar q(\omega)$.
\end{lemma}
\begin{proof}
Take $x \in M$ and $g\in C^1(M)$. Then,
\[
\begin{split}
\L_{\omega, \epsilon}g(x)-\L_\omega g(x) &=\sum_i \frac{g(y_{\epsilon, i, \omega})}{|T_{\omega, \epsilon}'(y_{\epsilon, i, \omega})|}-\sum_i
\frac{g(y_{i, \omega})}{|T_\omega'(y_{i, \omega})|} \\
&=\sum_i\left (\frac{g(y_{\epsilon, i, \omega})}{|T_{\omega, \epsilon}'(y_{\epsilon, i, \omega})|}-
\frac{g(y_{i, \omega})}{|T_\omega'(y_{i, \omega})|}  \right ) \\
&=\sum_i \frac{g(y_{\epsilon, i, \omega})-g(y_{i, \omega})}{|T_{\omega}'(y_{i, \omega})|}+\sum_i  \frac{g(y_{\epsilon, i, \omega})(|T_\omega'(y_{i, \omega})|-|T_{\omega, \epsilon}'(y_{\epsilon, i, \omega})|)}{|T_\omega'(y_{i, \omega})| \cdot |T_{\omega, \epsilon}'(y_{\epsilon, i, \omega})|} \\
&=:(I)+(II),
\end{split}
\]
where $y_{i, \om}=y_{0, i, \om}$.
Note that 
\[
|(I)| \le \|g'\|_\infty |y_{\epsilon, i, \omega}-y_{i, \omega}|\L_\omega \mathbf 1(x).
\]
Similarly, 
\[
\begin{split}
|(II)| &\le \sum_i \frac{|g(y_{\epsilon, i, \omega})|\cdot |T_{\omega, \epsilon}'(y_{\epsilon, i, \omega})-T_\omega'(y_{\epsilon, i, \omega})|}{|T_\omega'(y_{i, \omega})| \cdot |T_{\omega, \epsilon}'(y_{\epsilon, i, \omega})|} \\
&\phantom{=}+\sum_i\frac{|g(y_{\epsilon, i, \omega})|\cdot |T_{\omega}'(y_{\epsilon, i, \omega})-T_\omega'(y_{ i, \omega})|}{|T_\omega'(y_{i, \omega})| \cdot |T_{\omega, \epsilon}'(y_{\epsilon, i, \omega})|} \\
&\le \gamma_\omega^{-1}\|g\|_\infty d_{C^1} (T_{\omega, \epsilon}, T_\omega)\L_\omega \mathbf 1(x)+\gamma_\omega^{-1}\|g\|_\infty \|T_\omega''\|_\infty |y_{\epsilon, i, \omega}- y_{i, \omega}|\L_\omega \mathbf 1(x).
\end{split}
\]
This readily implies the conclusion of the lemma.
\end{proof}

\subsection{Verification of~\eqref{deriv} with $\B_s=C^1$ and $\B_{ss}=C^3$ and  of \eqref{hatL} and \eqref{C om}.}

We again focus for the case of simplicity to the   one-dimensional case.  Our arguments and exposition follows closely~\cite[Section 4.4]{DGS} although instead of Sobolev spaces here we consider spaces of smooth functions.
We consider a measurable map
 $\mathbf T \colon \Omega \to C^4(I\times M, M)$, where $I\subset (-1, 1)$ is an open interval containing $0$. We let 
 \[
 T_{\omega, \epsilon}:=\mathbf T(\omega)(\epsilon, \cdot), \quad (\omega, \epsilon)\in \Omega \times I.
 \]
 Let us recall certain notations from~\cite{DGS}. For $\phi\in C^r(M, \R)$, set 
 \[
 \begin{split}
		g_{\omega,\epsilon}&:=\frac{1}{|T'_{\omega,\epsilon}|}\in C^{3}(M,\R)\\
		V_{\omega,\epsilon}(\phi) &:=-\frac{\phi'}{T'_{\omega,\epsilon}}\cdot\partial_\epsilon T_{\omega,\epsilon}\in C^{r-1}(M,\R).
  \end{split}
  \]
  We also define
	\[
		J_{\omega,\epsilon}:=\frac{\partial_\epsilon g_{\omega,\epsilon}+V_{\omega,\epsilon}(g_{\omega,\epsilon})}{g_{\omega,\epsilon}}\in C^{2}(M,\R).
	\]

If $\L_{\omega, \epsilon}$ denotes the transfer associated corresponding to $T_{\omega, \epsilon}$ and $\phi$ is an observable, then the formal differentiation yields
 
\[
		\partial_\epsilon [\L_{\omega,\epsilon}\phi]=\L_{\omega,\epsilon}(J_{\omega,\epsilon}\cdot\phi+V_{\omega,\epsilon}\phi)
  \]
  and 
  \[
		\partial_{\epsilon}^2[\L_{\omega, \epsilon}\phi]= \L_{{\omega,\epsilon}}\left(J_{\omega,\epsilon}^2\phi+J_{\omega,\epsilon}(V_{\omega,\epsilon}\phi)+V_{\omega,\epsilon}(J_{\omega,\epsilon}\phi)+V_{\omega,\epsilon}(V_{\omega,\epsilon}\phi)+[\partial_\epsilon J_{\omega,\epsilon}]\cdot \phi+\partial_{\epsilon}[V_{\omega,\epsilon}\phi]\right).
  \]

In the sequel, we assume that there are random variables $K_i \colon \Omega \to [1, \infty)$, $i\in \{0, 1, 2\}$ such that for $\mathbb P$-a.e. $\omega \in \Omega$ and $\epsilon \in I$,
\begin{equation}\label{TU1}
    \| \partial_\epsilon^i T_{\omega, \epsilon}\|_{C^{4-i}} \le K_0(\omega) \quad i\in \{0, 1, 2\},
\end{equation}
\begin{equation}\label{TU2}
\| \partial_\epsilon^i g_{\omega, \epsilon}\|_{C^{2-i}} \le K_1(\omega) \quad i\in \{0, 1\},
\end{equation}
\begin{equation}\label{TU3}
\|\partial_\epsilon^i J_{\omega, \epsilon}\|_{C^{2-i}} \le K_2(\omega) \quad i\in \{0, 1\},
\end{equation}
and
\begin{equation}\label{TU4}
\|\L_{\omega, \epsilon}\|_{C^1} \le K_3(\omega).
\end{equation}
 Take $\phi \in C^3(M, \R)$. In the following, $c>0$ will denote a generic positive constant independent on $\om$ and $\varepsilon$ that can change from one occurrence to the next.
Firstly,  \eqref{TU3} gives that 
\[
	\|J_{\omega,\epsilon}^2\phi\|_{C^1} \le c\|J_{\omega, \epsilon}\|_{C^1} \| J_{\omega, \epsilon}\phi \|_{C^1} \le c\|J_{\omega, \epsilon}\|_{C^1}^2 \| \phi \|_{C^1}\le c(K_2(\omega))^2\| \phi \|_{C^3}, 
	\]
for $\mathbb P$-a.e. $\omega \in \Omega$ and $\epsilon \in I$.
Secondly, \eqref{TU1}-\eqref{TU3} imply that 
	\[
	\begin{split}
		\|J_{\omega,\epsilon}(V_{\omega,\epsilon}\phi)\|_{C^1} & \le c\|J_{\omega, \epsilon}\|_{C^1} \|\phi'g_{\omega, \epsilon} \partial_{\epsilon}T_{\omega, \epsilon}(\cdot)\|_{C^1}  \\
		&\le  c\|J_{\omega, \epsilon}\|_{C^1} \|g_{\omega, \epsilon} \partial_{\epsilon}T_{\omega, \epsilon}(\cdot)\|_{C^1}  \|\phi\|_{C^2}\\
		&\le  cK_0(\omega)K_1(\omega)K_2(\omega)\|\phi\|_{C^3},
	\end{split}
	\]
 for $\mathbb P$-a.e. $\omega \in \Omega$ and $\epsilon \in I$.
 Furthermore, 
	\[
	\begin{split}
		\|V_{\omega,\epsilon}(J_{\omega,\epsilon}\phi)\|_{C^1} &\le c\|g_{\omega, \epsilon} \partial_{\epsilon}T_{\omega, \epsilon}(\cdot)\|_{C^1}  \|(J_{\omega,\epsilon}\phi)'\|_{C^1} \\
		&\le c\|g_{\omega, \epsilon} \partial_{\epsilon}T_{\omega, \epsilon}(\cdot)\|_{C^1}\|J_{\omega,\epsilon}\phi\|_{C^2} \\
		&\le c\|g_{\omega, \epsilon} \partial_{\epsilon}T_{\omega, \epsilon}(\cdot)\|_{C^1}\| J_{\omega, \epsilon}\|_{C^2}\|\phi\|_{C^2}\\
		&\le  cK_0(\omega)K_1(\omega)K_2(\omega)\|\phi\|_{C^3},
	\end{split}
	\]
for $\mathbb P$-a.e. $\omega \in \Omega$ and $\epsilon \in I$.
In addition,
	\[
	\begin{split}
		\|V_{\omega,\epsilon}(V_{\omega,\epsilon}\phi)\|_{C^1} &\le c\|g_{\omega, \epsilon} \partial_{\epsilon}T_{\omega, \epsilon}(\cdot)\|_{C^1}  \|(V_{\omega,\epsilon}\phi)'\|_{C^1} \\
		&\le c\|g_{\omega, \epsilon} \partial_{\epsilon}T_{\omega, \epsilon}(\cdot)\|_{C^1}\|V_{\omega,\epsilon}\phi \|_{C^2} \\
		&\le c\|g_{\omega, \epsilon} \partial_{\epsilon}T_{\omega, \epsilon}(\cdot)\|_{C^1}\|g_{\omega, \epsilon} \partial_{\epsilon}T_{\omega, \epsilon}(\cdot)\|_{C^2}  \|\phi '\|_{C^2} \\
		&\le c\|g_{\omega, \epsilon} \partial_{\epsilon}T_{\omega, \epsilon}(\cdot)\|_{C^2}^2 \|\phi \|_{C^3}\\
		&\le c(K_0(\omega)K_1(\omega))^2\|\phi \|_{C^3},
	\end{split}
	\]
 for $\mathbb P$-a.e. $\omega \in \Omega$ and $\epsilon \in I$. On the other hand, 
 \[
 \| [\partial_\epsilon J_{\omega,\epsilon}]\cdot \phi\|_{C^1} \le c\| \partial_\epsilon J_{\omega,\epsilon}\|_{C^1}  \|\phi \|_{C^1} \le cK_2(\omega)\|\phi\|_{C^3},
 \]
 for $\mathbb P$-a-e. $\omega \in \Omega$ and $\epsilon \in I$. Finally, 
 \[
 \begin{split}
		\|\partial_{\epsilon}[V_{\omega,\epsilon}\phi]\|_{C^1} &=\|\phi' \partial_\epsilon g_{\omega, \epsilon}\partial_{\epsilon}T_{\omega, \epsilon}+\phi' g_{\omega, \epsilon} \partial_\epsilon^2 T_{\omega, \epsilon}\|_{C^1} \\
		&\le c\left(\|\partial_\epsilon g_{\omega, \epsilon}\partial_{\epsilon}T_{\omega, \epsilon} \|_{C^1}+
		\|g_{\omega, \epsilon} \partial_\epsilon^2 T_{\omega, \epsilon}\|_{C^1}\right)\|\phi'\|_{C^1} \\
		&\le c \left(\|\partial_\epsilon g_{\omega, \epsilon}\partial_{\epsilon}T_{\omega, \epsilon} \|_{C^1}+
		\|g_{\omega, \epsilon} \partial_\epsilon^2 T_{\omega, \epsilon}\|_{C^1}\right) \| \phi \|_{C^2}\\
		&\le cK_0(\omega)K_1(\omega) \| \phi \|_{C^3},
	\end{split}
 \]
 for $\mathbb P$-a.e. $\omega \in \Omega$ and $\epsilon \in I$. Putting all these estimates together yields that 
 \begin{equation}\label{kre}
 \|		\partial_{\epsilon}^2[\L_{\omega, \epsilon}\phi]\|_{C^1} \le C_3(\omega) \| \phi \|_{C^3} \quad \text{for $\mathbb P$-a.e. $\omega \in \Omega$, $\epsilon \in I$ and $\phi \in C^3(M, \R)$,}
 \end{equation}
 where
 \begin{equation}\label{C 3 def}
  C_3(\omega):=cK_3(\omega)(K_0(\omega)K_1(\omega)K_2(\omega)+(K_2(\omega))^2+ (K_0(\omega)K_1(\omega))^2), \quad \omega \in \Omega.    
 \end{equation}
 We define $\hat L_\omega \colon C^3 \to C^1$ by 
 \[
 \hat{\L}_\omega \phi=\L_\omega(J_{\omega, 0} \phi+V_{\omega, 0}\phi), \quad \phi \in C^3.
 \]
 Then, 
 \begin{equation}\label{tti}
 \| \hat{\L}_\omega \phi \|_{C^1} \le cK_3(\omega) (K_2(\omega)+K_0(\omega)K_1(\omega))\| \phi \|_{C^2} \le  C_3(\omega)\| \phi \|_{C^3},
 \end{equation}
 for $\mathbb P$-a.e. $\omega \in \Omega$ and $\phi \in C^3(M, \R)$.
 Observe that~\eqref{kre} implies~\eqref{deriv} by using Taylor's formula of order two (see~\cite{DGS} for this argument). Also, \eqref{tti} implies~\eqref{hatL}.
 In addition, 
 provided that $K_i\in L^{q_i}(\Omega, \mathcal F, \mathbb P)$ for some $q_i>0$, we have that $C_3\in L^s(\Omega, \mathcal F, \mathbb P)$, where $s=\min \{s_1, s_2, s_3\}$ and
 \[
 \frac{1}{s_1}=\frac{1}{q_0}+\frac{1}{q_1}+\frac{1}{q_2}+\frac{1}{q_3}, \quad  \frac{1}{s_2}=\frac{2}{q_0}+\frac{2}{q_1}+\frac{1}{q_3} \quad  \text{and} \quad  \frac{1}{s_3}=\frac{2}{q_2}+\frac{1}{q_3}.
 \]
 \begin{remark}
We note that~\eqref{tti} implies~\eqref{C om} with $C(\om)=C_3(\om)$.
 \end{remark}
 
 Next we observe that 
 \[
 \|g_{\omega, \epsilon}\|_{C^0}\le \gamma_\omega^{-1}, \quad \|g_{\omega, \epsilon}'\|_{C^0}=\left \| \frac{T_{\omega, \epsilon}''}{(T_{\omega, \epsilon}')^2} \right \|_{C^0} \le \gamma_\omega^{-2}K_0(\omega)
 \]
 and 
 \[
 \| g_{\omega, \epsilon}''\|_{C^0} \le \left \|\frac{T_{\omega, \epsilon}'''}{(T_{\omega, \epsilon}')^2}\right \|_{C^0}+2\left \| \frac{(T_{\omega, \epsilon}'')^2}{(T_{\omega, \epsilon}')^3}\right \|_{C^0} \le (\gamma_\omega^{-2}+2\gamma_{\omega}^{-3})(K_0(\omega))^2,
 \]
 for $\mathbb P$-a.e. $\omega \in \Omega$ and $\epsilon \in I$. Consequently, 
 \begin{equation}\label{sas1}
 \|g_{\omega, \epsilon}\|_{C^2} \le 2(\gamma_{\omega}^{-1}+\gamma_\omega^{-2}+\gamma_\omega^{-3})(K_0(\omega))^2 \quad \text{for $\mathbb P$-a-e. $\omega \in \Omega$ and $\epsilon \in I$.}
 \end{equation}
 Moreover, 
 \[
 \| \partial_\epsilon g_{\omega, \epsilon} \|_{C^0}=\left \| \frac{\partial_\epsilon T_{\omega, \epsilon}'}{(T_{\omega, \epsilon}')^2}\right \|_{C^0} \le \gamma_\omega^{-2}K_0(\omega)
 \]
 and
 \[
 \|(\partial_\epsilon g_{\omega, \epsilon})'\|_{C^0} \le \left \| \frac{\partial_\epsilon T_{\omega, \epsilon}''}{(T_{\omega, \epsilon}')^2}\right \|_{C^0}+2\left \| \frac{\partial_\epsilon T_{\omega, \epsilon}'T_{\omega, \epsilon}''}{(T_{\omega, \epsilon}')^3}\right \|_{C^0} \le \gamma_\omega^{-2}K_0(\omega)+2\gamma_\omega^{-3}(K_0(\omega))^2,
 \]
 for $\mathbb P$-a.e. $\omega \in \Omega$ and $\epsilon \in I$. Hence, 
 \begin{equation}\label{sas2}
 \| \partial_\epsilon g_{\omega, \epsilon}\|_{C^1}\le 2(\gamma_\omega^{-2}+\gamma_\omega^{-3})(K_0(\omega))^2 \quad \text{for $\mathbb P$-a-e. $\omega \in \Omega$ and $\epsilon \in I$.}
 \end{equation}
 From~\eqref{sas1} and~\eqref{sas2} we conclude that $K_1$ in~\eqref{TU2} can be taken as 
 \begin{equation}\label{K1 def}
K_1(\omega)=2(\gamma_\omega^{-1}+\gamma_\omega^{-2}+\gamma_\omega^{-3})(K_0(\omega))^2, \quad \omega \in \Omega.     
 \end{equation}
On the other hand, 
\[
\| J_{\omega, \epsilon}\|_{C^0} \le \left \| \frac{\partial_\epsilon T_{\omega, \epsilon}'}{T_{\omega, \epsilon}'}\right \|_{C^0}+\| g_{\omega, \epsilon}'\partial_\epsilon T_{\omega, \epsilon}\|_{C^0} \le (\gamma_\omega^{-1}+\gamma_\omega^{-2})(K_0(\omega))^2,
\]
\[
\begin{split}
\| J_{\omega, \epsilon}'\|_{C^0} &\le \left \| \frac{\partial_\epsilon T_{\omega, \epsilon}'' }{T_{\omega, \epsilon}'}\right \|_{C^0}+\left \| \frac{\partial_\epsilon T_{\omega, \epsilon}' T_{\omega, \epsilon}''}{(T_{\omega, \epsilon}')^2}\right \|_{C^0}+\|g_{\omega, \epsilon}''\partial_\epsilon T_{\omega, \epsilon}\|_{C^0}+\|g_{\omega, \epsilon}' \partial_\epsilon T_{\omega, \epsilon}'\|_{C^0} \\
&\le \gamma_\omega^{-1}K_0(\omega)+\gamma_\omega^{-2}(K_0(\omega))^2+ (\gamma_\omega^{-2}+2\gamma_\omega^{-3})(K_0(\omega))^2 +\gamma_{\omega}^{-2}(K_0(\omega))^2 \\
&\le (\gamma_\omega^{-1}+3\gamma_\omega^{-2}+2\gamma_\omega^{-3})(K_0(\omega))^2,
\end{split}
\]
\[
\begin{split}
\|J_{\omega, \epsilon}''\|_{C^0} &\le \left \|\frac{\partial_\epsilon T_{\omega, \epsilon}'''}{T_{\omega, \epsilon}'}\right \|_{C^0}+2\left \|\frac{\partial_\epsilon T_{\omega, \epsilon}'' T_{\omega, \epsilon}''}{(T_{\omega, \epsilon}')^2}\right \|_{C^0}+\left \| \frac{\partial_\epsilon T_{\omega, \epsilon}' T_{\omega, \epsilon}'''}{(T_{\omega, \epsilon}')^2}\right \|_{C^0} +2\left \|\frac{\partial_\epsilon T_{\omega, \epsilon}' (T_{\omega, \epsilon}'')^2}{(T_{\omega, \epsilon}')^3}\right \|_{C^0}\\
&\phantom{\le}+\|g_{\omega, \epsilon}'''\partial_\epsilon T_{\omega, \epsilon}\|_{C^0}+2\|g_{\omega, \epsilon}''\partial_\epsilon T_{\omega, \epsilon}'\|_{C^0} +\|g_{\omega, \epsilon}'\partial_\epsilon T_{\omega, \epsilon}''\|_{C^0},
\end{split}
\]
for $\mathbb P$-a.e. $\omega \in \Omega$.
Observe also that 
\[
\begin{split}
\|g_{\omega, \epsilon}'''\|_{C^0} & \le \left \| \frac{T_{\omega, \epsilon}^{(4)}}{(T_{\omega, \epsilon}')^2}\right \|_{C^0}+6\left \|\frac{T_{\omega, \epsilon}'' T_{\omega, \epsilon}'''}{(T_{\omega, \epsilon}')^3}\right  \|_{C^0}+6\left \|\frac{(T_{\omega, \epsilon}'')^3}{(T_{\omega, \epsilon}')^4}\right \|_{C^0} \\
&\le \gamma_\omega^{-2}K_0(\omega)+6\gamma_\omega^{-3}(K_0(\omega))^2+6\gamma_\omega^{-4}(K_0(\omega))^3 \\
&\le 6(\gamma_\omega^{-2}+\gamma_\omega^{-3}+\gamma_\omega^{-4})(K_0(\omega))^3,
\end{split}
\]
for $\mathbb P$-a.e. $\omega \in \Omega$ and $\epsilon \in I$. This now easily implies that 
\[
\|J_{\omega, \epsilon}''\|_{C^0}\le c(\gamma_\omega^{-1}+\gamma_\omega^{-2}+\gamma_\omega^{-3}+\gamma_\omega^{-4})(K_0(\omega))^4,
\]
for $\mathbb P$-a.e. $\omega \in \Omega$ and $\epsilon \in I$, where $c>0$ is some constant independent on $\omega$ and $\epsilon$. In order to bound $\|\partial_\epsilon J_{\omega, \epsilon}\|_{C^1}$, we begin by noting that
	\[
	\begin{split}
		\partial_\epsilon J_{\omega,\epsilon} &=\partial_\epsilon T_{\omega, \epsilon}' \bigg (\partial_\epsilon g_{\omega, \epsilon}-g_{\omega, \epsilon}' g_{\omega, \epsilon} \partial_{\epsilon}T_{\omega, \epsilon}\bigg )\\ 
		&\phantom{=} + T_{\omega, \epsilon}'\bigg (\partial_\epsilon^2 g_{\omega, \epsilon}-\partial_\epsilon g_{\omega, \epsilon}' g_{\omega, \epsilon} \partial_{\epsilon}T_{\omega, \epsilon}-g_{\omega, \epsilon}' \partial_\epsilon g_{\omega, \epsilon} \partial_{\epsilon}T_{\omega, \epsilon}-g_{\omega, \epsilon}'   g_ {\omega, \epsilon} \partial_{\epsilon}^2 T_{\omega, \epsilon}\bigg ).
	\end{split}
	\]
Hence,
\[
\begin{split}
\|\partial_\epsilon J_{\omega, \epsilon}\|_{C^0} &\le \|\partial_\epsilon T_{\omega, \epsilon}'\|_{C^0} (\| \partial_\epsilon g_{\omega, \epsilon}\|_{C^0}+\|g_{\omega, \epsilon}'\|_{C^0}\cdot \|g_{\omega, \epsilon}\|_{C^0}\cdot \|\partial_\epsilon T_{\omega, \epsilon}\|_{C^0}) \\
&\phantom{\le}+\|T_{\omega, \epsilon}'\|_{C^0}(\|\partial_\epsilon^2 g_{\omega, \epsilon}\|_{C^0}+\| \partial_\epsilon g_{\omega, \epsilon}'\|_{C^0}\cdot \|g_{\omega, \epsilon}\|_{C^0}\cdot \|\partial_\epsilon T_{\omega, \epsilon}\|_{C^0}+\\
&\phantom{\le}+\|g_{\omega, \epsilon}'\|_{C^0}\cdot \|\partial_\epsilon g_{\omega, \epsilon}\|_{C^0}\cdot \|\partial_\epsilon T_{\omega, \epsilon}\|_{C^0}+\|g_{\omega, \epsilon}'\|_{C^0}\cdot \| g_{\omega, \epsilon}\|_{C^0}\cdot \|\partial_\epsilon^2T_{\omega, \epsilon}\|_{C^0} ).
\end{split}
\]
Noting that 
\[
\|\partial_\epsilon^2 g_{\omega, \epsilon}\|_{C^0} \le \left \| \frac{\partial_\epsilon^2 T_{\omega, \epsilon}'}{(T_{\omega, \epsilon}')^2}\right \|_{C^0}+2\left \|\frac{(\partial_\epsilon T_{\omega, \epsilon}')^2}{(T_{\omega, \epsilon}')^3} \right \|_{C^0} \le (\gamma_\omega^{-2}+2\gamma_{\omega}^{-3})(K_0(\omega))^2,
\]
for $\mathbb P$-a-e. $\omega \in \Omega$, we can see that 
\[
\|\partial_\epsilon J_{\omega, \epsilon}\|_{C^0} \le c(\gamma_\omega^{-2}+\gamma_\omega^{-3}+\gamma_{\omega}^{-4})(K_0(\omega))^4,
\]
for $\mathbb P$-a.e. $\omega \in \Omega$ and $\epsilon \in I$. Next, observing that
\[
\|g_{\omega, \epsilon}'''\|_{C^0} \le (\gamma_\omega^{-2}+6\gamma_\omega^{-3}+6\gamma_\omega^{-4})(K_0(\omega))^3
\]
and 
\[
\|\partial_\epsilon^2 g_{\omega, \epsilon}'\|_{C^0}\le (\gamma_\omega^{-2}+6\gamma_\omega^{-3}+6\gamma_\omega^{-4})(K_0(\omega))^3,
\]
one can conclude that 
\[
\| \partial_\epsilon J_{\omega, \epsilon}'\|_{C^0} \le c(\gamma_\omega^{-2}+\gamma_{\omega}^{-3}+\gamma_\omega^{-4}+\gamma_\omega^{-5})(K_0(\omega))^5,
\]
for $\mathbb P$-a.e. $\omega \in \Omega$ and $\epsilon \in I$, where $c>0$ is some constant independent on $\omega$ and $\epsilon$. Hence, we can take $K_2$ in~\eqref{TU3} of the form
\begin{equation}\label{K2 def}
 K_2(\omega)=c(\gamma_\omega^{-1}+\gamma_\omega^{-2}+\gamma_{\omega}^{-3}+\gamma_\omega^{-4}+\gamma_\omega^{-5})(K_0(\omega))^5, \quad \omega \in \Omega.   
\end{equation}


\begin{remark}\label{K12}
Note that when $\gamma_\om \ge 1$, we can take
\[
K_1(\om)=c(K_0(\om))^2 \quad \text{and} \quad K_2(\om)=c(K_0(\om))^5,
\]
as well as $q_1=q_0/2$ and $q_2=q_0/5$.
\end{remark}

\begin{remark}
We note that similar sufficient conditions for \eqref{deriv} can be obtained using a slightly more direct approach. Let $\mathbf T$ be as above and 
let $y_\om(\varepsilon,\cdot)=y_{ \varepsilon, \om}$ be an inverse branch of $T_{\om, \varepsilon}$, that is $y$ is defined on some open set $U$ and 
$$
T_{\om, \varepsilon}(y_{\varepsilon, \om}(x))=x
$$
for all $x\in U$. Writing $T_\om(\varepsilon, x)=T_{\om, \varepsilon}(x)$,  we have that 
\begin{equation}\label{base}
 T_\om (\varepsilon,y_\om(\varepsilon,x))=x.  \end{equation}
By differentiating with respect to  $x$ or $\varepsilon$  we see that 
 $$
D_x y_\om(\varepsilon,x)=(D_xT_\om(\varepsilon,y_\om(\varepsilon, x)))^{-1}
 $$
 and
 $$
(D_\varepsilon T_\om)(\varepsilon,y_\om(\varepsilon,x))+(D_x T_\om)(\varepsilon,y(\varepsilon,x))(D_\varepsilon y)(\varepsilon,x)=0.
 $$
 Thus, 
 \begin{equation}\label{D y eps1}
(D_\varepsilon y)(\varepsilon,x)=-\left((D_x T)(\varepsilon,y(\varepsilon,x))\right)^{-1}(D_\varepsilon T)(\varepsilon,y_\om(\varepsilon,x)).     
 \end{equation}
Continuing this way we can calculate all the partial derivatives of $y$ up to the fourth order  by means of the derivatives up to  order four. Using that 
$$
\mathcal L_{\om, \varepsilon}g(x)=\sum_i\text{Jac}(y_{\varepsilon, i,\om}(x))g(y_{\varepsilon,\om,i}(x))
$$
we see that 
$$
\|D_\varepsilon^2\mathcal L_{\om, \varepsilon}g\|_{C^1}\leq A_1(\om)\|g\|_{C^3},
$$
where $A_1(\om)$ is a polynomial in the supremum over $\varepsilon\in I$ and $x$ of the derivatives up to order four of $(\varepsilon,x)\to T_{\varepsilon,\om}(x)$. Now estimates similar to the ones in the previous section  follow from the Lagrange form of Taylor remainders.
\end{remark}

\subsection{Proof of Theorem \ref{1D EG}}\label{pf1d}

First, we note that for each $\varepsilon\in I$ the random dynamical system $(T_{\om,\varepsilon})_{\om\in\Omega}$ 
satisfy the conditions in \cite[Section 3.3]{YH 23} with $\text{dim}(M)=1$, namely we are in the  circumstances of the example described after Remark \ref{R17} with $M$ being one-dimensional.
Thus (see \cite[Section 3.3]{YH 23}), the conditions of Section \ref{3.1} are in force. Note that we can ensure that all the estimates hold simultaneously for all $\varepsilon\in I$ since  $\|(D(T_{\om,\varepsilon})^{-1}\|_\infty \leq \gamma_\om^{-1}$ with $\gamma_\om$ which does not depend on $\varepsilon$.   Next, as explained in \cite[Section 3.3]{YH 23},  in Assumption \ref{Mom Ass} in Appendix \ref{App} we can take $D_\om=C_0N(\om)Z_\om$ as in \eqref{D def} with $N(\om)=\mathcal A(\om)$ and $Z_\om$  is defined in \eqref{Z def}. 
Now, by applying \cite[Lemma 3.5 (i)]{YH 23}, \cite[Lemma 3.7]{YH 23} and taking into account \eqref{psi cond} we see that   $Z_\om\in \cap_{p\geq 1}L^p(\Omega,\mathcal F, \mathbb P)$. 
Next, we note that the functions $\phi_{\om,\varepsilon}=\ln|T_{\om,\varepsilon}'|$ satisfy 
$$
\|\phi_{\om,\varepsilon}\|_\infty\leq \ln \mathcal A(\om) \quad \text{and} \quad 
\|\phi_{\om,\varepsilon}'\|_\infty\leq \mathcal A(\om).
$$
Taking into account these estimates and the formulas for $D_\om$ and $N(\om)$  we conclude that  Assumption \ref{Mom Ass} holds with $p=b_2=\tilde b=\bar p$ and  $q_0<q$ arbitrary. Later on we will have further restrictions on $q_0$ and $q$ that will guarantee that Assumption \ref{Poly Ass} is in force (we will have take both arbitrarily close to $\sqrt{\bar p}$).

Next,  by \cite[Lemma 3.10]{YH 23} the random variable $m(\om)$ defined in \eqref{m def}, which is identical to the one in \cite[(3.5)]{YH 23}, and which also appears in Assumption \ref{Poly Ass} has exponential tails. Moreover, since $\om\mapsto \gamma_\om$  depends only on $\om_0$ by  \cite[Lemma 3.5 (i)]{YH 23} and \cite[Corollary 3.7]{YH 23} in our circumstances $\|Z_{\om}-\mathbb E[Z_\om|\mathcal F_{-r,r}]\|_{L^p}$ decays\footnote{Note that \cite[Lemma 3.5 (i)]{YH 23} guarantees that $v_r$ in \cite[Corollary 3.7]{YH 23} satisfies $v_r=O(e^{-br}), b>0$, and so we indeed get the desired exponential decay.} exponentially fast to $0$.   
Thus, by \cite[Lemma 3.11]{YH 23} Assumption \ref{Inner aprpox} from Appendix \ref{App} 
holds with arbitrarily large $M$.

By applying Theorem \ref{OSC1} with $q_0,q$ arbitrarily close to $\sqrt{\bar p}$ (so that the restrictions in Assumption \ref{Poly Ass} will be satisfied)  we conclude that there is a constant $M_0$ such that when $n\geq M_0m(\om)$ then \eqref{spectral gap} holds with arbitrarily large $\beta$ and with $p_1<\sqrt{\bar p}$, but arbitrarily close to $\sqrt{\bar p}$. To get \eqref{spectral gap} when $n<M_0m(\om)$ we proceed like in Remark \ref{rem small n}. First, by Proposition \ref{Sec3.2 Main} we see that for $\mathbb P$-a.e. $\om$ and for all $\varepsilon\in I$ and $n\in\mathbb N$
\begin{equation}\label{C 0 }
\|\mathcal L_{\sigma^{-n}\om,\varepsilon}^n\textbf{1}\|_{\infty}\leq C_0(\om)  
\end{equation}
with $C_0(\om)\in L^r(\Omega,\mathcal F,\mathbb P)$ and $r$ is given by $1/r=1/q_0+1/\bar p$. Recalling that we can take $q_0$ arbitrarily close to $\sqrt{\bar p}$, we see that $r$ can be taken to be arbitrarily close to $\frac12\sqrt{\bar p}$. Thus, proceeding like in Remark \ref{rem small n} we get by taking $\beta$ and $d$ arbitrarily large in Assumption \ref{Poly Ass} (i) that \eqref{spectral gap} holds when $n<M_0m(\om)$ with $p_1$  arbitrarily  close to $r$, but smaller (using \eqref{r 0 form}). This proves that we can take arbitrarily large $\beta$ and $p_1$  arbitrarily close to $\frac12\sqrt{\bar p}$.

Note that in that above arguments we also showed that \eqref{weakcontr} holds with $p_0=r$ which can be taken to be arbitrarily close to $\frac12\sqrt{\bar p}$. Next, we claim that \eqref{tn1} holds with $p_2<\frac13\bar p$ which can be arbitrarily close to $\frac14\bar p$. Indeed, taking into account Remark \ref{Rem 5.5} and Lemma \ref{Lem 5.5} we see that \eqref{tn1}  holds with $C_2(\om)=\bar q(\om)$ which does not exceed $3E(\om)(\mathcal A(\om))^{2}$, where $E(\om)$ is any upper bound for 
$\sup_\varepsilon\|\mathcal L_{\om,\varepsilon}\textbf{1}\|_\infty$.
Now   taking $E(\om)$ as in \eqref{E bound} and using \eqref{Domega} we get that $E(\om)\leq D_\om \mathcal A(\om)$. Since $D_\om=C_0\mathcal A(\om)Z_\om$ and $Z_\om\in L^p$ for all $p$ we conclude that $\bar q(\om)$ belongs to $L^{q}$ for all $q<\frac 14\bar p$.

Now we show that \eqref{deriv}, \eqref{hatL} and \eqref{C om} hold with $C(\om)=C_3(\om)\in L^{p_3}$ with every $p_3<\frac15\bar p$. 
We first notice that in \eqref{TU1} we can take $K_0(\om)=\mathcal A(\om)$. Thus as explained   in Remark~\ref{K12}, in   \eqref{TU2} and \eqref{TU3} we can take $K_1(\om)=c(\mathcal A(\om))^2$ and $K_2(\om)=c(\mathcal A(\om))^5$, respectively, where $c>0$ is a constant. Moreover, since 
$$
\|\mathcal L_{\om,\varepsilon}\|_\infty=\|\mathcal L_{\om,\varepsilon}\textbf{1}\|_\infty\leq E(\om)
$$
and 
$$
\|(\mathcal L_{\om,\varepsilon}g)'\|_\infty\leq C\|g\|_{C^1}\mathcal A(\om)\|\mathcal L_{\om,\varepsilon}\textbf{1}\|_\infty
$$
we see that in \eqref{TU4} we can take $K_3(\om)=cD_\om (\mathcal A(\om))^2=c'Z_\om (\mathcal A(\om))^3$. Now, using the formula \eqref{C 3 def} for $C_3(\om)$ we conclude that \eqref{deriv} holds with every $p_3<\frac{\bar p}{12}$.


Next we show that \eqref{top} holds with any $p_4<\frac{\sqrt{\bar p}}{47}$. We first recall that the discussion following~ \eqref{2nd} yields   that we can take $C_4(\om)=A_4(\om)$, where $A_4(\om)$ is defined in Remark \ref{Rem C4}.
Next, we recall that in our circumstances we can take any $q_0<\sqrt{\bar p}$ and any $a_0>1$ in Theorem \ref{OSC1}.
We also note that $\|\phi_{\om,\varepsilon}\|_{C^4}\leq c_4 (\mathcal A(\om))^4$ for some constant $c_4>0$, where $\phi_{\om,\varepsilon}(x)=\ln|T_{\om,\varepsilon}'(x)|$. Thus, in Lemma~\ref{phi lem} we can take any $d<\frac14\bar p$. Next, by Lemma \ref{ver der} we see that condition \eqref{Pre Adler} holds with $p=\frac13\bar p$. Since $\gamma_\om\leq \|T_{\om,\varepsilon}' \|_\infty \leq \mathcal A(\om)$ we also have that $\gamma_\om\in L^{\bar p}$ and so we can take $q=\bar p$ in Lemma \ref{Lemma 4}. Using the above  we conclude that the numbers $t,u$ and $u_1$ from Lemma \ref{Lemma 4} can be taken so that $t$ is arbitrarily close to $\frac{(\bar p)^{3/2}}{\bar p+4\sqrt{\bar p}}$ (but smaller), $u$ is arbitrarily close to to $\frac{(\bar p)^{3/2}}{\bar p+8\sqrt{\bar p}}$ (but smaller) and $u_1$ is arbitrarily close to $\frac{(\bar p)^{3/2}}{4\bar p+12\sqrt{\bar p}}$ (but smaller). In particular we can take $t$ arbitrarily close to $\frac{\sqrt{\bar p}}{5}$, $u$ is arbitrarily close to  $\frac{\sqrt{\bar p}}{9}$ and $u_1$  arbitrarily close to $\frac{\sqrt{\bar p}}{16}$.
Using this, that $q_0$ can be arbitrarily close to $\sqrt{\bar p}$ and that $d$ can be arbitrarily close to $\frac14\sqrt{\bar p}$  we see that the numbers $v_i$ in Lemma \ref{phi lem} can be taken so that
\begin{equation}\label{v's}
 v_1\geq \frac{\sqrt{\bar p}}{7}-\delta,\,\, v_2\geq  
\frac{\sqrt{\bar p}}{9}-\delta,\,\,v_3\geq \frac{\sqrt{\bar p}}{16}-\delta,\,\,\, v_4\geq \frac{\sqrt{\bar p}}{18}-\delta ,  
\end{equation}
for an arbitrarily small $\delta>0$ (the choice depends on $\delta$).
Thus, by Corollary \ref{Aux Cor} we see that for every $\delta>0$ small enough there is a random variable $U_\delta\in L^{q-\delta}(\Omega,\mathcal F,
\mathbb P), q=\frac{\sqrt{\bar p}}{82}$ such that for $\mathbb P$ a.e. $\om$ and all $n\in\mathbb N$ and $\varepsilon\in I$ we have\footnote{$U_\delta(\om)=C_0(\om)\cdot Q_4(\om)$, where $C_0(\om)$ is as in \eqref{C 0 } and $Q_4(\om)$ is like in Corollary \ref{Aux Cor}.}
\begin{equation}\label{estt}
\|\mathcal   L_{\sigma^{-n}\om, \varepsilon}^n \|_{C^4}\leq U_\delta(\om).   
\end{equation}
As explained in Section \ref{top ver sec}, this implies that $\|h_{\om,\varepsilon}\|_{C^3}\leq U_\delta(\om)$ and so we can take any $p_4<\frac{\sqrt{\bar p}}{82}$. 

Next, let us show that \eqref{AB} holds with any $p_5<\frac18\sqrt{\bar p}$. Using the formula \eqref{A def}  for $A(\om)$ we see that $A(\om)\in L^{p_1}$. Recall also that $p_1$ can be taken to be arbitrarily close to $\sqrt{\bar p}$ (but smaller).  Using the formula \eqref{B def} for $B(\om)$ and taking into account that $E(\om)\leq D_\om \mathcal A(\om)$ (see above) 
we see that $B(\om)\in L^{q'}$ with $q'$ such that $1/q'$ is arbitrarily close to $4/\bar p+2/q_0+\frac{1}{p_1}$ (but larger),   where we took into account that $p_1\leq \frac{\bar p}2$, and that in
in Lemma \ref{phi lem} we can take any $d<\frac14\bar p$ (as explained above). Using  that we take any $q_0<\sqrt{\bar p}$ and $p_1<\frac12\sqrt{\bar p}$  we see that $q'$ can be arbitrarily close to $b:=\frac{\bar p}{4\bar p+4\sqrt{\bar p}}$, which is smaller than $\sqrt {\bar p}$. Thus we can take any $p_5<b=\min(b,p_1)$. Finally, note that $b>\frac1{8}\sqrt{\bar p}$.

In order to complete the proof of the theorem we need to show that \eqref{TOB} holds with $A_r(\cdot)$ like in the statement of the theorem. However, the same estimates above which led to \eqref{estt} yield that \eqref{TOB} 
holds with $A_r(\om)\in L^a(\Omega,\mathcal F,\mathbb P), r\leq 4$ with $a$ arbitrarily close to $\frac{\sqrt{\bar p}}{82}$ (but smaller).
\qed

\appendix
\section{Upper bounds on derivatives of inverse branches and related results}\label{App A}

The following lemma was used in Section \ref{top ver sec}. The lemma is very elementary it is included here for the sake of completeness.
\begin{lemma}\label{Der lemma}
Let $T:M\to M$ and $y:U\to M$ be two functions such that $T\circ y=Id$ on some open set $U$. 
Then
$$
((DT)\circ y)^{-1}=Dy,
D^2y=((DT)\circ y)^{-1}\left(Id-
(D^2 T)\circ y\cdot (Dy)^2\right)
$$
$$
D^3y
$$
$$
=((DT)\circ y)^{-1}\left(Id-(D^3 T\circ y)(Dy)^3-2(D^2 T\circ y)(D^2y) (Dy)-(D^2 T\circ y)(Dy)(D^2y)\right)
$$
and
$$
D^4y=((DT)\circ y)^{-1}\Big(Id-(D^4T\circ y)(Dy)^4-3(D^3 T\circ y)(D^2y)(Dy)^2
$$
$$
-2(D^3T\circ y)(Dy)(D^2y)(Dy)-2(D^2T\circ y)[(D^3y)(Dy)+(D^2y)^2]
$$
$$
-(D^3 T\circ y)(Dy)^2(D^2y)-(D^2T\circ y)[(D^2y)^2+(Dy)(D^3y)]-(D^2T\circ y)(Dy)(D^3y)]\Big)
$$
\end{lemma}

\begin{proof}
The lemma follows by differentiating both sides of $T\circ y=Id$ four times and expressing the $i$-derivative of $y$ by means of the first $i-1$ derivatives of $y$ and the first $i$ derivatives of $T$.    
\end{proof}

\begin{proof}[Proof of Lemma \ref{Lemma 4}]
  For every such a  branch $y$ there are inverse  branches $y_i$ of $T_{\sigma^{-i}\omega,\varepsilon}$  such that 
  $$
y=y_n\circ y_{n-1}\circ\ldots\circ y_1.
  $$
  Thus, 
  $$
Dy=F_n\cdot F_{n-1}\cdots F_1 
  $$
  where
   $F_j=D(y_j)\circ y_{j-1}\circ\ldots \circ y_1$. Therefore, 
\begin{equation}\label{D2F}
  D^2y=\sum_{k=1}^n F_{n}\cdot F_{n-1}\cdots F_{k+1}\cdot D(F_k)\cdot F_{k-1}\cdots F_1.
\end{equation}
  Now, using that $\|F_j\|_\infty\leq \gamma_{\sigma^{-j}\omega}^{-1}$ and that 
  $$
\|D(F_k)\|_\infty\leq\|D^2(y_k)\|_\infty\prod_{s=1}^{k-1}\|Dy_s\|_\infty\leq
c(\sigma^{-k}\omega)
  $$
  we see that 
  $$
\|D^2y\|_\infty\leq\prod_{j=1}^n\gamma_{\sigma^{-j}\omega}^{-1}\sum_{k=1}^n \gamma_{\sigma^{-k}\omega}c(\sigma^{-k}\omega)=\prod_{j=1}^n\gamma_{\sigma^{-j}\omega}^{-1}\sum_{k=1}^n\alpha(\sigma^{-k}\omega),
  $$
where $\alpha(\omega)=c(\omega)\gamma_\omega\in L^s(\Omega,\mathcal F,\mathbb P)$ and $s>0$ 
is given by $1/s=1/p+1/q$. By Lemma~\ref{B lemma}, for every $\delta>0$ 
there is a random variable $R\in L^s(\Omega,\mathcal F,\mathbb P)$ such that for $\mathbb P$ a.e. $\om \in \Omega$ and all $n\in \N$ we have $\alpha(\sigma^{-n}\omega)\leq R(\omega)n^{1/s+\delta}$. Therefore, there is a constant $C=C_{s,\delta}>0$ such that $\mathbb P$-a.s. for all $n\geq1$ we have
  \begin{equation}\label{above}
\sum_{k=1}^n\alpha(\sigma^{-k}\omega)\leq CR(\omega)n^{1+1/s+\delta}.      
  \end{equation}
Now \eqref{11} follows from \eqref{gam dec}.

Next, we establish \eqref{12}.
Using~\eqref{D2F} we have that 
  \begin{equation}\label{D 3}
  \begin{split}
D^3y &=\sum_{k=1}^nF_n\cdot F_{n-1}\cdots F_{k+1}\cdot D^2(F_k)\cdot F_{k-1}\cdots F_1    \\  
&\phantom{=}+2\sum_{1\leq i<j\leq n}F_n\cdots F_{j+1}\cdot D(F_j) \cdot F_{j-1}\cdots F_{i+1}\cdot D(F_i)\cdot F_{i-1}\cdots F_1  \\
&=:I_1+I_2.
\end{split}
\end{equation}
  Let us first bound $\|I_1\|_\infty$. 
 We have 
 \begin{equation}\label{F 1}
D(F_k)=\left(D^2(y_k)\circ y_{k-1}\circ\ldots\circ y_1\right)\cdot F_{k-1}\cdots F_1.
 \end{equation}
  Thus, with $G_k :=D^2(y_k)\circ y_{k-1}\circ\ldots\circ y_1$ we have
 \begin{equation}\label{F 2}
 \begin{split}
D^2(F_k) &=\left(D^3(y_k)\circ y_{k-1}\circ\ldots \circ y_1\right)\cdot (F_{k-1}\cdots F_1)^2 \\
&\phantom{=}+\sum_{j=1}^{k-1}G_k\cdot F_{k-1}\cdots F_{j+1}\cdot D(F_j)\cdot F_{j-1}\cdots F_1 \\
&=:J_{1,k}+J_{2,k}.
\end{split}
\end{equation}
Now, using the above upper bounds on $\|F_{j}\|_\infty$ and $\|D(F_j)\|_\infty$ we see that
\begin{equation}\label{J k 1}
\|J_{1,k}\|_\infty\leq c(\sigma^{-k}\om)\prod_{j=1}^{k-1}\gamma_{\sigma^{-j}\om}^{-2}\leq c(\sigma^{-k}\om)   
\end{equation}
and 
\begin{equation}\label{J k 2}
\begin{split}
\|J_{2,k}\|_\infty &\leq c(\sigma^{-k}\om)  \sum_{j=1}^{k-1}c(\sigma^{-j}\om)\prod_{s=1}^{j-1}\gamma_{\sigma^{-s}\om}^{-1}\prod_{v=j+1}^{k-1}\gamma_{\sigma^{-v}\om}^{-1} \\
&\leq c(\sigma^{-k}\om)\prod_{s=1}^{k-1}\gamma_{\sigma^{-s}\om}^{-1}
\sum_{j=1}^{k-1}\alpha(\sigma^{-j}\omega) \\
&\leq Cc(\sigma^{-k}\om)R(\om)k^{1+1/s+\delta},
\end{split}
\end{equation}
where in the last step we used \eqref{above} with  $n=k$.
Putting together the above estimates and using that $\|F_j\|_\infty\leq \gamma_{\sigma^{-j}\om}^{-1}$
 we get that 
\[
\begin{split}
\|I_1\|_\infty &\leq (1+CR(\om))\sum_{k=1}^n k^{1+1/s+\delta}c(\sigma^{-k}\om)\prod_{j=1}^{k-1}\gamma_{\sigma^{-j}\om}^{-1}\prod_{j=k+1}^{n}\gamma_{\sigma^{-j}\om}^{-1} \\
&=(1+CR(\om))\prod_{j=1}^{n}\gamma_{\sigma^{-j}\om}^{-1}\sum_{k=1}^nk^{1+1/s+\delta}\alpha(\sigma^{-k}\om) \\
&\leq  
C R(\om)(1+CR(\om))n^{2+2/s+2\delta}\prod_{j=1}^n\gamma_{\sigma^{-j}\om}^{-1} \\
&\leq C(1+CR(\om))R(\om)E_\om n^{-(a_0-2-2/s-2\delta)},
 \end{split}
\]
 where  the last inequality  uses \eqref{gam dec}.

In order to bound $\|I_2\|_\infty$, using that $\|F_j\|_\infty\leq \gamma_{\sigma^{-j}\om}^{-1}$ and $\|D(F_j)\|\leq c(\sigma^{-j}\om)$, we see that
  $$
\|I_2\|_\infty\leq 2\sum_{1\leq i<j\leq n}\gamma_{\sigma^{-1}\om}^{-1}\cdots \gamma_{\sigma^{-(i-1)}\om}^{-1}  c(\sigma^{-i}\om)\gamma_{\sigma^{-(i+1)}\om}^{-1}\cdots \gamma_{\sigma^{-(j-1)}\om}^{-1}c(\sigma^{-j}\om)\gamma_{\sigma^{-(j+1)}\om}^{-1}\cdots \gamma_{\sigma^{-n}\omega}^{-1}
$$
$$
\leq 2\prod_{k=1}^n\gamma_{\sigma^{-k}\om}^{-1} \sum_{1\leq i<j\leq n}\alpha(\sigma^{-i}\om)\alpha(\sigma^{-j}\om)\leq \prod_{k=1}^n\gamma_{\sigma^{-k}\om}^{-1}\left(\sum_{j=1}^n \alpha(\sigma^{-j}\om)\right)^2 
  $$
  $$
\leq E_\om n^{-a_0}C^2(R(\om))^2n^{2+2/s+2\delta},
  $$
  where in the last inequality we used~\eqref{gam dec} and~\eqref{above}. 
Now \eqref{12} follows from the above estimates on $\|I_1\|_\infty$ and $\|I_2\|_\infty$.

Now we bound $D^4y$. Differentiating both sides of \eqref{D 3} and bounding all the terms by their supremum norm we see that 
$$
\|D^4y\|_\infty\leq 8(L_1+L_2+L_3)
$$
where with $\mathcal I_n=\{1,2,...,n\}$
$$
L_1:=\sum_{j=1}^n\left(\prod_{j<s\leq n}\|F_j\|_\infty\right)\|D^3(F_j)\|_\infty \left (\prod_{1\leq s<j}\|F_j\|_\infty \right),
$$
$$
L_2=\sum_{1\leq i,j\leq n, i\not=j}\left(\prod_{s\in \mathcal I_{n}, s\not=i,j}\|F_s\|_\infty\right)\|D^2(F_j)\|_\infty\|D(F_i)\|_\infty 
$$
$$
L_3=\sum_{1\leq i<j<k\leq n}\left(\prod_{s\in \mathcal I_{n}, s\not=i,j,k}\|F_s\|_\infty\right)\|D(F_i)\|_\infty \|D(F_j)\|_\infty \|D(F_k)\|_\infty. 
$$
Next we estimate $\|D^3(F_j)\|_\infty$. We will use the following abbreviation $F_{a,b}:=F_{b}\cdots F_{a+1}\cdot F_a$.
By differentiating both sides of \eqref{F 2} and bounding each term by its supremum norm
we see that 
\[
\begin{split}
\|D^3(F_k)\|_\infty &\leq \|D^4(y_k)\|_\infty \|F_{1,k-1}\|_\infty^3+ 2\|D^3(y_k)\|_\infty\|F_{1,k-1}\|_\infty\sum_{j=1}^{k-1}\|F_{j+1,k-1}D(F_j)F_{1,j-1}\|_\infty \\
&\phantom{\leq}+\sum_{j=1}^{k-1}\|D(G_k)\|_\infty\|F_{j+1,k-1}\|_\infty\|D(F_j)\|_\infty\|F_{1,j-1}\|_\infty \\
&\phantom{\leq}+
\sum_{i,j\in \mathcal I_{k-1}, i\not=j}\|G_k\|_\infty\|D(F_i)\|_\infty\|D(F_j)\|_\infty\prod_{s\in \mathcal I_{k-1}, s\not=i,j}\|F_s\|_\infty \\
&\phantom{\leq}+\sum_{j=1}^{k-1}\|G_k\|_\infty\|D^2(F_j)\|_\infty\prod_{s\in \mathcal I_{k-1}, s\not=j}\|F_s\|_\infty \\
&=:U_1(k)+U_2(k)+U_3(k)+U_4(k)+U_5(k).
\end{split}
\]
Next, denote $\beta_{a,b}(\om)=\prod_{j=a}^{b}\gamma_{\sigma^{-j}\om}^{-1}$. Recall also the notation $\alpha(\om)=\gamma_\om c(\om)$. 
Then, using that $\|F_j\|_\infty\leq \gamma_{\sigma^{-j}\om}^{-1}$ and \eqref{Pre Adler} we see that
\begin{equation}\label{U1}
U_1(k)\leq c(\sigma^{-k}\om)\left(\beta_{1,k-1}(\om)\right)^3.  
\end{equation}
Moreover, using also that $\|D(F_j)\|_\infty \le c(\sigma^{-j}\om)\beta_{1, j-1}(\om)$ (see~\eqref{F 1}) we see that 
\begin{equation}\label{U2}
U_2(k)\leq 2c(\sigma^{-k}\om)\beta_{1,k-1}(\om) \sum_{j=1}^{k-1}\beta_{j+1,k-1}(\om)(\beta_{1,j-1}(\om))^2 c(\sigma^{-j}\om)
\end{equation}
$$
=2c(\sigma^{-k}\om)(\beta_{1,k-1}(\om))^2 \sum_{j=1}^{k-1}\beta_{1,j-1}(\om)\alpha(\sigma^{-j}\om).
$$
Additionally, using that 
$$
D(G_k)= (D^3(y_k) \circ y_{k-1}\circ \ldots \circ y_1)F_{1,k-1}
$$
we have 
\begin{equation}\label{U3}
U_3(k)\leq c(\sigma^{-k}\om)\beta_{1,k-1}(\om)\sum_{j=1}^{k-1}\beta_{j+1,k-1}(\om)(\beta_{1,j-1}(\om))^2 c(\sigma^{-j}\om)
\end{equation}
$$
=c(\sigma^{-k}\om)(\beta_{1,k-1}(\om))^2\sum_{j=1}^{k-1}\beta_{1,j-1}(\om)\alpha(\sigma^{-j}\om).
$$
Furthermore, 
we get that
\begin{equation}\label{U4}
\begin{split}
U_4(k) &\leq c(\sigma^{-k}\om)\sum_{1\leq i,j\leq k-1}c(\sigma^{-i}\om)\beta_{1,i-1}(\om)c(\sigma^{-j}\om)\beta_{1,j-1}(\om)\prod_{s\not=i,j, 1\leq s\leq k-1}\gamma_{\sigma^{-s}\om}^{-1} \\
&=c(\sigma^{-k}\om)\beta_{1,k-1}(\om)\sum_{1\leq i,j\leq k-1}\alpha(\sigma^{-i}\om)\beta_{1,i-1}(\om)\alpha(\sigma^{-j}\om)\beta_{1,j-1}(\om) \\
&\leq c(\sigma^{-k}\om)\beta_{1,k-1}(\om)\left(\sum_{i=1}^{k-1}\alpha(\sigma^{-i}\om)\beta_{1,i-1}(\om)\right)^2.
\end{split}
\end{equation}
Finally, using \eqref{J k 1} and \eqref{J k 2} to estimate $\|D^2(F_k)\|_\infty$ we see that
\begin{equation}\label{U5}
U_5(k)\leq C'(1+R(\om))c(\sigma^{-k}\om)\sum_{j=1}^{k-1}\beta_{1,j-1}(\om)\beta_{j+1,k-1}(\om)c(\sigma^{-j}\om)j^{1+1/s+\delta}
\end{equation}
$$
 \leq C'(1+R(\om))c(\sigma^{-k}\om)\beta_{1,k-1}(\om)\sum_{j=1}^{k-1}\alpha(\sigma^{-j}\om)j^{1+1/s+\delta},
$$
where $R(\om)\in L^s$ and $\delta>0$ can be taken to be arbitrarily small.
Using \eqref{gam dec}, the above estimates and that $\alpha(\sigma^{-n}\om)\leq R(\om)n^{1/s+\delta}$, we conclude that  there is a constant $C''>0$ such that
$$
\|D^3(F_k)\|_\infty\leq C''c(\sigma^{-k}\om)\Big(E_\om^3 k^{-3a_0}+E_\om^3 R(\om) k^{-(3a_0-1-1/s-\delta)}+E_\om^3 (R(\om))^2 k^{-(3a_0-2/s-2-2\delta)}
$$
$$
+(1+R(\om))R(\om)E_\om k^{-(a_0-2-2/s-2\delta)} \Big)\leq c(\sigma^{-k}\om)V(\om)k^{-\theta},
$$
where $\theta :=a_0-2-2/s-2\delta$
and $V(\om)\in L^d$, $1/d=3/q_0+2/s$. 

Using the above estimates we see that 
\begin{equation*}
  L_1\leq \beta_{1,n}(\om)V(\om)\sum_{j=1}^n j^{-\theta}\alpha(\sigma^{-j}\om)\leq CE_\om V(\om)R(\om)n^{-(a_0+\theta-1-1/s-\delta)}.  
\end{equation*}
Here we take $\delta$ small enough to ensure that $\theta-1/s-\delta\not=-1$ so that 
$
\sum_{j=1}^n j^{-(\theta-1/s-\delta)}=O(n^{-(\theta-1/s-\delta-1)}).
$

Next, we estimate $L_2$. Using \eqref{J k 1} and \eqref{J k 2} and that $\|D(F_k)\|_\infty\leq c(\sigma^{-k}\om)\beta_{1,k-1}(\om)$, 
we see that 
$$
L_2\leq C(1+R(\om))\sum_{1\leq i,j\leq n}c(\sigma^{-j}\om)j^{1+1/s+\delta}c(\sigma^{-i}\om)\beta_{1,i-1} (\om)\prod_{s\in \mathcal I_n, s\not=i,j}\gamma_{\sigma^{-s}\om}^{-1}
$$
$$
=C(1+R(\om))\beta_{1,n}(\om)\sum_{1\leq i,j\leq n}\alpha(\sigma^{-j}\om)j^{1+1/s+\delta}\alpha(\sigma^{-i}\om)\beta_{1,i-1}(\om)
$$
$$
=C(1+R(\om))\beta_{1,n}(\om)\sum_{i=1}^n\alpha(\sigma^{-i}\om)\beta_{1,i-1}(\om)\sum_{j=1}^n \alpha(\sigma^{-j}\om)j^{1+1/s+\delta}.
$$
Using that $\alpha(\sigma^{-k}\om)\leq R(\om)k^{1/s+\delta}$ and \eqref{gam dec} we conclude that 
$$
L_2\leq C'(1+R(\om))E_\om^2 (R(\om))^2 n^{-(2a_0-\frac 3 s -3-3\delta)}.
$$

To complete the proof we need to estimate $L_3$. Using the upper bounds $\|D(F_k)\|_\infty\leq c(\sigma^{-k}\om)\beta_{1,k-1}(\om)$ and $\|F_k\|_\infty\leq \gamma_{\sigma^{-k}\om}^{-1}$ we see that 
\[
\begin{split}
L_3 &\leq \beta_{1,n}(\om)\sum_{1\leq i,j,k\leq n}\alpha(\sigma^{-i}\om)\alpha(\sigma^{-j}\om)\alpha(\sigma^{-k}\om)\beta_{1,i-1}(\om)\beta_{1,j-1}(\om)\beta_{1,k-1}(\om) \\
&=\beta_{1,n}(\om)\left(\sum_{j=1}^n \alpha(\sigma^{-j}\om)\beta_{1,j}(\om)\right)^3.
\end{split}
\]
Using that $\alpha(\sigma^{-j}\om)\leq R(\om)j^{1/s+\delta}$ and \eqref{gam dec} we conclude that 
$$
L_3\leq CE_\om^4(R(\om))^3 n^{-(4a_0-3-3/s-3\delta)}.
$$
Combining the estimates of $L_1,L_2,L_3$ the proof of \eqref{13} is complete.
\end{proof}

\begin{proof}[Proof of Lemma \ref{Lemma 5}]
Using~\eqref{D2F}, $\|F_j\|_\infty \le \gamma_{\sigma^{-j}  \om}^{-1}$ and $\|(DF_k)\|_\infty \le c(\sigma^{-k} \om)\prod_{j=1}^{k-1} \gamma_{\sigma^{-j} \om}^{-1}$,
we see that
 $$
\|D^2y\|_\infty\leq \sum_{k=1}^n\alpha(\sigma^{-k}\omega)\left(\prod_{j=1}^{k-1}\gamma_{\sigma^{-j}\omega}^{-2}\right)\left(\prod_{j=k}^{n}\gamma_{\sigma^{-j}\omega}^{-1}\right).
 $$
 By~\eqref{gam dec}, we have that 
 $$
\|D^2y\|_\infty\leq AE_\omega^2\sum_{k=1}^n k^{-2a_0}\alpha(\sigma^{-k}\omega)E_{\sigma^{-(k-1)}\omega}
 $$
 where $A$ is a constant.
Set $\beta(\omega):=\alpha(\omega)E_{\sigma\omega}$. Then $\beta\in L^{s}$, where $s>0$ is defined by $\frac1s=\frac{1}{q_0}+\frac{1}{p}+\frac{1}{q}$.
By Lemma \ref{B lemma}, for every $\delta>0$ we have $\beta(\sigma^{-k}\omega)\leq R'(\omega)k^{1/s+\delta}$ with $R'\in L^s(\Omega,\mathcal F,\mathbb P)$. Consequently,
$$
\|D^2y\|_\infty\leq E_\omega^2 R'(\omega)\sum_{k=1}^n k^{1/s+\delta-a_0}\leq CE_\omega^2 R'(\omega) n^{-(a_0-1-1/s-\delta)},
$$
where $C=C_{s,a_0,\delta}>0$ is a constant. This proves the first bound. 

To prove the second bound, we start like in the proof of the previous lemma (see~\eqref{D 3} and~\eqref{F 2}) and write 
$$
D^3y=I_1+I_2.
$$
To bound $\|I_1\|_\infty$ we write 
$$
D^2(F_k)=J_{1,k}+J_{2,k}.
$$
Using that 
\begin{equation}\label{Using}
\|F_j\|_\infty\leq \gamma_{\sigma^{-j}\om}\,\,\,\text{ and }\,\,\,\|D(F_j)\|_\infty\leq c(\sigma^{-j}\om)\prod_{k=1}^{j-1}\gamma_{\sigma^{-k}\om}^{-1},
\end{equation}
together with \eqref{gam dec}
 we see that
$$
 \|J_{1,k}\|_\infty\leq c(\sigma^{-k}\om)\prod_{j=1}^{k-1}\gamma_{\sigma^{-j}\om}^{-2}\leq c(\sigma^{-k}\om)E_\om^2 k^{-2a_0},
 $$ 
and 
\[
\begin{split}
\|J_{2,k}\|_\infty &\leq c(\sigma^{-k}\om)  \sum_{j=1}^{k-1}c(\sigma^{-j}\om) \prod_{s=1}^{j-1}\gamma_{\sigma^{-s}\om}^{-2}\prod_{v=j+1}^{k-1}\gamma_{\sigma^{-v}\om}^{-1} \\
&\leq
c(\sigma^{-k}\om)
R_0(\om) E_\om \sum_{j=1}^{k-1}j^{-2a_0}j^{1/p+\delta}E_{\sigma^{-j}\om}
\end{split}
\]
where $R_0(\om)\in L^p$ is a random variable such that $c(\sigma^{-j}\om)\leq R_0(\om)j^{1/p+\delta}$ (for arbitrarily small $\delta$, see Lemma \ref{B lemma}).
Now, using that $E_\om\in L^{q_0}$ we have $E_{\sigma^{-j}\om}\leq R_1(\om)j^{\frac1{q_0}+\delta}$, with $R_1\in L^{q_0}$. Thus, 
$$ 
\|J_{2,k}\|_\infty\leq Cc(\sigma^{-k}\om)R(\om) E_\om R_1(\om) k^{-(2a_0-1/p-1/q_0-1-2\delta)},
$$
where $C=C_{q_0,a_0,\delta}>0$ is a constant.
Setting
$$
U(\om)=2\max( R_0(\om)E_\om R_1(\om), E_\om^2)\in L^{a},\,\, \frac1{a}=\frac{1}{p}+\frac{2}{q_0}
$$
we see that 
\begin{equation}\label{D 2 F }
\|D^2(F_k)\|_\infty \leq C''c(\sigma^{-k}\om)U(\om)k^{-\theta},
\end{equation}
where $\theta=2a_0-1/p-1/q_0-1-2\delta$ and $C''>0$ is a constant. 
We conclude that  there is a constant $A''>0$ such that
$$
\|I_1\|_\infty\leq A''U(\om)\sum_{k=1}^n c(\sigma^{-k}\om)k^{-\theta}\prod_{j=1}^{k-1}\gamma_{\sigma^{-j}\om}^{-1}\prod_{j=k+1}^{n}\gamma_{\sigma^{-j}\om}^{-1}.
$$
Using \eqref{gam dec}  we see that 
\[
\begin{split}
\|I_1\|_\infty &\leq A'''U(\om)R_0(\om)E_\om\sum_{k=1}^n k^{-a_0}k^{-(\theta-1/p-\delta)}E_{\sigma^{-k}\om} \\
&\leq 
A'''Q(\om)\sum_{k=1}^nk^{-a_0}k^{-(\theta-1/p-\delta)}k^{1/q_0+\delta} \\
&\leq CQ(\om)n^{-(a_0+\theta-1/p-1/q_0-2\delta-1)},
\end{split}
\]
where $Q(\om):=U(\om)E_\om R_1(\om)R_0(\om)$, and $C, A'''>0$ are constants.

In order to bound $\|I_2\|_\infty$, using \eqref{Using} we get 
  $$
\|I_2\|_\infty\leq 2\sum_{1\leq i<j\leq n}\gamma_{\sigma^{-1}\om}^{-3}\cdots \gamma_{\sigma^{-(i-1)}\om}^{-3}c(\sigma^{-i}\om)\gamma_{\sigma^{-(i+1)}\om}^{-2}\cdots \gamma_{\sigma^{-(j-1)}\om}^{-2}c(\sigma^{-j}\om)\gamma_{\sigma^{-(j+1)}\om}^{-1}\cdots \gamma_{\sigma^{-n}\omega}^{-1}.
$$
Using also \eqref{gam dec} and that $c(\sigma^{-j}\om)\leq R_0(\om)j^{1/p+\delta}$ and $E_{\sigma^{-j}\om}\leq R_1(\om)j^{1/q_0+\delta}, R_1\in L^{q_0}$
we see that 
\[
\begin{split}
\|I_2\|_\infty &\leq cE_\om^3(R_0(\om))^2 (R_1(\om))^3\sum_{i=1}^{n-1}i^{-3a_0}i^{1/p+2/q_0+3\delta}\sum_{j=i+1}^n(j-i)^{-2a_0}j^{1/q_0+1/p+2\delta} \\
 &\le cE_\om^3(R_0(\om)^2 (R_1(\om))^3n^{1/q_0+1/p
 +2\delta}\sum_{i=1}^{n-1}i^{-3a_0}i^{1/p+2/q_0+3\delta}\sum_{j=i+1}^n(j-i)^{-2a_0} \\
 &\leq cE_\om^3(R_0(\om))^2 (R_1(\om))^3n^{-(3a_0-3/q_0-2/p-5\delta-1)},
\end{split}
\]
where in the last inequality we used that $a_0>1/2$ so that 
$$
\sum_{j=i+1}^n(j-i)^{-2a_0}= \sum_{k=1}^\infty k^{-2a_0}<\infty,
$$
and $c>0$ is a constant.

Let us prove the estimate on $\|D^4y\|_\infty$. First note that \eqref{U1}--\eqref{U4} still hold when $\gamma_\om$ is not necessarily bounded below by $1$.
To estimate the term $U_5(k)$ that appears in the upper bound of $\|D^3(F_k)\|_\infty$, arguing like in the proof of Lemma \ref{Lemma 4} but using \eqref{D 2 F } instead of \eqref{J k 1} and \eqref{J k 2} we get that 
$$
U_5(k)\leq C''c(\sigma^{-k}\om)U(\om)\sum_{j=1}^{k-1}c(\sigma^{-j}\om)\beta_{1,j-1}(\om)\beta_{j+1,k-1}(\om)j^{-\theta}
$$
$$
=C''c(\sigma^{-k}\om)U(\om)\beta_{1,k-1}(\om)\sum_{j=1}^{k-1}\alpha(\sigma^{-j}\om)j^{-\theta}.
$$
Now using that $\alpha(\om)\in L^s$ by Lemma \ref{B lemma} we have $\alpha(\sigma^{-j}\om)\leq R(\om)j^{1/s+\delta}$ for $R(\om)=R_\delta(\om)\in L^s$ and arbitrarily small $\delta>0$. Using also \eqref{gam dec} we  conclude that 
\begin{equation}\label{U5'}
U_5(k)\leq C'''c(\sigma^{-k}\om)U(\om)R(\om)E_\om k^{-(a_0+\theta-1-1/s-\delta)}.    
\end{equation}
Arguing like in the proof of Lemma \ref{Lemma 4}, using \eqref{U1}-\eqref{U4} and \eqref{U5'} instead of \eqref{U5} we conclude that
$$
\|D^3(F_k)\|_\infty\leq c(\sigma^{-k}\om)V(\om)k^{-\theta_1}
$$
with $\theta_1=3a_0-2/p-1/q_0-2/s-3\delta$ and $V(\om)\in L^{d}$ with $d$ given by $1/d=3/q_0+3/s$. Moreover,
$$
L_1\leq CE_\om V(\om)R(\om)n^{-(a_0+\theta_1-1-1/s-\delta)}.
$$

Next, arguing like in the proof of Lemma \ref{Lemma 4} but using \eqref{D 2 F } instead of \eqref{J k 1} and \eqref{J k 2} we get that 
$$
L_2\leq C''U(\om)E_\om^2 (R(\om))^2 n^{-(2a_0+\theta-2-2/s-2\delta)}.
$$
Finally, we note that 
the estimate on $L_3$ in the proof of Lemma \ref{Lemma 4} still holds as it only uses~\eqref{Using}.
\end{proof}

\begin{proof}[Proof of Lemma \ref{phi lem}]
In order to simplify the notation we omit the subscript $\varepsilon$. Let us write
\begin{equation}\label{Repres}
S_n^{\sigma^{-n}\omega}\phi \circ y=\sum_{j=0}^{n-1}\phi_{\sigma^{j-n}\omega}\circ y_j
\end{equation}
where
$y_{j}:=T_{\sigma^{-n}\omega}^j\circ y$, which is an inverse branch of $T_{\sigma^{j-n}\omega}^{n-j}$, and so 
$$
\|D(y_{j})\|_\infty\leq\rho_{\sigma^{j-n}\omega,n-j}
$$
where for all $\om$ and every $n$ we set $\rho_{\om,n}=\prod_{j=0}^{n-1}\gamma_{\sigma^{-j}\om}^{-1}$.
Thus, 
\begin{equation}\label{Sn phi 1}
\left\|D(S_n^{\sigma^{-n}\omega}\phi \circ y_{i,n})\right\|_\infty\leq \sum_{j=0}^{n-1}\|D(\phi_{\sigma^{j-n}\omega})\|_\infty\rho_{\sigma^{j-n}\omega,n-j}\leq V_1(\omega)  
\end{equation}
where
$$
V_1(\om)=\sum_{j\geq 1}B_4(\sigma^{-j}\om)\rho_{\sigma^{-j}\om,j}.
$$
Note that by invoking \eqref{gam dec} we get that 
$$
V_1(\om)\leq E_\om\sum_{j\geq 1}j^{-a_0}B_4(\sigma^{-j}\om).
$$
Thus, $\|V_1\|_{L^{v_1}}\leq \sum_{j\geq 1}j^{-a_0}\|B_4\|_{L^{d}}\|E\|_{L^{q_0}}$ with $v_1$ given by $1/v_1=1/d+1/q_0$.

Next, using again~\eqref{Repres}
we see that 
$$
D^2(S_n^{\sigma^{-n}\omega}\phi \circ y)=\sum_{j=0}^{n-1}\left(D^2(\phi_{\sigma^{j-n}\omega}\circ y_j)(D(y_{j}))^2+(D(\phi_{\sigma^{j-n}\omega})\circ y_{j})D^2(y_{j}) \right)=:I_1+I_2.
$$
Arguing like in the above we see that
$$
\|I_1 \|_\infty \leq  V_{1,2}(\omega).
$$
where 
$$
V_{1,2}(\om):=\sum_{j\geq 1}B_4(\sigma^{-j}\om)\rho_{\sigma^{-j}\om,j}^2.
$$
Using \eqref{gam dec} we see that $$
V_{1,2}(\om)\leq E_\om^2\sum_{j\geq 1}j^{-2a_0}B_4(\sigma^{-j}\om)
$$
and thus
$$
\|V_{1,2}(\cdot)\|_{L^{v}}\leq \|E_\om\|_{L^{q_0}}^2\|B_4(\cdot)\|_{L^d}\sum_{j\geq 1}j^{-2a_0}\leq C\|E_\om\|_{L^{q_0}}^2\|B_4(\cdot)\|_{L^d}<\infty,
$$
where $v$ is given by $\frac1{v}=\frac{2}{q_0}+\frac{1}d$.
In order to bound $I_2$, using either Lemma~\ref{Lemma 4} or Lemma \ref{Lemma 5} we get that
$$
\|D^2(y_{j})\|_\infty\leq (n-j)^{-(a_0-1/s-1-\delta)}C_{\sigma^{-(n-j)}\omega}. 
$$
Thus, with $\eta=a_0-1/s-1-\delta$, 
$$
\|I_2 \|_\infty \leq \sum_{j=0}^{n-1}\|D(\phi_{\sigma^{j-n}\omega})\|_\infty C_{\sigma^{-(n-j)}\omega}(n-j)^{-\eta}=\sum_{k=1}^n\|D(\phi_{\sigma^{-k}\omega})\|_\infty C_{\sigma^{-k}\omega}k^{-\eta}\leq V_{2,2}(\om),
$$
where 
$$
V_{2,2}(\om):=\sum_{k=1}^{\infty}B_4(\sigma^{-k}\om)C_{\sigma^{-k}\om}k^{-\eta}.
$$
Notice that  since $\eta>1$ we have $\om\mapsto B_4(\om)C_\om\in L^{a}(\Omega,\mathcal F,\mathbb P)$ where $a$ is given by $\frac1{a}=\frac{1}{d}+\frac{1}{t}$.
Thus, we can take
$$
V_2(\om)=V_1(\om)+V_{1,2}(\om)+V_{2,2}(\om)
$$
which belongs to $L^{v_2}(\Omega,\mathcal F,\mathbb P)$, where $v_2=\min(a,v)$. Note that $1/v_2=\max(1/a,1/v)=\frac1d+\max(1/t,2/q_0)=\frac1d+2/\min(q_0,2t)$, as stated in Lemma \ref{phi lem}.
The reminding  estimates are similar. We first use \eqref{Repres} and then we use the formula for the third and fourth derivatives of compositions of two functions and the bounds in Lemma \ref{Lemma 4} and Lemma \ref{Lemma 5} on the derivatives of the function $y_j$.


A tedious computation shows that with 
 with $\rho_{\om,n}=\prod_{j=0}^{n-1}\gamma_{\sigma^{-j}\om}^{-1}$,
 we can take
$$
V_3(\om)=V_2(\om)
$$
$$
+\sum_{k=1}^\infty B_4(\sigma^{-k}\om)\rho_{\sigma^{-k}\om,k}^3+
2\sum_{k=1}^\infty B_4(\sigma^{-k}\om)\rho_{\sigma^{-k}\om,k}C_{\sigma^{-k}\om}k^{-\eta}+
\sum_{k=1}^\infty B_4(\sigma^{-k}\om)A_{\sigma^{-k}\om}k^{-\zeta}
$$
and, with some constant $c_4>0$,
$$
V_4(\om)=V_3(\om)+c_4\sum_{j=1}^\infty B_4(\sigma^{-j}\om)\rho_{\sigma^{-j}\om,j}^4+c_4\sum_{j=1}^\infty B_4(\sigma^{-j}\om)C_{\sigma^{-j}\om}\rho_{\sigma^{-j}\om,j}^2j^{-\eta}
$$
$$
+c_4\sum_{j=1}^\infty B_4(\sigma^{-j}\om)A_{\sigma^{-j}\om}\rho_{\sigma^{-j}\om,j}j^{-\zeta}+c_4\sum_{j=1}^\infty B_4(\sigma^{-j}\om)C_{\sigma^{-j}\om}^2j^{-2\eta}+c_4\sum_{j=1}^\infty B_4(\sigma^{-j}\om)R_{\sigma^{-j}\om}j^{-\kappa}.
$$
Using \eqref{gam dec} to replace $\rho_{\sigma^{-j}\om,j}$ by $E_\om j^{-a_0}$
and then summing up the resulting norms and using that $\zeta,\eta,\kappa,a_0>1$ we obtain that $V_i(\om)\in L^{v_i}$ with $v_i$ as in the statement of the lemma.
\end{proof}

\section{Effective spectral gap for non-normalized transfer operators}\label{App}
In this section we prove \eqref{spectral gap} for the operators $\mathcal L_{\omega,\varepsilon}$ under appropriate assumptions. In \cite{YH 23} this was done for the operators $L_{\omega,\varepsilon}$ given by $L_{\omega,\varepsilon}(g)=\mathcal L_{\omega,\varepsilon}(gh_{\omega,\varepsilon})/h_{\sigma \omega,\varepsilon}$. Passing to the normalized operators $L_{\omega,\varepsilon}$ was required in order to control the statistical properties of appropriate random Birkhoff sums, and it required several a priory estimates on $h_{\omega,\varepsilon}$ which are not needed when dealing with $\mathcal L_{\omega,\varepsilon}$. On the other hand, $L_{\omega,\varepsilon}$ is Markov operator (i.e. $L_{\omega,\varepsilon}\textbf{1}=\textbf{1}$) which was important for the proof of the main results in \cite{YH 23}.

\subsection{The random dynamical environment}\label{Rand env}
Let $(X_j)_{j\in\mathbb Z}$ be a stationary ergodic sequence of random variables defined on a common probability space $(\Omega_0,\mathcal F_0,\mathbb P_0)$. For every $k,k_1,k_2\in\mathbb Z$ such that $k_1\leq k_2$ we define 
$$
\mathcal F_{-\infty,k}=\mathcal F\{X_j: j\leq k\},\,
\mathcal F_{k_1,k_2}=\mathcal F\{X_j:  k_1\leq j\leq k_2\}\,\text{ and }\,\, 
\mathcal F_{k,\infty}=\mathcal F\{X_j: j\geq k\}.
$$
Here $\mathcal F\{X_j: j\in A\}$ denotes the $\sigma$-algebra generated by the family of random variables $\{X_j: j\in A\}$, and $A\subset\mathbb Z$ is a set.
Recall that the upper   $\psi$-mixing coefficients  of the process $(X_j)_{j\in\mathbb Z}$ are given by 
$$
\psi_U(n)=\sup_{k\in\mathbb Z}\sup\left\{\frac{\mathbb P_0(A\cap B)}
{\mathbb P_0(A)\mathbb P_0(B)}-1:A\in\mathcal F_{-\infty,k}, 
B\in\mathcal F_{k+n,\infty}, \ \mathbb P_0(A)\mathbb P_0(B)>0\right\}.
$$
Next, recall that the two-sided  $\alpha$-mixing coefficients  of  $(X_j)_{j\in\mathbb Z}$ are given by 
\begin{equation}\label{alpha def}
\alpha(n)=\sup_{k\in\mathbb Z}\sup\left\{|\mathbb P_0(A\cap B)-\mathbb P_0(A)\mathbb P_0(B)|:A\in\mathcal F_{-\infty,k}, \ B\in\mathcal F_{k+n,\infty}\right\}.
\end{equation}

Our dynamical environment  $(\Omega,\mathcal F,\mathbb P,\sigma)$ is the left shift system formed by   $(X_j)_{j\in\mathbb Z}$. Namely, $\Omega=\Omega_0^\mathbb Z$, $\mathcal F$ is the appropriate product $\sigma$-algebra, $\mathbb P$ is the unique measure such that for every finite collection of sets $A_i\in \mathcal F_0, |i|\leq m$, the  corresponding cylinder set
$A=\{(\om_k)_{k=-\infty}^\infty: \om_i\in A_i, |i|\leq m\}$  satisfies $\mathbb P(A)=\mathbb P_0(X_i\in A_i; |i|\leq m)$. Moreover, for $\om=(\om_k)_{k\in\mathbb Z}$ we have  $\sigma(\om)=(\om_{k+1})_{k\in\mathbb Z}$ (henceforth we will drop the brackets and write $\sigma(\om)=\sigma\om$). This means that, when considered as a random point,  $(\om_j)_{j\in\mathbb Z}$ has the same distribution as the random path $(X_j)_{j\in\mathbb Z}$. Henceforth we will abuse the notation and identify $\mathcal F_{k,\ell}$ and the sub-$\sigma$-algebra of $\mathcal F$ generated by the projections on the coordinates $\om_j, k\leq j\leq \ell$.

\subsection{Mixing moment and approximation conditions}\label{MixMom}
Let $D_\om, B_\om$ and $N(\om)$ be as in~\eqref{Domega}, \eqref{B need} and~\eqref{N need}, respectively.
We begin with the following class of moment conditions.
\begin{assumption}\label{Mom Ass}
For some $\tilde b>2, p,q,q_0>1$ and $b_2>1$ such that $q_0<q$ and $qq_0>\tilde b$ we have 
$$
\ln D_\om\in L^{qq_0}(\Omega,\mathcal F,\mathbb P), \  B_\om\in L^{\min(\tilde b,p)}(\Omega,\mathcal F,\mathbb P), \ N(\om)\in L^{b_2}(\Omega,\mathcal F,\mathbb P)
$$
where for a random variable $Y_\om$ we write $Y_\om\in L^p$ if $\om\mapsto Y_\om\in L^p(\Omega,\mathcal F,\mathbb P)$. 
\end{assumption}
Next, for every $1\leq p\leq \infty$ we consider the following approximation coefficients
$$
\beta_{p}(r)=\|\gamma_{\om}^{-1}-\mathbb E[\gamma_{\om}^{-1}|\mathcal F_{-r,r}]\|_{L^p(\Omega,\mathcal F,\mathbb P)}, \,\,\,\,\,b_p(r)=\|B_\om-\mathbb E[B_\om|\mathcal F_{-r,r}]\|_{L^p(\Omega,\mathcal F,\mathbb P)},
$$
$$
d_p(r)=\|\ln D_\om-\mathbb E[\ln D_\om|\mathcal F_{-r,r}]\|_{L^p(\Omega,\mathcal F,\mathbb P)},\,\,\,\,\,
n_p(r)=\|N(\om)-\mathbb E[N(\om)|\mathcal F_{-r,r}]\|_{L^p(\Omega,\mathcal F,\mathbb P)}.
$$

Next for all $u,\theta,M,b_0,r_1>0$ we  consider the following assumption.
\begin{assumption}\label{Ass A u}
One of the following conditions holds:
\vskip0.1cm
(i) $(\gamma_{\sigma^j\om})_{j\geq 0}$ is an i.i.d sequence and $\mathbb E[\gamma_\om^{-u}]<1$;
\vskip0.2cm
or
\vskip0.2cm

(ii) $\gamma_\om\geq 1$, $\mathbb P(\gamma_\om=1)<1$ and 
\begin{equation}\label{psi limsup}
\limsup_{k\to\infty}\psi_U(k)<\min\left(\frac{1}{\mathbb E[\gamma_\om^{- u}]}, \frac 1{\theta}\right)-1;    
\end{equation}
or 

(iii) $\gamma_\om\geq 1$, $\mathbb P(\gamma_\om=1)<1$ and either
\begin{equation}\label{al 1}
\alpha(n)=O(n^{-(M-1)})    
\end{equation}
or 
\begin{equation}\label{al 2}
\alpha(n)=O(e^{-b_0n^{\eta_1}}).    
\end{equation}

\end{assumption}

We also consider the following assumptions.
\begin{assumption}\label{Inner aprpox}
 For all $M_0\in\N$  and every $r\in\N$  there are sets $A_r=A_{r,M_0}$ measurable with respect to  $\mathcal F_{-r,r}$ and $B_r=B_{r,M_0}\in\mathcal F$ such that, with $L_{M_0}=\{\omega: m(\omega)\leq M_0\}$ we have
$$
A_{r}\subset L_{M_0}\cup B_{r}, \,\,\lim_{r\to\infty}\mathbb P(A_r)=\mathbb P(L_{M_0}).
$$
Moreover, either $\mathbb P(B_{r,M_0})=O(r^{-M})$ for some $M>0$ or $\mathbb P(B_{r})=O(e^{-br^a})$ for some $b,a>0$.
\end{assumption}

\begin{assumption}\label{Poly Ass}
With $\tilde b>2, p,q,q_0>1$ and $b_2>1$ as in Assumption \ref{Mom Ass} and $u$ as in Assumption \ref{Ass A u}, for some $p_0, u,\tilde u, \tilde p, p_0, b,v, u_0,v_0>1$  such that 
$$
\frac1{p_0}+\frac{1}{q_0}=\frac{1}q,\,\,\,\, \frac{1}{qq_0}=\frac{1}{\tilde b}+\frac{1}v,\,\,\,\, \frac{1}b=\frac{1}{p}+\frac{1}{u},\,\,\,\, 
\frac{1}{\tilde b}=\frac{1}{\tilde p}+\frac{1}{\tilde u},\,\,\,\frac{1}{\tilde u}\geq \frac1{u_0}+\frac1{v_0}
$$
we have the following:

\vskip0.1cm
(i) $\mathbb P(m(\omega)\geq n)=O(n^{-\beta d-1-\varepsilon_0})$ for some $\beta,\varepsilon_0>0$ and $d\geq q$ such that
$
\beta d+\varepsilon_0>\max\left(\frac{p_0}{q_0}+p_0-1, v\right).
$
\vskip0.1cm
(ii)
either $\lim_{r\to\infty}\beta_\infty(r)=0$ or $\beta_{\tilde u}(r)=O(r^{-A})$ for some $A>0$ such that $A>2\tilde u+1$.
\vskip0.1cm
(iii) 
$
d_1(r)+b_{p}(r)+n_{b_2}(r)+\min(\beta_\infty(r),c_r\beta_{u_0}(r))=O(r^{-M})
$
for some $M>0$, where $c_r=r^{2-\frac{1-\varepsilon A}{v_0}}$
 and 
$\varepsilon>0$ satisfies that $\varepsilon A>2\tilde u+1$.     
\end{assumption}

\begin{theorem}\label{OSC1}

Let  Assumptions \ref{Mom Ass}, \ref{Ass A u}, \ref{Inner aprpox} and \ref{Poly Ass} be in force, where in Assumption \ref{Inner aprpox} we suppose that $\mathbb P(B_r)=O(r^{-M})$. 
When  $\beta_\infty(r)\to 0$ we set $a_0=M$, while when  $\beta_{\tilde u}(r)=O(r^{-A})$ we set 
$$
a_0=\min\left(M,1+\frac{1-\varepsilon A}{\tilde u}\right),
$$
where $M,\varepsilon,\tilde u,A$  come from the above assumptions. 
Suppose that $A$  and $M$ are large enough so that $a_0>\beta d+3$ and that $\beta>1$.
Then, there  are constants $\theta,M_0>0$ which can be recovered from the proof such that 
if either  part (i)  from Assumption  \ref{Ass A u}  holds, part (ii) of Assumption  \ref{Ass A u}  holds with  $u$ and $\theta$ or  \eqref{al 1} holds with the above $M$, then for $\mathbb P$-a.e. $\om$ and all $\varepsilon\in I$ there are unique  equivariant densities $h_{\om, \varepsilon}$ and
there is a random variable $R(\om)\in L^t(\Omega,\mathcal F,\mathbb P)$, where $t$ is defined by $1/t=1/q+1/d$  
such that for  all  $n\geq M_0m(\omega)$  and every $C^1$ function $g: M\to\mathbb R$,
$$
\left\|\mathcal L_{\om, \varepsilon}^n g-m(g)h_{\sigma^n\omega, \varepsilon}\right\|_{\infty}\leq \|g\|_{C^1} R(\omega)n^{-\beta}.
$$
\end{theorem}

\begin{remark}\label{rem small n}
Note that when $\xi_\omega=1$ then $m(\om)=0$ and so we get the estimates for all $n$. Note also that a slight modification of the arguments in \cite{YH 23} shows that  we can choose $M_0$ to be the smallest number such that $\mathbb P(m(\om)=M_0)>0$, namely $M_0=\essinf m(\cdot)$. Thus, we get the result for every $n\geq m(\om)$ if $\mathbb P(m(\om)=1)>0$.

To get estimates in Theorem \ref{OSC1} when $n<M_0m(\omega)$ we consider the case when there is a random variable $C_0(\omega)\in L^{r}$ such that for $\mathbb P$-a.e. $\om$ for all $n$ and $\varepsilon\in I$,
$$
\|\mathcal L_{\sigma^{-n}\om,\varepsilon}^n\textbf{1}\|_\infty\leq C_0(\omega).
$$
Now, since $h_{\om,\varepsilon}$ is the uniform limit of $\mathcal L_{\sigma^{-n}\om,\varepsilon}^n\textbf{1}$ we see that 
$$
\|h_{\om,\varepsilon}\|_\infty\leq C_0(\omega).
$$
Thus for $n<M_0m(\omega)$ and a function $g$ such that $\|g\|_\infty\leq 1$ we have 
$$
\|\mathcal L_{\om,\varepsilon}^n g-m(g)h_{\sigma^n\omega,\varepsilon}\|_\infty\leq 
\|\mathcal L_{\om,\varepsilon}^n\textbf{1}\|_{\infty}+C_0(\sigma^n\omega)\leq 2C_0(\sigma^n\omega).
$$
Now, since $C_0\in L^{r}$, by Lemma \ref{B lemma} for every $\delta>0$ we have $C_0(\sigma^n\omega)\leq R_0(\omega)n^{1/r+\delta}$, $R_0\in L^r$ and so 
\[
\begin{split}
\|\mathcal L_{\om,\varepsilon}^n g-m(g)h_{\sigma^n\omega,\varepsilon}\|_\infty &\leq 
2R_0(\omega)n^{1/r+\delta}\\
&=2R_0(\omega)n^{1/r+\delta+\beta}n^{-\beta} \\
&\leq  2\left(R_0(\omega)(M_0m(\omega))^{1/r+\delta+\beta}\right)n^{-\beta}.
\end{split}
\]
Thus for $n<M_0m(\omega)$ we can take 
$$
R(\om)=2R_0(\omega)(M_0m(\omega))^{1/r+\delta+\beta}.
$$
Notice that since $\mathbb P(m(\om)\geq n)=O(n^{-\beta d-1-\varepsilon_0})$ we have $m(\cdot)\in L^l$ for every $l<\beta d+1+\varepsilon_0$. Thus $\om\mapsto R(\om)\in L^{r_2}$ where $r_2$ is given by $\frac1{r_2}=\frac{1}r+\frac{1}{r_0}$ and
\begin{equation}\label{r 0 form}
r_0=\frac{l}{1/r+\beta+\delta}.    
\end{equation} 
Hence \eqref{spectral gap} holds with $p_1=\min(t, r_2)$.
\end{remark}

\subsection{Proof of Theorem \ref{OSC1}}
Set
$$
Q_\omega=\sum_{k\geq 1}B_{\sigma^{-k}\omega}\prod_{j=1}^k\gamma_{\sigma^{-j}\omega}^{-1}.
$$
Then by \cite[Lemma 5.13]{YH 23} we get that $\omega\mapsto Q_\om\in L^b(\Omega,\mathcal F,\mathbb P)$, where $b$ is as in Assumption \ref{Poly Ass}.  
Fix some $s>2$ and 
let the cone $\mathcal C_\omega$ be given by
$$
\mathcal C_\om=\{g:M\to (0,\infty) g(x)\leq g(y)e^{sQ_\omega d(x,y)}\text{ if }d(x,y)\leq \xi_\omega\}.
$$
Then by \cite[Lemma 5.7.3]{HK}, for every $\varepsilon\in I$
$$
\mathcal L_{\omega,\varepsilon}\mathcal C_\omega\subset \mathcal C_{\sigma\omega}.
$$
Moreover, by \cite[Lemma 5.7.3]{HK} and \cite[Eq. (5.7.18)]{HK}, for all $n\geq m(\omega)$ we have that
\begin{equation}\label{Diam est 1}
\Delta_{n}(\omega,\varepsilon):=
\sup_{f,g\in\mathcal C_\omega}d_{\mathcal C_{\sigma^n\omega}}(\mathcal L_{\omega,\varepsilon}^n f, \mathcal L_{\omega,\varepsilon}^n g)\leq 
 d_n(\omega),
\end{equation}
where $d_{\mathcal C}$ is the Hilbert projective metric associated with a cone $\mathcal C$ and 
$$
d_n(\omega)=4\sum_{j=0}^{n-1}B_{\sigma^j\om}+2\sum_{j=0}^{n-1}\ln (D_{\sigma^j\omega})+2\ln(s''_{\sigma^{n}\omega})+2sQ_\omega,
$$ 
with
$$
 s''_\omega:=\frac{2s}{s-1}\cdot \frac{Q_{\sigma^{-1}\omega}}{2B_{\sigma^{-1}\omega}}+1+\frac{s+1}{s-1}.
$$
Repeating the proofs of \cite[Eq. (5.24)]{YH 23}, \cite[Proposition 5.19]{YH 23} and \cite[Corollary 5.20]{YH 23} we get that  
there exists a constant $M_0>0$ such that for every $n\geq M_0m(\omega)$ and all $f,g\in\mathcal C_\omega$
we have 
\begin{equation}\label{Random Contraction}
d_{\mathcal C^+}(\mathcal L_{\om,\varepsilon}^n f,\mathcal L_{\om,\varepsilon}^n g)\leq d_{\mathcal C_\om}(\mathcal L_{\om,\varepsilon}^n f,\mathcal L_{\om,\varepsilon}^n g)\leq U(\om) K(\omega)n^{-\beta} ,
\end{equation}
where $U(\om)=d_{m(\om)}(\om)$, $\omega\mapsto K(\omega)\in L^d(\Omega,\mathcal F,\mathbb P)$ and $\mathcal C^+$ is the cone of positive functions. 
Moreover, by arguing as in the proof of \cite[Lemma 5.16]{YH 23} we get that $\omega\mapsto U(\omega)\in L^q(\Omega,\mathcal F,\mathbb P)$.
Taking $g=h_{\omega,\varepsilon}$ we see that 
$$
d_{\mathcal C^+}(\mathcal L_{\om,\varepsilon}^n f,\mathcal L_{\om,\varepsilon}^n g)=d_{\mathcal C_\om^+}(\mathcal L_{\om, \varepsilon}^n f,h_{\sigma^n\omega, \varepsilon}).
$$
 By applying \cite[Lemma 3.5] {Kifer Thermo} with the measure $m$ and the functions $F=\mathcal L_{\om,\varepsilon}^n f/m(f)$ and $G=h_{\sigma^n\omega,\varepsilon}$, we get that for $\mathbb P$ a.e. $\om \in \Omega$, every $\varepsilon\in I$ and all $n\geq M_0m(\om)$ we have
$$
\|\mathcal L_{\om,\varepsilon}^n f-m(f)h_{\sigma^n\omega,\varepsilon}\|_{\infty}\leq m(f) U(\om)K(\om)n^{-\beta}.
$$
Now the estimate in Theorem \ref{OSC1}   follows from \cite[Lemma 5.4]{YH 23} which allows us to upgrade the estimates from functions $f$ in the cone $\mathcal C_\om$ to general $C^1$ functions, up to multiplying the above right hand side by $12\xi_\om^{-1}(1+4/Q_\om)$. 

\subsection*{Acknowledgments}
 This paper has been funded by European Union – NextGenerationEU-Statistical properties of random dynamical systems and other contributions to mathematical analysis and probability theory (Davor Dragi\v cevi\' c). D.D is deeply grateful to Julien Sedro for numerous discussions on linear response for random dynamical systems along the years.

\subsection*{Data availability statement}
The authors did not use any data in the research.

\subsection*{Conflict of interests statement}
The authors do not have any conflict of interests.

\end{document}